\documentclass[english]{article}

\usepackage{babel}           
\usepackage[utf8]{inputenc}  
\usepackage[T1]{fontenc}     
\usepackage{amsmath, amsthm} 
\usepackage{amsfonts,amssymb}
\usepackage{mathtools}
\usepackage{bm}
\usepackage{hyperref}
\usepackage{verbatim}    
\usepackage{mathtools}
\usepackage[shortlabels]{enumitem}   
\usepackage{tikz}
\usetikzlibrary{positioning}
\usepackage{adjustbox}
\usepackage{array}
\usepackage{float}
\usepackage{mathrsfs}
\usepackage[title]{appendix}

\usepackage{graphicx}
\usepackage{subcaption}
\usepackage{cleveref}	      

\newtheorem{Theorem}{Theorem}[section]
\newtheorem{Definition}[Theorem]{Definition}
\newtheorem{Proposition}[Theorem]{Proposition}
\newtheorem*{Proposition*}{Proposition}
\newtheorem{Corollary}[Theorem]{Corollary}
\newtheorem{Lemma}[Theorem]{Lemma}
\newtheorem{Remark}[Theorem]{Remark}

\Crefname{Lemma}{Lemma}{Lemmas}
\Crefname{Theorem}{Theorem}{Theorems}

\newcommand{\cC}{\mathcal{C}}
\newcommand{\cO}{\mathcal{O}}

\newcommand{\cF}{\mathcal{F}}

\newcommand{\cS}{\mathcal{S}}

\newcommand{\bk}{\bm{k}}
\newcommand{\bx}{\bm{x}}
\newcommand{\by}{\bm{y}}
\newcommand{\bu}{\bm{u}}

\newcommand{\hb}{h_{\rm b}}
\newcommand{\pb}{p_{\rm b}}

\def\Ma{{\rm Ma}}
\def\Ro{{\rm Ro}}
\def\Fr{{\rm Fr}}
\def\Rey{{\rm Re}}

\renewcommand{\L}{\mathscr{L}}
\renewcommand{\P}{\mathscr{P}}
\newcommand{\T}{\mathscr{T}}

\newcommand{\sPi}{{\sf \Pi}}

\renewcommand{\S}{{\sf S}}
\newcommand{\A}{{\sf A}}
\newcommand{\G}{{\sf G}}
\newcommand{\sL}{{\sf L}}
\newcommand{\M}{{\sf M}}
\newcommand{\sT}{{\sf T}}

\newcommand{\RR}{\mathbb{R}}

\newcommand{\TT}{\mathbb{T}}
\newcommand{\NN}{\mathbb{N}}
\renewcommand{\i}{{\rm i}}

\DeclareMathOperator{\Id}{Id}
\DeclareMathOperator{\dd}{{\rm d}\!}
\DeclareMathOperator{\Ran}{Ran}
\DeclareMathOperator{\Ker}{Ker}
\DeclareMathOperator{\Spec}{Spec}
\DeclareMathOperator{\esssup}{ess\,sup}

\newcommand{\eps}{\eps}

\let\eps=\varepsilon
\newcommand{\eqdef}{\coloneqq}

\DeclarePairedDelimiter\norm{\big\lvert}{\big\rvert}
\DeclarePairedDelimiter\Norm{\big\lVert}{\big\rVert}
\DeclarePairedDelimiter\bra{\big\langle}{\big\rangle}

\title{Three-scale singular limits with applications to rapidly rotating fluids and the hyperbolization of dispersive systems}
\author{Vincent Duchêne\thanks{Univ Rennes, CNRS, IRMAR - UMR 6625, F-35000 Rennes, France. \texttt{vincent.duchene@univ-rennes.fr}}
\and Arnaud Duran\thanks{Université Lyon 1, Centrale Lyon, INSA Lyon, Université Jean Monnet, CNRS, ICJ UMR5208, 69622 Villeurbanne, France. \texttt{arnaud.duran@univ-lyon1.fr}} \thanks{Institut Universitaire de France (IUF).}
\and Khawla Msheik\thanks{CentraleSupelec, CNRS - UPR 288, Laboratoire Énergetique Moleculaire et Macroscopique, Combustion (EM2C), 91190 Gif-sur-Yvette, France.}
}

\date{\today}
\begin{document}
\thispagestyle{empty}
\maketitle

\begin{abstract}
	We consider singular problems for a general class of quasilinear hyperbolic systems that involve two {\em a priori} independent stiff parameters. We argue that such situations may lead to the rapid development of small-amplitude spatial oscillations of small wavelength starting from arbitrarily smooth initial data. Despite this phenomenon we provide sufficient conditions on initial data that secure the uniform control of solutions and show strong convergence in the singular limit of small stiff parameters. We apply our general theory to the rapidly rotating shallow-water system with bottom topography, and to hyperbolic systems stemming from a constraint-relaxation strategy applied to dispersive models for the propagation of water waves ---specifically the Benjamin--Bona--Mahony, Boussinesq--Peregrine and Serre--Green--Naghdi equations.
\end{abstract}


\section{Introduction}

\subsection{Presentation}\label{S.presentation}
In this work we are interested in the behavior of solutions to general three-scale systems of hyperbolic quasilinear equations, of the form
\[
	\S_0(U)\partial_t U + \sum_{l=1}^d\S_l(U)\partial_{x_l}U +\G(U)U = \frac1\eps \L_\delta U .
\]
Here, $U:(t,\bx)\in \RR\times\RR^d\to\RR^n$ (with $d,n\in\NN^\star$) is the unknown and $\S_l$ ($l\in\{0,\dots,d\})$ (respectively $\G$) are given smooth functions into $n\times n$ real-valued symmetric matrices (respectively real-valued matrices), while $\L_\delta$ is a constant-coefficient differential or pseudo-differential operator, skew-adjoint for the $L^2(\RR^d)^n$ inner-product. Precise assumptions of our framework are listed in  \Cref{S.main-results} below. A prototypical example for the considered stiff operators is 
\[ \L_\delta = \sL_0+\delta\sum_{l=1}^d\sL_l \partial_{x_l}\]
where $\sL_0$ is a skew-symmetric matrix and for all $l\in\{1,\dots,d\}$, $\sL_l$ is a symmetric matrix. Importantly, $\delta$ and $\eps$ are positive parameters that are in principle independent, and we are specifically interested in the asymptotic regime 
\[0<\eps\ll\delta\ll 1.\]

Systems of such type arise naturally for instance in geophysical models in the presence of rapid rotation and rapidly propagating waves (we specifically discuss the rotating shallow-water system in this work). In such situation the dimensionless parameter $\eps$ is the so-called Rossby number defined as the ratio of inertial force to Coriolis force, while $\eps/\delta$ is the square of the so-called Mach number defined as the ratio of the flow velocity to the wave celerity. Hence we are interested in a regime of low Rossby ($\eps\ll 1$) and low Mach ($\eps/\delta\ll 1$) numbers, where the effects of fast rotation are predominant ($\delta\ll1$). We are also motivated by applications to the ``hyperbolization'' of weakly dispersive equations, such as the Benjamin--Bona--Mahony equation, Boussinesq--Peregrine or the Serre--Green--Naghdi systems. In such situation the dimensionless parameter $\delta$ is (the square root of) the so-called shallowness parameter measuring the strength of dispersive effects while $\eps^{-1}$ is an artificial relaxation parameter which formally determines the accuracy of solutions to the system of quasilinear equations as approximations to the corresponding  dispersive equations.

Of paramount importance to our framework with respect to the existing literature is the presence of the symmetrizer $\S_0(U)$ depending non-trivially on the unknown. We argue that, in conjunction with the presence of the two independent parameters in our asymptotic regime, non-trivial symmetrizers allow for a mechanism of {\em development of a new spatial scale} which, to the best of our knowledge, has not yet been unveiled in the literature. Specifically, we shall argue that solutions to systems belonging to the class studied in this work emerging from smooth initial data 
 may gradually develop spatial oscillations with wavelength of size $\delta$, on a timescale of size $\eps/\delta$.

Of course such rapid development small spatial scales is a strong obstacle to our aspiration for uniform (with respect to $0<\eps\ll\delta\ll 1$) strong control of solutions and eventually strong convergence as $\eps\searrow0$ and, possibly, $\delta\searrow 0$. In order to obtain such results, we will ensure that the rapid spatial oscillations are of small magnitude, securing in particular a uniform bound on the Lipschitz norm of solutions at each time. This will be enforced through an assumption of well-prepared initial data. However our aim is to provide the least possible restrictions on the initial data that secure the propagation in time on a relevant timescale of strong controls of solutions, uniformly with respect to $0<\eps\ll\delta\ll 1$.

\subsection{Outline}\label{S.outline}

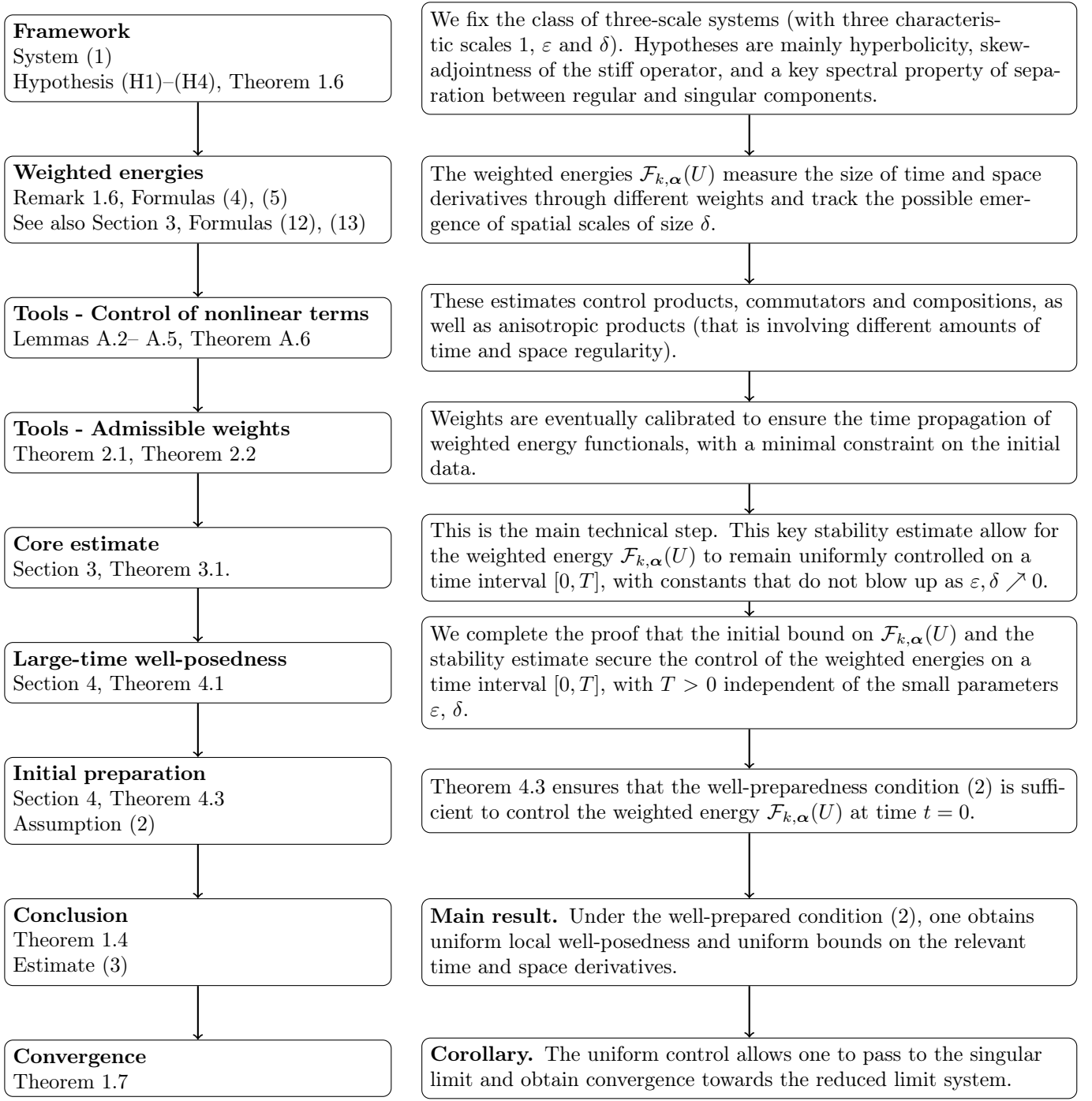
\begin{figure}
	\begin{center}
		\makebox[\textwidth][c]{
			\begin{adjustbox}{max width=1.4\textwidth, max height=0.90\textheight}
				\begin{tikzpicture}[
					refbox/.style={draw,rounded corners,align=left,text width=6.2cm,inner sep=4pt},
					explbox/.style={draw,rounded corners,align=left,text width=10.7cm,inner sep=4pt},
					arrow/.style={->, thick},
					node distance=9mm and 5mm
					]
					
					\node[refbox] (Aref) {
						\textbf{Framework}\\
						System~\eqref{eq.general} \\
						Hypothesis \ref{H1}--\ref{H4}, \Cref{R.weights}
					};
					
					\node[explbox, right=of Aref] (Aexp) {
						We fix the class of three-scale systems (with three characteristic scales $1$, $\varepsilon$ and $\delta$). Hypotheses are mainly hyperbolicity, skew-adjointness of the stiff operator, and a key spectral property of separation between regular and singular components.
					};
					
					\node[refbox, below=of Aref] (Bref) {
						\textbf{Weighted energies}\\
						Remark \ref{R.weights}, Formulas \eqref{eq.def-F}, \eqref{eq.def-alpha}\\
						See also \Cref{S.Stability}, Formulas \eqref{eq.def-F-stability},  \eqref{eq.def-F0-stability}
					};
					
					\node[explbox, right=of Bref] (Bexp) {
						The weighted energies $\cF_{k,\bm\alpha}(U)$ measure the size of time and space derivatives through different weights and track the possible emergence of spatial scales of size \(\delta\).
					};
					
					\node[refbox, below=of Bref] (Cref) {
						\textbf{Tools - Control of nonlinear terms}\\
						Lemmas \ref{L.product}-- \ref{L.composition}, 
						\Cref{L.Product}
					};
					
					\node[explbox, right=of Cref] (Cexp) {
						These estimates control products, commutators and compositions, as well as anisotropic products (that is involving different amounts of time and space regularity).
					};
					
					\node[refbox, below=of Cref] (C2ref) {
						\textbf{Tools - Admissible weights}\\
						\Cref{D.alpha-admissible}, \Cref{L.weights}
					};
					
					\node[explbox, right=of C2ref] (C2exp) {
						Weights are eventually calibrated to ensure the time propagation of weighted energy functionals, with a minimal constraint on the initial data.
					};
					
					\node[refbox, below=of C2ref] (Dref) {
						\textbf{Core estimate}\\
						\Cref{S.Stability}, \Cref{P.Stability}.
					};
					
					\node[explbox, right=of Dref] (Dexp) {
						This is the main technical step. This key stability estimate allow for the weighted energy $\cF_{k,\bm\alpha}(U)$ to remain uniformly controlled on a time interval $[0,T]$, with constants that do not blow up as \(\varepsilon,\delta \nearrow 0\).
					};

					\node[refbox, below=of Dref] (Eref) {
						\textbf{Large-time well-posedness}\\
						\Cref{S.completion}, \Cref{P.Well-posedness}
					};
					
					\node[explbox, right=of Eref] (Eexp) {
						We complete the proof that the initial bound on $\cF_{k,\bm\alpha}(U)$ and the stability estimate secure the control of the weighted energies on a time interval $[0,T]$, with $T>0$ independent of the small parameters $\varepsilon$, $\delta$.
					};

					\node[refbox, below=of Eref] (Fref) {
						\textbf{Initial preparation}\\
						\Cref{S.completion}, \Cref{P.preparation}\\
						Assumption \eqref{eq.well-prepared}
					};
					
					\node[explbox, right=of Fref] (Fexp) {
						\Cref{P.preparation} ensures that the well-preparedness condition \eqref{eq.well-prepared} is sufficient to control the weighted energy $\cF_{k,\bm\alpha}(U)$ at time \(t=0\).
					};
					
					\node[refbox, below=of Fref] (Gref) {
						\textbf{Conclusion}\\
						\Cref{T.Well-posedness}\\
						Estimate \eqref{eq.estimate}
					};
					
					\node[explbox, right=of Gref] (Gexp) {
						\textbf{Main result.}
						Under the well-prepared condition \eqref{eq.well-prepared}, one obtains uniform local well-posedness and uniform bounds on the relevant time and space derivatives.
					};
					
					\node[refbox, below=of Gref] (Href) {
						\textbf{Convergence}\\
						\Cref{C.convergence}
					};
					
					\node[explbox, right=of Href] (Hexp) {
						\textbf{Corollary.}
						The uniform control allows one to pass to the singular limit and obtain convergence towards the reduced limit system.
					};
					
					\draw[arrow] (Aref.south) -- (Bref.north);
					\draw[arrow] (Bref.south) -- (Cref.north);
					\draw[arrow] (Cref.south) -- (C2ref.north);
					\draw[arrow] (C2ref.south) -- (Dref.north);
					\draw[arrow] (Dref.south) -- (Eref.north);
					\draw[arrow] (Eref.south) -- (Fref.north);
					\draw[arrow] (Fref.south) -- (Gref.north);
					\draw[arrow] (Gref.south) -- (Href.north);
					
					\draw[arrow] (Aexp.south) -- (Bexp.north);
					\draw[arrow] (Bexp.south) -- (Cexp.north);
					\draw[arrow] (Cexp.south) -- (C2exp.north);					
					\draw[arrow] (C2exp.south) -- (Dexp.north);
					\draw[arrow] (Dexp.south) -- (Eexp.north);
					\draw[arrow] (Eexp.south) -- (Fexp.north);
					\draw[arrow] (Fexp.south) -- (Gexp.north);
					\draw[arrow] (Gexp.south) -- (Hexp.north);
					
				\end{tikzpicture}
			\end{adjustbox}
		}
		\medskip
		\small
		\caption{Roadmap of the proof of \Cref{T.Well-posedness} and \Cref{C.convergence}. Left column: references to formulas and intermediate results. Right column: role of each step.} \label{F.roadline}
	\end{center}
\end{figure}

	The remainder of the paper is organized as follows. In \Cref{S.main-results} we introduce the general framework, state the precise hypotheses imposed on the class of systems under consideration, and present the main results. \Cref{S.discussion} places these results in the context of the existing literature on singular limits, while \Cref{S.scale} illustrates the small-scale development mechanism that may occur within the present class of systems.
	The technical part of the paper begins in \Cref{S.Admissible}, where we introduce
	 admissible weights which are key ingredients of our weighted energy method. Stability estimates adapted to our class of systems are proved in \Cref{S.Stability}; they are then used in \Cref{S.completion} to establish uniform well-posedness results. 
	The analysis is further refined in \Cref{S.H5H6} under stronger assumptions on the symmetrizer.
	Finally, \Cref{S.applications} is devoted to applications. 
	We first consider the rapidly rotating shallow-water system in \Cref{S.rotating-SV}, and then the hyperbolized Benjamin--Bona--Mahony, Boussinesq--Peregrine and Serre--Green--Naghdi systems in \Cref{S.hBBM}, \Cref{S.hBP} and \Cref{S.LCT}, respectively. 
	For the reader's convenience, we collect functional analysis tools in \Cref{S.technical} and notations in \Cref{S.notations}.
	
	We now indicate how these sections may be read depending on the reader's goals. The paper is organized so as to be read at two complementary levels. 
	
	Readers primarily interested in the consequences and applications of the main results may focus first on \Cref{S.main-results}, where the general framework and the central well-posedness result, \Cref{T.Well-posedness}, are stated, together with the convergence result of \Cref{C.convergence} and the refined statements under additional structural assumptions. The discussion in \Cref{S.discussion} then places these results within the existing literature on two-scale and multi-scale singular limits, while \Cref{S.scale} exhibits through elementary models and numerical experiments the mechanism of rapid development of small spatial oscillations 
	within the class of systems considered in this paper. One may then proceed directly to \Cref{S.applications}, where the hypotheses and conclusions of the proposed developments are translated into concrete geophysical and water-wave models, as indicated just above. This route highlights the practical meaning of the well-preparedness assumptions, the role of the two independent stiff parameters, and the improvement over standard arguments in situations where nonlinearity or spatially varying coefficients (induced by bathymetry) may trigger the small-scale development.
	
	Readers interested in the proof of \Cref{T.Well-posedness} and tools employed to this aim should instead regard \Cref{S.Admissible,S.Stability,S.completion} together with Annex~\ref{S.technical} as the core of the paper. Annex~\ref{S.technical} provides the functional-analytic tools needed in the proofs, and in particular the anisotropic space-time product estimate. The notion of admissible weights is introduced in \Cref{S.Admissible}. These weights are the main device allowing one to measure, and propagate, different amounts of spatial and temporal regularity according to the scales $\delta$, $\varepsilon$ implied in the dynamics. \Cref{S.Stability} contains the central stability estimate: it is the technical heart of the argument and shows how the skew-adjoint stiff operator, the spectral separation assumption, and the admissibility properties of the weights ---that we discuss in the following section--- combine to control weighted energies uniformly with respect to the parameters. \Cref{S.completion} then converts this stability estimate into the uniform well-posedness result of \Cref{T.Well-posedness}. A key ingredient is the proof that the preparation assumption on a limited number of time derivatives of the initial data implies the initial boundedness of the full weighted energy. Finally, \Cref{S.H5H6} explains how the argument simplifies or improves when the symmetrizer satisfies stronger structural assumptions, leading to \Cref{T.H5,T.H6}.
	
	To make the logical structure of the proof more accessible, especially to non-specialist readers, \Cref{F.roadline} summarizes the main ingredients leading to \Cref{T.Well-posedness} and indicates the role of each intermediate result.

\subsection{Main results}\label{S.main-results}
We are interested in solutions to three-scale systems of the form
\begin{equation}\label{eq.general}
	\S_0(U)\partial_t U + \sum_{l=1}^d\S_l(U)\partial_{x_l}U +\G(U)U = \frac1\eps\L_\delta U ,
\end{equation}
where $U:(t,\bx)\in \RR\times\RR^d\to\RR^n$ (with $d,n\in\NN^\star$) is the unknown, $\delta$ and $\eps$ are positive parameters which restricted to some set\footnote{We do not specifically impose additional restrictions on the parameters (such as the prescribed limit $\delta\approx \eps^\theta$ with $\theta\in[0,1]$) but in practical applications, restrictions may be induced by Hypotheses~\ref{H1}--\ref{H4} (and possibly~\ref{H5} or~\ref{H6}), or to fulfill other assumptions of our main results.}
\[	\mathcal S\subset \{ (\eps,\delta)\in\RR^2 \ : \ 0<\eps\leq\delta\leq 1\}\]
and we assume the system satisfies the following hypotheses.
\begin{enumerate}[({H}1),series=H]
	\item \label{H1} $\S_l$ ($l\in\{0,\dots,d\})$ (respectively $\G$) are smooth functions into $n\times n$ real-valued symmetric matrices (respectively $n\times n$ real-valued matrices), and
	$\G(\bm0)$ is negative-semidefinite. 
	\item \label{H2} There exists $\Omega\subset \RR^n$ (the domain of hyperbolicity) open neighborhood of $\{\bm{0}\}\in\RR^n$ such that for any $U\in \Omega$, $\S_0(U)$ is definite positive. Consequently, for any $K\subset \Omega$ compact, there exists a constant $c_K>0$ such that for any $V\in\RR^n$ and any $U\in K$ and any $(\eps,\delta)\in\cS$
	\[ \bra{\S_0(U) V,V}\geq c_K \norm{V}^2\]
	where $\bra{\cdot,\cdot}$ and $\norm{\cdot}$ denote the Euclidean inner product and norm in $\RR^n$.
	\item \label{H3} $\L_\delta$ is a constant-coefficient differential or pseudo-differential operator (see \Cref{D.Fourier-multipliers}) $\L_\delta=\sL_\delta(\frac1{\i}\nabla_{\bx})$ whose symbol $\sL_\delta(\cdot) $ is skew-Hermitian and satisfies for some constant $N_\L>0$ independent of $(\eps,\delta)\in\cS$
		\[ \forall \bk \in \RR^d, \quad \Vert \sL_\delta(\bk)\Vert \leq N_\L\, (1+\delta|\bk|).\]
	\item \label{H4} There exist constants $C_\L\geq0$ and $c_\L>0$ such that for all $\bk \in \RR^d$, there exists $\sPi_\delta(\bk)$ an orthogonal projection matrix commuting with $\sL_\delta(\bk)$ such that for all $(\eps,\delta)\in\cS$ and $U\in\RR^n$
	\begin{align*}  
	\norm{ \big(\sPi_\delta \sL_\delta \sPi_\delta\big)(\bk)U} &\leq \eps\, C_\L\, (1+|\bk|) \norm{ U},\\
	\norm{ \big((\Id-\sPi_\delta) \sL_\delta (\Id-\sPi_\delta)\big)(\bk) U} &\geq  \phantom{\eps} \, c_\L\, (1+\delta|\bk|) \norm{(\Id-\sPi_\delta)(\bk) U}.
	\end{align*}
\end{enumerate}
We also consider situations with an extra hypothesis on the matrix-valued function $\S_0$:
	\begin{enumerate}[resume*=H]
		\item \label{H5} For any $K\subset \Omega$, there exists $C_{k,K}'>0$ such that for all $\ell\in\{1,\dots,k\}$ and any $(\eps,\delta)\in\cS$
		\[ \Norm{ \S_0^{(\ell)}}_{(\RR^n)^\ell\to M_n(\RR)}\leq \frac\eps\delta\, C_{k,K}' .\]
		\item \label{H6}  For any $K\subset \Omega$, there exists $C_{k,K}''>0$ such that for all $\ell\in\{1,\dots,k\}$ and any $(\eps,\delta)\in\cS$
		\[ \Norm{ \S_0^{(\ell)}}_{(\RR^n)^\ell\to M_n(\RR)}\leq \big(\frac\eps\delta\big)^\ell C_{k,K}'' .\]
	\end{enumerate}
\begin{Remark}\label{R.uniform}
	We do allow the matrix-valued functions $\S_0$, $\S_l$ ($l\in\{0,\dots,d\})$ and $\G$ to depend on the parameters $\eps$ and $\delta$, but this dependency should not be singular, in the sense that the control on derivatives, the open set $\Omega$ and the lower bound on $c_K>0$ can be chosen uniformly with respect to parameters $(\eps,\delta)\in\cS$.

	The smoothness assumption on $\S_0$, $\S_l$ ($l\in\{0,\dots,d\})$ and $\G$ can be relaxed to being differentiable of order $k$ with bounded derivatives (in the considered compact subdomain $K\subset \Omega\subset\RR^n$), where $k\in\NN$ is the index of regularity at stake in the statement. Specifically we assume that for any $K\subset \Omega$, there exists $C_{k,K}>0$ such that for all $\ell\in\{0,\dots,k\}$ and any $(\eps,\delta)\in\cS$
	\[ \Norm{ \S_0^{(\ell)}}_{(\RR^n)^\ell\to M_n(\RR)} +\sum_{l=1}^d\Norm{ \S_l^{(\ell)}}_{(\RR^n)^\ell\to M_n(\RR)}+\Norm{ \G^{(\ell)}}_{(\RR^n)^\ell\to M_n(\RR)} \leq C_{k,K}\]
	where the index $(\ell)$ denotes the $\ell$-th order total derivative.
	For the sake of readability, the dependency of our estimates with respect to the upper bound on derivatives up to the order $k$, namely $C_{k,K}$, is omitted in our statements.
	
	The assumption that $\G(\bm0)$ is negative-semidefinite is not crucial to our analysis. Our results still hold withdrawing this assumption, yet on a timescale which is only uniformly bounded from below instead of being inversely proportional to the size of the initial data.
\end{Remark}
\begin{Remark}
	All our results adapt trivially to the periodic framework $\bx\in (L\TT)^d$ in which case the hypotheses on the symbol $\sL_\delta$ should hold for any wavevectors on the grid $\bk\in(\frac{2\pi}{L}\NN)^d$.
\end{Remark}
\begin{Remark}\label{R.H4}
	Hypothesis~\ref{H4} is an assumption of scale-separation between groups of eigenvalues of the symbol $\sL_\delta$ since $\sPi_\delta$ may be defined as the orthogonal projection onto the direct sum of the eigenspaces associated with eigenvalues of size $\cO(\eps\,  (1+|\bk|))$, while we assume that all other eigenvalues are of size $\gtrsim (1+\delta|\bk|) $.  Such property may stem either from structural properties of the operator, restrictions on the considered set of parameters $(\eps,\delta)\in\cS$, or a combination of both.
	
	Indeed, since $\sL_\delta$ is skew-symmetric, its eigenvalues lie on the imaginary axis. When $\sL_\delta$ depends holomorphically on the parameter $\delta$ and in particular in the prototypical example
	\[ \sL_\delta(\bk) = \sL_0+\delta \sL_1(\bk)\]
	where $\sL_0$ and $\sL_1(\bk)$ are $n\times n$ skew-symmetric matrices, then the eigenvalues of $\sL_\delta(\bk)$ depend holomorphically on the parameter $\delta$: we have \cite[Theorem 6.1 in Chap. II]{Kato95} for all $\lambda_\delta\in \Spec(\sL_\delta(\bk))$,
	\[ \lambda_\delta=\lambda^{(0)} + \delta \lambda^{(1)} + \delta^2 \lambda^{(2)} +\dots\]
	with $\lambda^{(0)}\in  \Spec(\sL_0)$. As $\eps$ may take arbitrarily small values, eigenvalues $\lambda_\delta\in \Spec(\sL_\delta(\bk))$ of size $|\lambda_\delta|\lesssim\eps\,  (1+|\bk|)$ necessarily correspond to $\lambda^{(0)}=0$. Hence for Hypothesis~\ref{H4} to hold we need that such eigenvalues originating from $\lambda^{(0)}=0\in \Spec(\sL_0)$ (when $\delta=0$) satisfy $|\lambda_\delta|\lesssim\eps\,  (1+|\bk|)$ on the considered set of parameters $(\eps,\delta)\in\cS$. This may follow either from structural properties of the operator inducing $\lambda^{(\ell)}=0$ for $\ell=1,2,\dots$ (in which case $\lambda_\delta=0$), restrictions on the set of parameters $(\eps,\delta)\in\cS$ (namely $\delta\lesssim \eps$), or a combination of both (for instance $\lambda^{(1)}=0$ and $\delta^2\lesssim\eps$). 
\end{Remark}

Our main result builds upon the standard local-in-time well-posedness of the initial-value problem for system~\eqref{eq.general} and provides the propagation in time of the control of some suitable energy functionals uniformly with respect to parameters $(\eps,\delta)\in\cS$. 

\begin{Theorem}[Well-posedness]\label{T.Well-posedness}
	Let $k\in\NN,\ k>d/2+1$. For any $(\eps,\delta)\in\cS$ and any $U_0\in H^k(\RR^d)$ satisfying the hyperbolicity condition $U_0(\RR^d)\subset K\subset\Omega$, there exists a unique $U\in \cC(I_{\eps,\delta};H^k(\RR^d))$ maximal-in-time classical solution to system~\eqref{eq.general} under hypotheses~\ref{H1}--\ref{H3} emerging from the initial data $U\big\vert_{t=0}=U_0$, and one has $U\in \cap_{j=0}^k\cC^j(I_{\eps,\delta};H^{k-j}(\RR^d))$. 
	
	Under the additional hypothesis~\ref{H4}, for any $C_0>0$, $j_0\in\NN^\star$, $i_0\geq 0$ such that $i_0+j_0>d/2+1$ and $ j_\sharp\in\NN $ such that $j_\sharp\geq j_0$, there exists $T>0$, $C>0$ and $\lambda\geq 1$ depending uniquely on $k,j_0,i_0,c_K,N_\L,C_\L,c_\L$ and $C_0$ such that the following holds. 
	
	Assuming that $(\eps,\delta)\in\cS$ satisfies $\delta\leq1/\lambda$, $ \eps \lambda\leq \delta$ and $(\eps\lambda)^{j_\sharp-j_0}\leq \delta^{i_0}$ and
		\begin{equation}\label{eq.well-prepared}
			M_{j_\sharp,0}\eqdef \left.\left( \sum_{j=0}^{j_\sharp-1} \norm{\partial_t^j U}_{H^{k-j}}+\frac{\eps^{ j_\sharp-j_0}}{\delta^{i_0}} \norm{\partial_t^{j_\sharp}U}_{H^{k-j_\sharp}}\right)\right\vert_{t=0} \leq C_0 
		\end{equation}
	then $I_{\eps,\delta}\supset [-T/M_{j_\sharp,0},T/M_{j_\sharp,0}]$ and for all $t\in [-T/M_{j_\sharp,0},T/M_{j_\sharp,0}]$ holds ${\{U(t,\bx)  :  \bx\in\RR^d\}\subset \Omega}$~and
			\begin{equation}\label{eq.estimate}
				\forall j\in\{0,\dots,j_0\}, \ \forall i\in [0,k-j], \quad \norm{\partial_t ^j U(t,\cdot)}_{H^i} \leq C \, \, M_{j_\sharp,0} \,\times\, \max(\{1,\delta^{i_0+j_0-i-j}\}) .
			\end{equation}
\end{Theorem}

\begin{Remark}[Well-prepared initial data]\label{R.well-prepared}
	The upper bound~\eqref{eq.well-prepared} is an assumption of well-prepared initial data which is in general stronger than the standard assumption $\frac1\eps \L_\delta U_0=\cO(1)$ for small values of $\delta$---even in the situation where $\delta\in(0,1]$ is fixed if one desires controls that are uniform with respect to $\delta\in(0,1]$. The main motivation of this work is to provide the weakest possible constraints on the data, and from this viewpoint to set $j_\sharp$ as small as possible. Observe that it may impose some additional
	restrictions on the parameters $(\eps,\delta)\in\cS$ through the assumption ${(\eps\lambda)^{j_\sharp-j_0}\leq \delta^{i_0}}$: setting $j_\sharp=j_0$ imposes $\delta\geq 1$, while setting $j_\sharp=j_0+1$  yields  $0<\eps\lesssim \delta^{\max(\{1,i_0\})}$ and setting $j_\sharp=j_0+i_0$ with $i_0> 1$ yields $\eps\lesssim\delta$.
	From another viewpoint, if one considers the prescribed limit $\delta\approx \eps^\theta$ with $\theta\in[0,1]$, we find that it imposes $j_\sharp\geq j_0+\lceil \theta i_0\rceil$. Endpoint cases $\theta=0$ ($\delta\approx 1$) and $\theta=1$ ($\delta\approx \eps$)
	are respectively the most and least favorable situations in terms of restrictions on initial data.
\end{Remark}
\begin{Remark}[Admissible weights and development of a small spatial scale]\label{R.weights}
	
		\Cref{T.Well-posedness} follows from a more precise result whose statement is postponed to \Cref{P.Well-posedness} and which provides the propagation of the control of a class of energy functionals involving time and space derivatives of its variable with different weights:
\begin{equation}\label{eq.def-F}
	\cF_{k,{\bm\alpha}}(U)\eqdef\sup\big(\big\{ \alpha_{j,i}^{-1} \norm{\partial_t^j U}_{H^{i}} \ : \ (i,j)\in\RR_+\times\NN,\ i+j\leq k\big\}\big)
\end{equation}
with suitable choices of $\bm\alpha\eqdef(\alpha_{j,i}>0\ : \  (i,j)\in\RR_+\times\NN,\ i+j\leq k)$
describing the typical sizes of the corresponding time and space derivatives of the solutions. 
The explicit set of weights we use to prove \Cref{T.Well-posedness} is 
\begin{equation}\label{eq.def-alpha}
	\alpha_{j,i}\eqdef\max(\{1,\lambda^{j-j_0}, \lambda^{j-j_0}\delta^{i_0+j_0-i-j},\eps^{j_0-j}\delta^{i_0-i}\}).
\end{equation}
	We represent these weights (setting $\lambda=1$ as its value does not play an essential role) in \Cref{F.diagram}.
	
	Notice the dependency of the weights with respect to $\delta$ (from which stems the non-uniform upper bounds $\norm{\partial_t^j U(t,\cdot)}_{H^i}=\cO(\delta^{i_0+j_0-i-j})$ for $i+j>i_0+j_0$) that can be viewed as the signature of the aforementioned development of a spatial scale of size $\delta$ ---a situation that does not occur in the two-scale problem $\delta\approx 1$--- that we discuss in more details later on.	
	
	Incidentally, a close inspection at the proof of \Cref{T.Well-posedness} shows that the result holds for {\em any} $\lambda$ sufficiently large and such that $\delta\leq1/\lambda$, $ \eps \lambda\leq \delta$ and $(\eps\lambda)^{j_\sharp-j_0}\leq \delta^{i_0}$; and that in that case the estimate can be sharpened to
		\[ \forall j\in\{0,\dots,j_0\}, \ \forall i\in [0,k-j], \quad \norm{\partial_t ^j U(t,\cdot)}_{H^i} \leq C \, \, M_{j_\sharp,0} \,\times\, \max(\{1,\lambda^{j-j_0}\delta^{i_0+j_0-i-j}\}) \]
		provided one restricts the time interval to  $t\in[-T/(M_{j_\sharp,0}(1+\lambda)),T/(M_{j_\sharp,0}(1+\lambda))]$. Such result carries additional information on the progressive development of the small spatial scale. 
\end{Remark}
\begin{figure}
	\begin{centering}
		\subcaptionbox{$\delta=1$\label{F.diagram-c}}{\includegraphics[width=.33\textwidth]{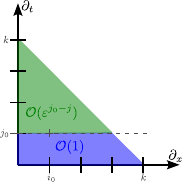}}%
		\subcaptionbox{$\eps<\delta<1$\label{F.diagram-b}}{\includegraphics[width=.33\textwidth]{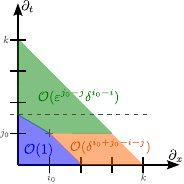}}%
		\subcaptionbox{$\delta=\eps$\label{F.diagram-a}}{\includegraphics[width=.33\textwidth]{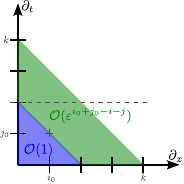}}%
	\end{centering}
	
	\caption{Representation of the admissible set of weights $\bm\alpha$ set in~\eqref{eq.def-alpha} with $j_0=i_0=1$, and $\lambda=1$. The abscissa depicts spatial regularity indices while the ordinate depicts time regularity indices. The horizontal dashed line represent the minimal value that $j_\sharp\in\NN$ may take in \Cref{T.Well-posedness}. 
	}
	\label{F.diagram}
\end{figure}

Because the \Cref{T.Well-posedness} provides among other things a uniform control of the Lipschitz norm in {\em time and space} of solutions (with well-prepared initial data and suitable choice of parameters $(\eps,\delta)\in\cS$), one may infer the strong convergence of solutions as $\eps\searrow 0$. The analysis was provided in \cite[Section~4]{ChengJuSchochet18} and we only state the corresponding result. 
\begin{Corollary}[Convergence]\label{C.convergence}
	Consider $\cS'\subset \cS$ such that entries of $\S_{0},\S_{i},\G$ depend continuously on the parameters $(\eps,\delta)\in\cS'$, and for all $\bk\in\RR^d$, $\sPi_\delta(\bk)\to \sPi(\bk)$ and $(\frac1\eps\sPi_\delta \sL_\delta)(\bk) \to \sT(\bk)$ as $(\eps,\delta)\to (0,\bar\delta)$ in $\cS'$. Denote $\P$ (respectively $\T$) the pseudo-differential operator with symbol $\sPi$ (respectively $\sT$). 
	
	For all $U_0^{\eps,\delta}$ converging to $U_0$ in $H^k(\RR^d)$ as $(\eps,\delta)\to (0,\bar\delta)$ in $\cS'$ and such that the assumptions of \Cref{T.Well-posedness} hold uniformly, ${U^{\eps,\delta}\in \cC(I;H^k(\RR^d))}$ the solutions to ~\eqref{eq.general} emerging from the initial data $U^{\eps,\delta}\big\vert_{t=0}=U_0^{\eps,\delta}$ converge weak-$*$ in $L^\infty(I;H^k(\RR^d))$ and strongly in $\cC(I;H^{k_-}(\RR^d))$ for any $k_-<k$ towards $U\in L^\infty(I;H^k(\RR^d))\cap {\rm Lip}(I;L^2(\RR^d))$ the unique solution in that space to
	\begin{equation}\label{eq.limit}
		\begin{cases}\P\left(\S_{0}(U)\partial_t U + \sum_{l=1}^d\S_{i}(U)\partial_{x_l}U +\G(U)U + \T U\right) = 0 ,\\[1ex]
		(\Id-\P)U=0, \\[1ex]
		 U\vert_{t=0}= U_{0},
	\end{cases}
	\end{equation}
	where (misusing notations) entries of $\S_{0},\S_{i},\G$ are evaluated at $\eps=0$ and $\delta=\bar{\delta}$.
\end{Corollary}
\begin{Remark}\label{R.convergence}
	The assumption that $\sPi_\delta(\bk)\to \sPi(\bk)$ and $(\frac1\eps\sPi_\delta \sL_\delta)(\bk) \to \sT(\bk)$ as $(\eps,\delta)\in\cS'\to(0,\bar\delta)$ is not very restrictive. As detailed in \cite[Lemma 4.1]{ChengJuSchochet18} and based on the perturbation theory of self-adjoint matrices \cite[Chap. II.6]{Kato95}, if $\sL_\delta$ depends holomorphically (or sufficiently smoothly) on $\delta$ and $\cS'$ enforces to the prescribed limit $\eps \delta^n \to \sigma>0$ with $n\in\NN$, then one has $\lim\sPi_\delta=\sPi_{\bar\delta}$ and one may obtain $\sT=\lim \frac1\eps\sPi_\delta \sL_\delta$ from a reduction process (in which case $\sT$ may depend on $\sigma$). Moreover one has $\sT=0$ in situations where $\eps \delta^n \to \infty$ and $\eps \delta^{n+1} \to 0$.   
\end{Remark}

To conclude with our main results we state that if Hypothesis~\ref{H5} or~\ref{H6} additionally holds, our results can be improved either to provide a stronger control on the solution or assuming less stringent assumptions on the initial data. The weights in~\eqref{eq.def-F} that provide the following results are displayed in \Cref{S.H5H6} and illustrated in \Cref{F.diagram-2}.
	\begin{Theorem}	\label{T.H5}
	If Hypothesis~\ref{H5} holds in addition to~\ref{H1}--\ref{H4}, \Cref{T.Well-posedness} and \Cref{C.convergence} hold and one has additionally
		\[ \forall j\in\{0,\dots,j_0\}, \ \forall i\in [0,k-j], \quad \norm{\partial_t ^j U(t,\cdot)}_{H^i} \leq C \, \, M_{j_\sharp,0} \,\times\, \max(\{1,\eps^{j_0-j}\delta^{i_0-i}\}) .\]
\end{Theorem}
\begin{Theorem}\label{T.H6}
	If Hypothesis~\ref{H6} holds in addition to~\ref{H1}--\ref{H4}, \Cref{T.Well-posedness} also holds with $j_0=0$ and $i_0>d/2+1$. In that case we do not have necessarily a uniform bound on $\norm{\partial_t U}_{L^2}$, and \Cref{C.convergence} may not hold.
\end{Theorem}
\begin{figure}
	\begin{centering}
		\subcaptionbox{\Cref{T.Well-posedness}, $j_0=i_0=1$\label{F.diagram-2-a}}{\includegraphics[width=.33\textwidth]{diagram.pdf}}%
		\subcaptionbox{\Cref{T.H5}, $j_0=i_0=1$\label{F.diagram-2-b}}{\includegraphics[width=.33\textwidth]{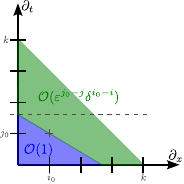}}%
		\subcaptionbox{\Cref{T.H6}, $j_0=0,i_0=2$\label{F.diagram-2-c}}{\includegraphics[width=.33\textwidth]{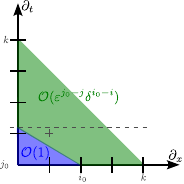}}%
	\end{centering}
	
	\caption{Representation of the admissible weight $\bm\alpha$ used in the proof of \Cref{T.Well-posedness}, \Cref{T.H5} and \Cref{T.H6}. One has $\eps<\delta<1$, and $\lambda=1$. The abscissa depicts spatial regularity indices while the ordinate depicts time regularity indices.
		The horizontal dashed line represent the minimal value that $j_\sharp\in\NN$ may take in \Cref{T.Well-posedness}, \Cref{T.H5} and \Cref{T.H6} respectively.
		}
	\label{F.diagram-2}
\end{figure}

\subsection{Discussion}\label{S.discussion}

\paragraph{Standard two-scale singular limits}
The class of systems we study, namely~\eqref{eq.general} with Hypotheses~\ref{H1}--\ref{H4}, embeds well-studied singular limits, in particular considering operators $\L_\delta$ of the form
\[ \L_\delta = \sL_0+\delta\sum_{l=1}^d\sL_l \partial_{x_l}\]
where $\sL_0$ is a skew-symmetric matrix and for all $l\in\{1,\dots,d\}$, $\sL_l$ is a symmetric matrix.
 First and foremost, when $\sL_0=0$, our problem belongs to {\em low Mach number} or {\em weakly compressible} limits whose study have been initiated by Browning and Kreiss~\cite{BrowningKreiss82}, Klainerman and Majda~\cite{KlainermanMajda81,KlainermanMajda82} among others; see a detailed account in~\cite{Gallagher05,Schochet05,Alazard08}. In our framework of general systems with well-prepared initial data, the decisive result has been obtained by Schochet~\cite{Schochet86a}. In this work the author proves the analogue of \Cref{T.Well-posedness} and \Cref{C.convergence}, that is uniform local-in-time well-posedness and convergence of solutions for well-prepared initial data ---which in this case take a simpler form, namely $\norm{U_0}_{H^k}=\cO(1)$ and $\norm{\partial_t U\vert_{t=0}}_{H^{k-1}}=\cO(1)$. Unsurprisingly since our strategy follows closely the approach in~\cite{Schochet86a}, our result is entirely consistent: when $\delta\approx 1$, that is when order-zero stiff terms are not dominant, we find that the aforementioned notion of well-prepared initial data secures uniform local-in-time well-posedness and convergence results; see \Cref{R.well-prepared} and \Cref{F.diagram-c}. We additionally prove that the control of energy functionals associated with stronger notions of well-prepared data (up to the strongest one considered in~\cite{BrowningKreiss82}) propagate in time. We argue in the following section that, due to a mechanism of development of a small spatial scale, such result cannot hold (in the absence of additional structural assumptions on the system at stake) when order-zero stiff terms are dominant, that is when $\delta\ll 1$.

It is tempting to compare the case when $\delta\lesssim \eps$ (or, equivalently, $\sL_l=0$ for all $l\in\{1,\dots,d\}$) with another widely studied two-scale singular limit, that is {\em relaxation limits}. For such problems, a seminal work is~\cite{Yong99} by Yong, which investigates uniform well-posedness, convergence of solutions, and describes the boundary layer (in time) for ill-prepared initial data. It should be noticed however that the stability condition therein (see also~\cite{KawashimaYong04}) exhibit the (partially) dissipative nature of stiff terms while, due to our assumption of skew-symmetry for the matrix $\sL_0$, our framework is fully conservative. As for general balance laws with stiff source terms in the conservative framework, we may first refer to the work of Gallagher~\cite{Gallagher98} which considers system of the form~\eqref{eq.general} assuming $\S_0=\Id$. This restriction is crucial as can be seen from the small-scale development mechanism described in the following section which prevents suitable stability estimates in the absence of additional structural assumptions on the system such as considered by Schochet in~\cite{Schochet87}, or a very strong notion of well-prepared initial data. Indeed, explicit solutions to the toy model~\eqref{eq.toy} with $\delta=0$ show that one cannot expect in general a uniform local-in-time well-posedness and convergence result outside of strongly prepared initial data, namely $\norm{U_0}_{H^k}=\cO(1)$ and $\norm{\partial_t^j U\vert_{t=0}}_{H^{k-j}}=\cO(1)$ for all $j\in\{1,\dots,k\}$ (a result which can easily be obtained by the approach of~\cite{BrowningKreiss82}); see \Cref{R.well-prepared} and \Cref{F.diagram-a}.

Our work provides a connection between the framework of low Mach number limits ($\delta\approx 1$) and the framework of conservative stiff source terms ($\delta\lesssim \eps$) by considering the situation $\eps\ll\delta\ll 1$. Our results express a notion of preparation for initial data which ``interpolates'' between the standard notion of the former framework, and the stronger notion associated with the latter; see \Cref{R.well-prepared} and \Cref{F.diagram-b}. Let us once again emphasize that additional structural hypotheses on the considered system may allow to weaken the necessary assumptions of initial preparation; we discuss such hypotheses in~\ref{H5} and~\ref{H6}, and the outcome is illustrated in \Cref{F.diagram-2}. We do not explicitly discuss the strongest structural assumption $\Norm{ \S_0'}_{(\RR^n)^\ell\to M_n(\RR)}\lesssim \eps $ since, as already mentioned and exemplified in \Cref{S.rotating-SV}, satisfactory results can be obtained through standard arguments.
\medskip

\paragraph{Three-scale singular limits} Let us now review some works in the literature that specifically consider three-scale singular limits. As mentioned previously, the analysis of fluid motion at planetary scale provides natural examples where different effects arise at distinct scales, determined by the size of independent dimensionless numbers, such as Froude ($\Fr$), Rossby ($\Ro$) and Mach ($\Ma$) numbers (see {\em e.g.~\cite{Vallis17}}). There has been many mathematical works on such systems, and it is outside of the scope of this discussion to review all of them. Let us mention that most of the works in the literature focus on a {\em prescribed limit}, that is relate all dimensionless parameters to a single vanishing parameter $\eps\searrow 0$. Yet in the past decade some effort has been made to extend singular limits so as to consider multi-scale limits. For instance, concerning the Boussinesq system (that is assuming incompressibility, $\Ma=0$), we can mention the works of Mu and Schochet~\cite{MuSchochet22} and Mu and Wei~\cite{MuWei23}, and Jo, Kim and Lee~\cite{JoKimLee24} that investigate the simultaneous limit $\Fr\searrow 0$) and Rossby $\Ro\searrow 0$ with dominant stratification ($\Fr/\Ro\searrow 0$) for the former, dominant rotation ($\Ro/\Fr\searrow 0$) for the second, and balanced stratification and rotation for the latter; see also the earlier works of Ju and Mu~\cite{JuMu19,MuJu21}. In the situation of compressible fluids ($\Ma> 0$) but neglecting stratification effects ($\Fr=\infty$), Mu~\cite{Mu24} and Ko, Pausader, Takada \& Widmayer~\cite{KoPausaderTakadaEtAl25} considered recently the rotating isentropic 3-dimensional Euler equations in the two-scale singular limit $\Ro\searrow 0$, $\Ma\searrow 0$ and either $\Ro/\Ma= \nu$ with $\nu>0$ fixed, or $\Ro/\Ma\searrow 0$, or (in the latter work) $\Ma/\Ro\searrow 0$. The 2-dimensional system (with a special focus on the rotational shallow-water system which is a specific case) was already considered by Cheng and Tadmor in~\cite{ChengTadmor08,Cheng09}. Considering now the general system with the effect of stratification, compressibility, rotation as well as viscosity, Feireisl, Gallagher, Gerard-Varet \& Novotný~\cite{FeireislGallagherGerard-VaretEtAl12} as well as Feireisl and Novotný~\cite{FeireislNovotny14} studied the limit of small Rossby, Mach, Froude {\em and} Reynolds numbers. In these works the limit is restricted to a single parameter, but they consider different scalings $\Ro=\eps$, $\Ma=\eps^m$, $\Fr=\eps^n$, $\Rey=\eps^{-\alpha}$, and the situation is considered as multi-scale when the limit system describing the oscillatory part of the solution cannot be rescaled to be parameter-independent. Del Santo, Fanelli, Sbaiz and Wróblewska-Kamińska~\cite{DelSantoFanelliSbaizEtAl23} considered a similar setting (with $\alpha=0$) but with low stratification ($m>2 n$). We let the reader refer to this latter work for an interesting overview of previous results of multiscale analyses.

In all the works mentioned in the previous paragraph, the authors use a combination of energy methods and dispersive estimates (the last mentioned work relying additionally on a compensated compactness argument), the latter being key to consider either ill-prepared initial data, or extend the lifespan of solutions. In essence the authors consider a specific system of equation and use as much as possible the structure of the system to obtain a fine description of the asymptotic behavior of solutions. Our perspective is different, as we wish to obtain robust results that apply to a general class of systems. As such, let us clarify that our analysis brings essentially no novel information with respect to previous works\footnote{This comment is not entirely true: our analysis provides the propagation in time of the high-regularity control of solutions that allow small-amplitude rapid spatial oscillations for sufficiently well-prepared data, a setting which is not considered in the literature.}, since the uniform control of solutions on the timescale we consider (and in particular for well-prepared initial data) for such systems follows from standard energy methods, due to the fact that $\S_0=\Id$ in their framework (again, we emphasize this point in one of our applications in \Cref{S.rotating-SV}). 

Of course we cannot expect in general that physically realistic systems satisfy the restrictive assumption $\S_0=\Id$. Let us mention for instance the works~\cite{KleinAchatzBreschEtAl10,BreschKleinLiu22} on atmospheric flow models where non-constant Brunt–Väisälä frequency yields space-dependency in the symmetrizer $\S_0$. Even more delicate is the situation where $\S_0$ depends on the solution which we investigate in this work.
\medskip

\paragraph{Studies for general classes of systems} The works whose framework is the closest to our analysis are the ones of Cheng, Ju and Schochet~\cite{ChengJuSchochet18} and Schochet and Xu~\cite{SchochetXu20} dealing with general systems of the form~\eqref{eq.general} satisfying~\ref{H1}-\ref{H2} but not (necessarily)~\ref{H3}-\ref{H4}, and assuming additionally~\ref{H6}. Specifically, the authors consider singular operators $\L_\delta=\L_0+\delta\L_1$ where $\L_0,\L_1$ are constant-coefficient, skew-symmetric (for the $L^2$ inner-product) differential operators of order at most 1, and assume that the symmetrizer is of the form $\S_0(U)=\widetilde\S_0(\tfrac\eps\delta U)$, which leads to Hypothesis~\ref{H6}.
In the same line of works, we would also like to mention~\cite{ChengJuSchochet21} which studies the three-scale singular limit of a specific system of the compressible ideal magnetohydrodynamics equations where $\L_0,\L_1$ are order-one differential operators, and~\cite{SchochetXu22} which investigates a class of systems with a special structure that allows to consider ill-prepared initial data. Returning to~\cite{ChengJuSchochet18,SchochetXu20}, the authors prove uniform estimates on solutions provided that the initial data are such that $\norm{U\vert_{t=0}}_{H^{k}}+\eps^\alpha\norm{\partial_t U\vert_{t=0}}_{H^{k-1}}=\cO(1)$ and the parameters $\eps,\delta>0$ satisfy $\eps\lesssim \delta^{\frac{k}{1-\alpha}}$ for some parameter $\alpha\in [0,1)$ ($\alpha=0$ in~\cite{ChengJuSchochet18} and $\alpha\in(0,1)$ in~\cite{SchochetXu20}).
This should be compared with \Cref{T.H6} (that is \Cref{T.Well-posedness} with $j_0=0$, $j_\sharp=1$ and $i_0=k$) relying on the condition $\norm{U\vert_{t=0}}_{H^{k}}+\eps\delta^{-k}\norm{\partial_t U\vert_{t=0}}_{H^{k-1}}=\cO(1)$, which is less stringent when $\eps\lesssim \delta^{\frac{k}{1-\alpha}}$. Moreover, our parameter restriction $\eps\leq \delta^{k}/\lambda$ is less stringent than $\eps\lesssim \delta^{\frac{k}{1-\alpha}}$ unless $\alpha=0$. Hence our result at least partially improves~\cite{ChengJuSchochet18,SchochetXu20}, although it should be recalled that our framework imposes additional assumptions on the singular operator $\L_\delta$ (Hypotheses~\ref{H3} and~\ref{H4}).
 This improvement stems from the fact that we use spatially weighted energy functionals (as advocated in \cite[eq.~(3.1)]{ChengJuSchochet18}) as well as the strategy introduced by Schochet in~\cite{Schochet86a} to infer the control of our energy functionals from the control of $\norm{U}_{H^{k}}$ and all time derivatives $\norm{\partial_t^j U}_{L^2}$, $j=1,\dots,k$.

Finally, let us mention a preceding work of the first author~\cite{Duchene19} which considered a system of the form~\eqref{eq.general} satisfying (almost; see discussion in \Cref{S.LCT}) Hypotheses~\ref{H1}--\ref{H4}. While the strategy in this work is similar to the one employed here ---and in particular follows the ideas of~\cite{Schochet86a}--- the energy functional considered in~\cite{Duchene19} required a much stronger restriction on initial data:
\[  \sum_{j=0}^{j_0}\norm{\partial_t^j U\vert_{t=0}}_{H^{k-j}}+\sum_{j=j_0+1}^{k} (\eps/\delta)^{j-j_0} \norm{\partial_t^j U}_{H^{k-j}}=\cO(1)\]
where $j_0\in\{1,\dots,k\}$. The main incentive of this work was to relax the assumption of preparation of initial data, and our main novel ingredient is the use of suitable spatially weighted energy functionals, which are adapted to the small-scale development mechanism we discuss in the following section.

\subsection{Small-scale development mechanism}\label{S.scale}

Let us now describe the mechanism of small-scale development that can be triggered by systems of the form~\eqref{eq.general}. To this aim we consider the following toy model
\begin{equation}\label{eq.toy}
	\left\{\begin{aligned}
	e^w\partial_tu&=\frac\delta\eps\partial_x u-\frac1\eps v,\\[1ex]
	e^w\partial_tv&=\frac\delta\eps\partial_x v+\frac1\eps u,\\[1ex]
	\partial_tw&=0.
	\end{aligned}\right.
\end{equation}
It is obvious that system~\eqref{eq.toy} is of the form~\eqref{eq.general} and satisfies~\ref{H1}--\ref{H3}. Notice however that it does not satisfy~\ref{H4} since $\sL_\delta$ the symbol associated with the operator 
\[\L_\delta\eqdef \begin{pmatrix}
	\delta\partial_x&1&0\\
	-1&\delta\partial_x&0\\
	0&0&0
\end{pmatrix}\]
has eigenvalues $\lambda_0(k)=0$ and $\lambda_\pm(k)=\i(\pm1+\delta k) $ which continuously evolves from size $\gtrsim 1$ at $k=0$ to size $\cO(\eps)$ for Fourier wavenumbers at distance $\epsilon$ from the critical value $k_\star\eqdef \mp\delta^{-1}$.
\medskip

Denoting $w_0\eqdef w(t=0,\cdot)$ and $z=u+\i v$, we find that~\eqref{eq.toy} reads equivalently
\[ e^{w_0}\partial_t z=\frac\delta\eps\partial_x z+\frac\i\eps z.\]
Let us first notice that if $\delta=0$, then the explicit solutions
\[z(t,\cdot)=z_0(\cdot)\exp(t\, \frac\i\eps e^{-w_0(\cdot)}), \quad z_0(\cdot)=z(t=0,\cdot),\]
exhibit fast spatial oscillations of shrinking-with-time typical wavelength $\ell$ that emerge on the timescale of order $\cO(\eps/\ell)$, as soon as $w_0$ has spatial variations. This naturally leads to investigating the situation $0<\delta \ll 1$ which is the framework of our study.

Consider now the case $\delta>0$.
We introduce the flow map $\Phi(t,x)$ defined by
\[\partial_t\Phi(t,\cdot)=-\frac\delta\eps e^{-w_0(\Phi(t,\cdot))}, \qquad \Phi(t=0,x)=x,\]
and one quickly checks that $Z(t,\cdot)=z(t,\Phi(t,\cdot))$ satisfies the identity
\[ Z(t,x)=Z(0,x)e^{-\frac\i\delta(\Phi(t,x)-x)}.\]

It is now apparent that a new scale is triggered (when $\delta\ll1$) by variations of  $\Phi(t,x)-x$ that emerge on the timescale of size $\eps/\delta$. Indeed, we have
\[ \norm{(\partial_x u,\partial_x v)}_{L^\infty}\approx \norm{\partial_x z}_{L^\infty}=\norm{ \tfrac1{\partial_x\Phi}(\partial_x Z)\circ\Phi^{-1}}_{L^\infty}\]
and
\[|\partial_x Z|=|(\partial_x z_0)\circ\Phi-\frac{\i}\delta (\partial_x\Phi-1)(z_0\circ\Phi)|.\]
Yet one has
\[\partial_x\Phi(t,\cdot)=e^{w_0(\cdot)-w_0(\Phi(t,\cdot))},\]
which shows that 
\[ \forall t,x\in\RR^2,\quad \exp(\inf(w_0)-\sup(w_0)) \leq \partial_x\Phi\leq \exp(\sup(w_0)-\inf(w_0)),\]
and we have $\Phi(t,x)\leq x-\frac\delta\eps t e^{-\sup(w_0)}$, hence we see by a continuity argument that, given $x_1,x_2\in\RR$ such that $w_0(x_1)\neq  w_0(x_2)$, there exists $0\leq t\leq \frac\eps\delta e^{\sup(w_0)}|x_1-x_2|$ such that
\[ \inf|\partial_x\Phi(t,\cdot)-1|\geq |e^{w_0(x_1)-w_0(x_2)} -1|.\]
As a conclusion, we find that unless one imposes decay or smallness assumptions on the initial data, then one has
\[\sup_{t\in [0,\frac\eps\delta]}\norm{(\partial_x u,\partial_x v)}_{L^\infty}\approx \delta^{-1}.\]
The same reasoning for higher-order derivatives (provided initial data are sufficiently regular) yields
\[ \forall k\in\NN, \quad \sup_{t\in [0,\frac\eps\delta]}\norm{(\partial_x^k u,\partial_x^k v)}_{L^\infty}\approx \delta^{-k}.\]
Such estimates are the signature of the emergence of oscillations with wavelength $\delta$ on the timescale of size $\cO(\eps/\delta)$.
\bigskip

One could reasonably oppose that this behavior may be a consequence of the failure of Hypothesis~\ref{H4} about wavenumber $k_\star=\mp\delta^{-1}$ which was exhibited previously. We claim that it is not the case and that such behavior also occurs for systems satisfying the additional hypothesis~\ref{H4}, and in particular to the toy model
\begin{equation}\label{eq.toy2}
	\left\{\begin{aligned}
		e^w\partial_tu&=\frac\delta\eps\partial_x u-\frac1\eps v,\\[1ex]
		e^w\partial_tv&=-\frac\delta\eps\partial_x v+\frac1\eps u,\\[1ex]
		\partial_tw&=0.
	\end{aligned}\right.
\end{equation}
System~\eqref{eq.toy2} is of the form~\eqref{eq.general} and $\widetilde\sL_\delta$ the symbol associated with the operator 
\[\widetilde \L_\delta\eqdef \begin{pmatrix}
	\delta\partial_x&1&0\\
	-1&-\delta\partial_x&0\\
	0&0&0
\end{pmatrix},\]
has eigenvalues $\widetilde\lambda_0(k)=0$ and $\widetilde\lambda_\pm(k)=\pm\i\sqrt{1+(\delta k)^2} $ and hence
satisfies the required estimate in~\ref{H4} with $\sPi_\delta(\bk)$ the orthogonal projection matrix onto $\Ker(\widetilde\sL_\delta)$.
\medskip

We cannot carry out similar computations for~\eqref{eq.toy} as we did for~\eqref{eq.toy2}, which is why we resort to numerical experiments to exhibit the development of spatial oscillations with small wavelength. These computations were generated using the Fourier pseudo-spectral method for spatial discretization, and the standard fourth-order Runge-Kutta method (RK4) for time integration. Specifically, we used the Julia library \texttt{WaterWaves1D.jl}~\cite{DucheneNavaro}.

We set the parameter $\eps=10^{-3}$ and numerically solve~\eqref{eq.toy2} for $\delta\in\{2.\, 10^{-2}, 10^{-2}, 5.\,10^{-3}\}$. We set as initial data
\[u\vert_{t=0}=\tfrac12 e^{-x^2}, \quad v\vert_{t=0}=0, \quad w\vert_{t=0}=\tfrac12 e^{-x^2}.\]
We used $N=2^{12}$ regularly spaced collocation points on the space interval $[-\pi,\pi)$, and a time-step $\Delta t=10^{-4}$. The fact that Fourier coefficients decrease up to machine precision for large wavenumbers gives us confidence that our numerical solutions are fully resolved, while the small timestep and fourth-order of the time-integration solver provides largely enough precision for the needs of our illustration (setting $\Delta t=10^{-3}$ does not lead to any noticeable difference).

 We plot in \Cref{F.plts} the outcome at different values of time, for $\delta\in\{2.\, 10^{-2}, 10^{-2}, 5.\,10^{-3}\}$. One clearly sees that oscillations with small characteristic wavelength emerge progressively. By looking at the amplitude of Fourier coefficients we see that as $\delta$ decreases, it takes a longer time for oscillations to fully develop because their characteristic wavelength decreases. Shortly put we observe for the model~\eqref{eq.toy2} a behavior which is fully consistent with the one exhibited for model~\eqref{eq.toy}.

\begin{figure}[htbp]
	\centering
	\subfloat[$\delta=2.\, 10^{-2}$]{\label{fig:a}\includegraphics[width=0.6\linewidth]{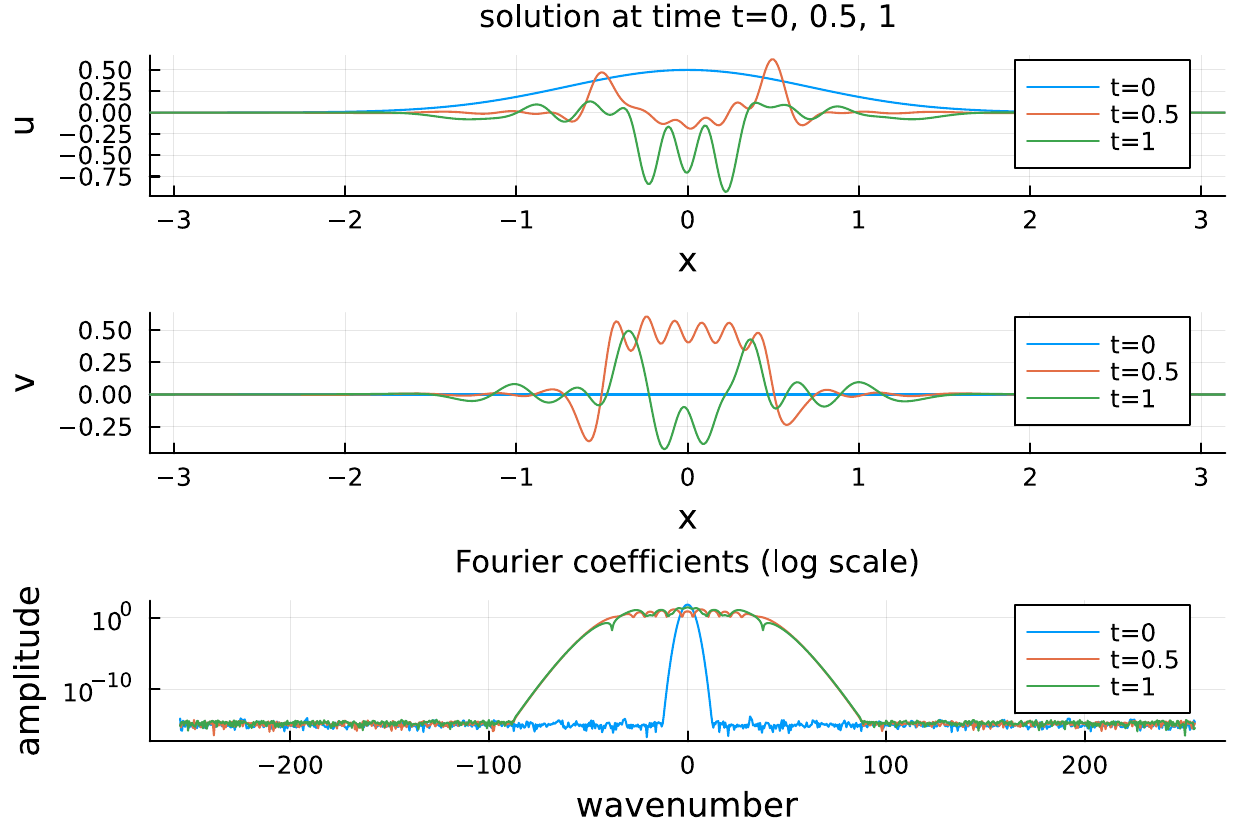}}\qquad
	\subfloat[$\delta= 10^{-2}$]{\label{fig:b}\includegraphics[width=0.6\linewidth]{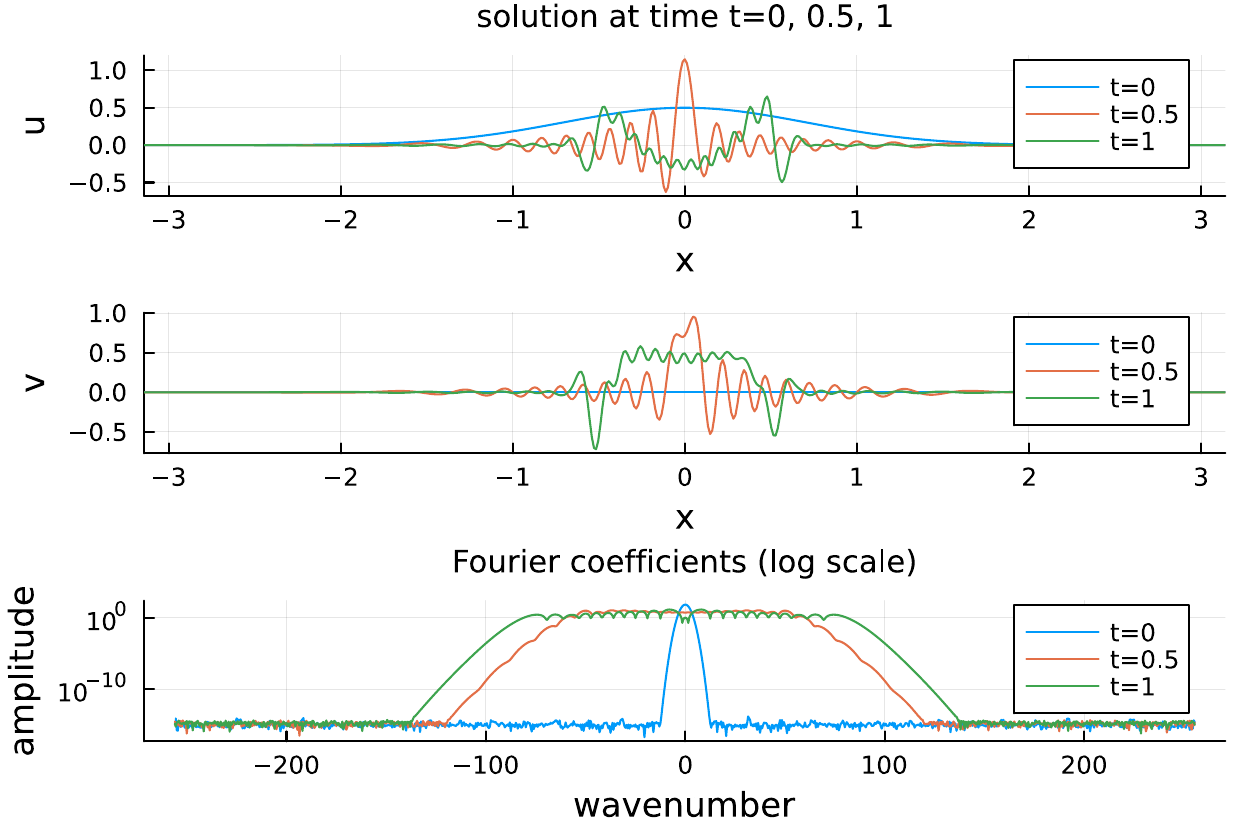}}\qquad
	\subfloat[$\delta=5.\, 10^{-3}$]{\label{fig:c}\includegraphics[width=0.6\linewidth]{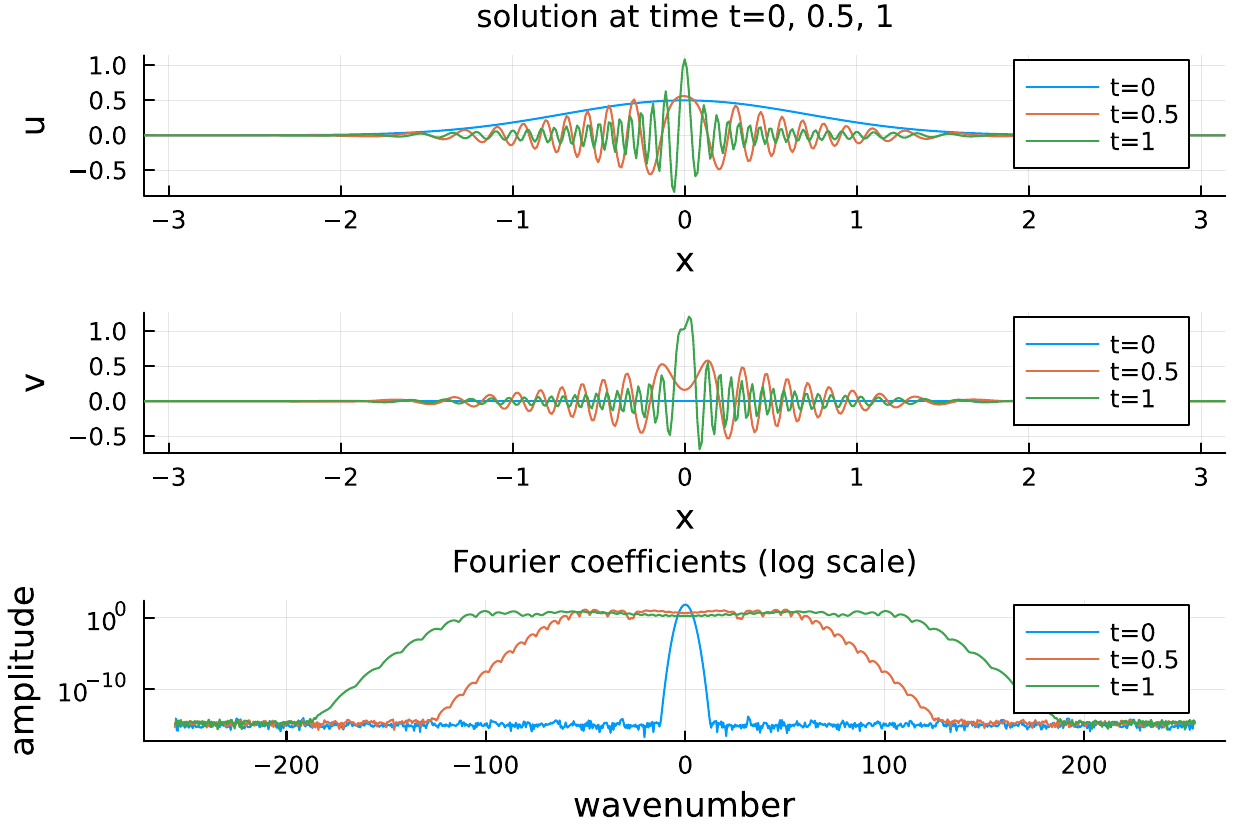}}
	\caption{Numerical solutions to the toy model~\eqref{eq.toy2}. The top two panels represent $u$ and $v$, the bottom panel represents the modulus of Fourier coefficients of $u$ (in semi-log scale).}
	\label{F.plts}
\end{figure}

\section{Admissible weights}\label{S.Admissible}

We define our class of admissible weights as follows.
\begin{Definition}[Admissible weights]\label{D.alpha-admissible}
	We say the set $\bm\alpha=\big(\alpha_{j,i}>0 \ : \  (i,j)\in\RR\times\NN,\ i+j\leq k\big)$ is $\lambda$-{\em admissible} if there exist $j_0\in\NN^\star$ and $s_0>d/2$ (clear from the context) such that the following relations hold:
	\begin{enumerate}[({A}1),series=A]
		\item \label{ii} $ \alpha_{0,s_0+1}\leq 1$ and $\alpha_{1,s_0}\leq 1$;
		\item \label{iii} $\ell\in\NN\mapsto \alpha_{j+\ell,i-\ell}$ and $i\in\RR\mapsto \alpha_{j,i}$ (and hence $j\in\NN\mapsto \alpha_{j,i}$) are non-decreasing;
		\item \label{v} $ \alpha_{j+1,i-1}\leq \frac\delta\eps\alpha_{j,i}$ for all $(i,j)\in\RR\times\NN$ and $\alpha_{j+1,0}\leq \frac{\delta^i}\eps\alpha_{j,i}$ if moreover $i\in [0,1]$;
		\item \label{iv} For all $j\in\{0,\dots,j_0-1\}$, $\alpha_{j+1,i-1}\leq \lambda\alpha_{j,i}$ if $i\geq 1$, and $\alpha_{j+1,0}\leq \lambda \alpha_{j,0}$;
		\item \label{l} For all $j\geq j_0$, $ \alpha_{j+1,i-1}\geq \lambda\alpha_{j,i}$  and if moreover $i>s_0+1$, $ \alpha_{j,i}\geq \lambda\alpha_{j,i-1}$;
		\item \label{vi} For all $(j^\star,i^\star)$ and $(j_\star,i_\star)$ with $j_\star\leq j^\star$ and $(j_n,i_n+j_n)_{n=0,\dots,\ell}$ satisfying
					\[\sum_{n=0}^\ell j_n = j^\star+\ell j_\star, \quad \sum_{n=0}^\ell i_n+j_n = i^\star+j^\star+\ell (i_\star+j_\star)\]
					and for all $ n\in\{0,\dots,\ell\}$,
					\[
				\begin{cases}
						(j_n,i_n+j_n)\in[j_\star,j^\star]\times[i_\star+j_\star,i^\star+j^\star]  & \text{ if } i_\star+j_\star \leq i^\star+j^\star,\\
						(j_n,i_n+j_n)=\theta_n (j_\star,i_\star+j_\star)+(1-\theta_n)(j^\star,i^\star+j^\star) ,\ \theta_n=\frac{(i_n+j_n)-(i^\star+j^\star)}{(i_\star+j_\star)-(i^\star+j^\star)} & \text{ if } i_\star+j_\star > i^\star+j^\star,
					\end{cases}
				\]
		one has
		\begin{equation}\label{eq.ineq-alpha}
		\prod_{n=0}^\ell \alpha_{j_n,i_n} \leq  \alpha_{j_\star,i_\star}^\ell \times \alpha_{j^\star,i^\star}.
		\end{equation}
	\end{enumerate}
\end{Definition}

We now provide explicit families of weights that are admissible in the sense of \Cref{D.alpha-admissible}.
\begin{Lemma}
\label{L.weights}
Let $j_0\in\NN^\star$ and $i_0\geq 0$ such that $s_0\eqdef i_0+j_0-1>d/2$. Set for $0<\eps\leq\delta\leq1$
	\begin{equation} \label{eq.def-alpha-general}
		\alpha_{j,i} = \max\big(\big\{\alpha_{j,i}^{(1)}, \dots, \alpha_{j,i}^{(P)}\big\}\big) \quad \text{ where }\quad 
		\alpha_{j,i}^{(p)}\eqdef \frac{\beta_p^j\gamma_p^{i+j} }{\beta_p^{j_0} \gamma_p^{i_0+j_0}} 
		\end{equation}
	with $P\in\NN^\star$ and $\beta_p,\gamma_p$ (for $p\in\{1,\dots,P\}$) are such that
	\[1=\beta_1\leq  \dots\leq\beta_P\leq \frac\delta\eps\quad\text{ and }\quad 1=\gamma_1\leq\dots\leq\gamma_P\leq \frac1\delta.\]
	Then $\bm\alpha=\big(\alpha_{j,i} \ : \  (i,j)\in\RR_+\times\NN,\ i+j\leq k\big)$ satisfies~\ref{ii}--\ref{v} and~\ref{vi} in \Cref{D.alpha-admissible}.
\end{Lemma}

\begin{proof}
	Only~\ref{vi} deserves an explanation.
	 Assume first that $i^\star+j^\star< i_\star+j_\star$. In that case the result follows from the fact that the function $\max$ is a log-log function of its arguments, and composition rules. Specifically,  
	 we have that $\theta_n\in[0,1]$ ($n\in\{0,\dots,\ell\}$) is such that $j_n=\theta_n j^\star+(1-\theta_n)j_\star$ and $i_n+j_n=\theta_n(i^\star+j^\star)+(1-\theta_n)(i_\star+j_\star)$, and one has $\sum_{n=0}^\ell\theta_n=1$.
	 It follows that 
	 \[	\forall n\in\{0,\dots,\ell\},\ 	\forall p\in\{0,\dots,P\} ,\qquad \alpha_{j_n,i_n}^{(p)}=(\alpha_{j^\star,i^\star}^{(p)})^{\theta_n}(\alpha_{j_\star,i_\star}^{(p)})^{1-\theta_n}.\]
	 For any $n\in\{0,\dots,\ell\}$, there exists $p_n\in \{1,\dots,P\}$ such that
	 \[\alpha_{j_n,i_n}=\alpha_{j_n,i_n}^{(p_n)}=(\alpha_{j^\star,i^\star}^{(p_n)})^{\theta_n}(\alpha_{j_\star,i_\star}^{(p_n)})^{1-\theta_n}\leq (\alpha_{j^\star,i^\star})^{\theta_n}(\alpha_{j_\star,i_\star})^{1-\theta_n}.\]
	The desired inequality~\eqref{eq.ineq-alpha} follows immediately.
	
	Let us now assume that $i^\star+j^\star\geq  i_\star+j_\star$. Then 
	there exists $(\theta_n,\sigma_n)\in[0,1]^2$ ($n\in\{0,\dots,\ell\}$) such that $j_n=\theta_n j^\star+(1-\theta_n)j_\star$ and $i_n+j_n=\sigma_n(i^\star+j^\star)+(1-\sigma_n)(i_\star+j_\star)$, and one has $\sum_{n=0}^\ell\theta_n=\sum_{n=0}^\ell\sigma_n=1$.
	Denote for $n\in\{0,\dots,\ell\}$
	\[p_n=\arg\max\big(\big\{\alpha_{j_n,i_n}^{(p)} \ : \ p\in\{1, \dots, P\}\big\}\big).\]
	Up to permuting indices, one can assume without loss of generality that
	\[\forall n\in\{0,\dots,\ell-1\}, \quad p_n\leq p_\ell.\] 
	Then we write
	\begin{align*}		
		\prod_{n=0}^\ell \alpha_{j_n,i_n} 
		&=\prod_{n=0}^\ell \alpha^{(p_n)}_{j_n,i_n}=\prod_{n=0}^\ell  \beta_{p_n}^{j_n-j_0}\gamma_{p_n}^{i_n+j_n-i_0-j_0}\\
		&=\prod_{n=0}^\ell  \beta_{p_n}^{\theta_n(j^\star-j_0)+(1-\theta_n)(j_\star-j_0)}\gamma_{p_n}^{\sigma_n(i^\star+j^\star-i_0-j_0)+(1-\sigma_n)(i_\star+j_\star-i_0-j_0)}\\
		&=
		\Big(\prod_{n=0}^{\ell-1}\beta_{p_n}^{j_\star-j_0}\gamma_{p_n}^{i_\star+j_\star-i_0-j_0}\Big)\Big(\beta_{p_\ell}^{j^\star-j_0}\gamma_{p_\ell}^{i^\star+j^\star-i_0-j_0} \Big)\\
		&\quad \times 
		\Big(\prod_{n=0}^{\ell-1} \beta_{p_n}^{\theta_n (j^\star-j_\star)}\gamma_{p_n}^{\sigma_n(i^\star+j^\star-i_\star-j_\star)}\Big)
		\Big(\beta_{p_\ell}^{(\theta_\ell-1) (j^\star-j_\star)}\gamma_{p_\ell}^{(\sigma_\ell-1)(i^\star+j^\star-i_\star-j_\star)}\Big)\\
		&\leq  \Big(\prod_{n=0}^{\ell-1}\alpha_{j_\star,i_\star}^{(p_n)} \Big) \times \alpha^{(p_\ell)}_{j^\star,i^\star}
	\end{align*}
	where we used in the last inequality that $j^\star-j_\star,i^\star+j^\star- i_\star-j_\star\geq 0$, that for any $n\in\{0,\dots,\ell-1\}$, $\theta_n,\sigma_n\geq0$ and $p_n\leq p_\ell$ and that $p\in\{1,\dots,P\}\mapsto\beta_p,\gamma_p\geq1$ are non-decreasing (so that $\beta_{p_n}\leq \beta_{p_\ell}$ and $ \gamma_{p_n}\leq \gamma_{p_\ell}$) and recall that $\sum_{n=0}^\ell\theta_n=\sum_{n=0}^\ell\sigma_n=1$. The desired inequality~\eqref{eq.ineq-alpha} follows immediately.
\end{proof}
\begin{Corollary}\label{C.weights}
	Let $j_0\in\NN^\star$ and $i_0\geq 0$ such that $s_0\eqdef i_0+j_0-1>d/2$. Set for $0<\eps\leq\delta\leq1$ and $1\leq \lambda\leq \max(\{\delta/\eps,1/\delta\})$
\[\alpha_{j,i}\eqdef\max(\{1,\lambda^{j-j_0}, \lambda^{j-j_0}\delta^{i_0+j_0-i-j},\eps^{j_0-j}\delta^{i_0-i}\}).\]
Then $\bm\alpha^\lambda_{j_0,i_0}=\big(\alpha_{j,i} \ : \  (i,j)\in\RR_+\times\NN,\ i+j\leq k\big)$ is $\lambda$-admissible in the sense of \Cref{D.alpha-admissible}.
\end{Corollary}
\begin{proof}
	Notice $\bm\alpha^\lambda_{j_0,i_0}$ is of the form~\eqref{eq.def-alpha-general} with $P=4$ and $(\beta_1,\gamma_1)=(1,1)$, $(\beta_2,\gamma_2)=(\lambda,1)$,  $(\beta_3,\gamma_3)=(\lambda,1/\delta)$, and $(\beta_4,\gamma_4)=(\delta/\eps,1/\delta)$, so  \Cref{L.weights} applies. 
		Notice then that since $\lambda\leq\delta/\eps$ one has
		\[ \alpha_{j,i}=\begin{cases}\max(\{1,\lambda^{j-j_0}/\delta^{i+j-i_0-j_0}\}) & \text{ if $j\leq j_0$,}\\
			\max(\{\lambda^{j-j_0},1/(\eps^{j-j_0}\delta^{i-i_0})\}) & \text{ if $j\geq j_0$.}
		\end{cases}
			\] 
		In particular $\alpha_{j,i}=1$ if $j\leq j_0$ and $i+j\leq i_0+j_0=s_0+1$, and $\alpha_{j,i}=\eps^{j_0-j}\delta^{i_0-i}$ if $j\geq j_0$ and $i+j\geq i_0+j_0=s_0+1$.	From this we easily infer~\ref{iv} and~\ref{l}.
\end{proof}

\section{Stability}\label{S.Stability}

This section is dedicated to the proof of a stability estimate concerning the system
\begin{equation}\label{eq.linearized}
	\S_0(V)\partial_t U + \sum_{l=1}^d\S_l(V)\partial_{x_l}U + \G(V) U + R = \frac1\eps\L_\delta U
\end{equation}
where $U:(t,\bx)\in \RR\times\RR^d\to\RR^n$, $V:(t,\bx)\in \RR\times\RR^d\to\RR^{n'}$ and we assume Hypotheses~\ref{H1}--\ref{H4} to hold. System~\eqref{eq.linearized} is a linearized version of~\eqref{eq.general} and the vector $R$ accounts for harmless source terms.\footnote{While such source terms are unessential for the present study, we believe they would arise for instance when seeking asymptotic expression for solutions to~\eqref{eq.linearized} ---$V$ being then a reference solution and $U$ the difference between the exact and reference solutions---   and include it for the sake of completeness. } We seek to propagate in time the control of the energy functional 
\begin{equation}\label{eq.def-F-stability}
	\cF_{k,{\bm\alpha}}(U)\eqdef\sup\big(\big\{ \alpha_{j,i}^{-1} \norm{\partial_t^j U}_{H^{i}} \ : \ (i,j)\in\RR_+\times\NN,\ i+j\leq k\big\}\big)
\end{equation}
 under the assumption that $\cF_{k,{\bm\alpha}}(V)$ and $\cF_{k,{\bm\alpha}}(R)$ are controlled. 
 
 We also introduce the following energy functional\footnote{Of course we abuse notation in that $\cF_{k,{\bm\alpha}}^{(0)}$ depends on the prescribed time-and-space depending function $V$.}
 \begin{align}	
 	\cF_{k,{\bm\alpha}}^{(0)}(U)&\eqdef\sup\big( \big\{\alpha_{j,i}^{-1} \bra{\S_0(V)\partial_t^j \Lambda^i U,\partial_t^j \Lambda^i U}^{1/2}\ : \ (i,j)\in \omega_{k,{\bm\alpha}}^{(0)}\big\}\big) ,
 	 \label{eq.def-F0-stability} 
 \end{align} 
where 
\begin{align*}
	\omega_{k,{\bm\alpha}}^{(0)}
	\eqdef \big\{(i,j) \ : \ j\in\{0,\dots,j_0-1\},\ i\in [0,k-j]\big\} \bigcup  \big\{(0,j) \ : \ j\in \{0,1,\dots,k\}\big\}.
\end{align*}

 Following the strategy of~\cite{Schochet86a}, we shall provide an energy inequality for the functional $\cF_{k,{\bm\alpha}}^{(0)}(U)$ by means of the energy method and use Grönwall's inequality to obtain the desired estimate, while the control of $\cF_{k,{\bm\alpha}}(U)$ is inferred from the control of $\cF_{k,{\bm\alpha}}^{(0)}(U)$ using directly the system of equations~\eqref{eq.linearized}. The main difficulty is to obtain estimates which are uniform with respect to $(\eps,\delta)$ satisfying $0<\eps\leq \delta\leq 1$. Our main result reads as follows.

\begin{Proposition}[Stability] \label{P.Stability}
	Let $k\in\NN,\ k>d/2+1$ and $M^{(0)},M>0$. Let $T>0$ and consider $V,R\in \cap_{j=0}^k\cC^j([0,T];H^{k-j}(\RR^d))$ and $U\in \cap_{j=0}^{k+1}\cC^j([0,T];H^{k+1-j}(\RR^d))$ be solutions to system~\eqref{eq.linearized} under hypotheses~\ref{H1}--\ref{H4} (defined in the introduction), and such that 
	\[\forall (t,\bx)\in[0,T]\times\RR^d, \quad V(t,\bx)\in K\subset\Omega\]
	with $K$ a compact subset of the hyperbolicity domain defined in Hypothesis~\ref{H2}.
	There exists  $C^{(0)}$ depending uniquely on $k,c_K,C_{\L},c_\L$ and $M^{(0)}$ and $C$, $\lambda_0$, depending additionally on $M$, such that the following holds. 
	
	For any $(\eps,\delta)$ satisfying $0<\eps\leq \delta\leq 1$ and $C\times(\eps/\delta)\leq 1$, and for any set of weights $\bm\alpha$, $\lambda$-admissible in the sense of \Cref{D.alpha-admissible} with $\lambda\geq \lambda_0$ and such that
	\[\sup_{t\in[0,T]}\cF_{k,{\bm\alpha}}^{(0)}(V)\leq M^{(0)} \quad  \text{ and } \quad  \sup_{t\in[0,T]}\cF_{k,{\bm\alpha}}(V)\leq M,\]
	one has
	\begin{align*} 
		\cF_{k,{\bm\alpha}}^{(0)}(U)\big\vert_{t=T} &\leq  \Big(\cF_{k,{\bm\alpha}}^{(0)}(U) \big\vert_{t=0} + C \int_0^T \cF_{k,{\bm\alpha}}(R)\dd t\Big)\times   \exp\big(C\, \, (1+\lambda)\, M\, T\big) ,\\
		\cF_{k,{\bm\alpha}}(U)\big\vert_{t=T} &\leq  C^{(0)}\, \cF_{k,{\bm\alpha}}^{(0)}(U)\big\vert_{t=T} +  C^{(0)}\, \cF_{k,{\bm\alpha}}(R)\big\vert_{t=T}.
	\end{align*}
\end{Proposition}

We provide in \Cref{S.general-estimates} some general estimates on system~\eqref{eq.linearized}. We first apply energy estimates on space and time derivatives of solutions, yielding the differential inequality~\eqref{eq.est-energy}. We then notice that unsuitable contributions vanish when energy estimates are applied to time derivatives only, which yields the differential inequality~\eqref{eq.est-energy-0}. These two differential inequalities and Grönwall's inequality eventually yield the first estimate in \Cref{P.Stability} for suitable choices of weights. The crucial remark is that we cannot readily infer a self-contained control of the solutions, due to the lack of control of mixed space-and-time derivatives. Specifically the energy estimates are useful either when dealing with time derivatives of order less than $j_0-1$ (for admissible weights to provide suitable uniform control of higher time derivatives) or without space derivatives (due to the aforementioned cancellation). In order to bootstrap the control of the remaining space-and-time derivatives, we need to use in a stronger way the structure of our system, and in particular Hypotheses~\ref{H4}. We first use the system~\eqref{eq.linearized} on the non-singular domain of the operator $\L_\delta$ as defined through Hypothesis~\ref{H4} to infer some control on the regular component of space-and-time derivatives involving essentially lower-order time derivatives: see~\eqref{eq.est-forward}. Finally we use the system~\eqref{eq.linearized} on the singular domain of the operator $\L_\delta$ to infer some control on the singular component of space-and-time derivatives involving higher-order time derivatives: see~\eqref{eq.est-backward} and~\eqref{eq.est-backward-0}. We show in \Cref{S.bootstrap} how these estimates can be bootstrapped (using a finite induction process helped by the parameter $\lambda$ and using fully the assumptions of admissible weights) to obtain the second estimate in \Cref{P.Stability}. The key estimates in \Cref{S.bootstrap} are~\eqref{eq.control-F0} and~\eqref{eq.control-F}, which quickly yield \Cref{P.Stability}.

\subsection{General estimates}\label{S.general-estimates}

In the following we denote $C(\cdot)$ a non-negative quantity depending non-decreasingly on its argument, which may change from line to line. Recall we note $a\lesssim b$ when $a\leq C\, b$ with $C$ an unessential constant.
Dependency with respect to the dimension and regularity indices are omitted. We recall that $\S_0$, $\S_l$ and $\G$ may depend on the parameters $\eps$ and $\delta$, but that this dependency is non-singular with respect to $(\eps,\delta)\in\cS$, and will never be spelled out. In other terms, all estimates in this section are uniform with respect to $(\eps,\delta)\in\cS$ unless these parameters explicitly appear. We set $s_0>d/2$ and we shall use without notice the continuous Sobolev embedding $H^{s_0}(\RR^d)\hookrightarrow L^\infty(\RR^d)$. Recall we define $\Lambda^i\eqdef (\Id-\Delta_{\bx})^{i/2}$ and that for all $s\in\RR $, $\norm{V}_{H^s}=\norm{\Lambda^s V}_{L^2}$.

We first prove {\bf energy estimates}, using the symmetric structure of the system of equations~\eqref{eq.linearized} to derive a differential inequality on a suitable energy functional. We then prove {\bf forward estimates} using the system of equations on the regular space defined by the projection matrix $\sPi_\delta$ introduced in  Hypothesis~\ref{H4}. We finally prove {\bf backward estimates} using the system of equations and the invertibility of the operator $\L_\delta$ on the singular space.

In order to better visualize the different contributions in these estimates, we introduce the notation
\begin{multline} \label{eq.def-Nell}
	N_\ell(\underbrace{V,\dots,V}_{\text{$\ell$ times}};U)\big\vert_{j_\star,k_\star}^{j^\star,k^\star}\\
	\eqdef \sup\Big( \Big\{ \big(\prod_{n=1}^\ell \norm{ \partial_t^{j_n} V}_{H^{i_n}} \big)\norm{ \partial_t^{j_0} U}_{H^{i_0}} \ : \   ((i_n,j_n))_n\in(\RR_+\times\NN)^{1+\ell} \text{ satisfies } \\
	\sum_{n=0}^\ell j_n = j^\star+\ell j_\star, \quad \sum_{n=0}^\ell i_n+j_n = k^\star+\ell k_\star,\\
	\text{ and for all $ n\in\{0,\dots,\ell\}$, }  (j_n,i_n+j_n)\in[(j_\star,k_\star),(j^\star,k^\star)]
	 \Big\}\Big)
\end{multline}
where we denote $	 (j_n,i_n+j_n)\in[(j_\star,k_\star),(j^\star,k^\star)]$ if and only if
\[	\begin{cases}
	(j_n,i_n+j_n)\in[j_\star,j^\star]\times[k_\star,k^\star]  & \text{ if } k_\star \leq k^\star,\\
	(j_n,i_n+j_n)=\theta_n (j_\star,k_\star)+(1-\theta_n)(j^\star,k^\star) ,\ \theta_n=\frac{(i_n+j_n)-k^\star}{k_\star-k^\star} & \text{ if } k_\star > k^\star.
\end{cases}\]
This notation allows to monitor the indices of the highest-regularity norm $(j^\star,k^\star)$ that will eventually arise in tame estimates; see \Cref{R.tame} and \cref{eq.key} below.

\paragraph{Energy estimates.} 
Let $i\in \RR_+$ and $j\in \NN$ such that $i+j\leq k$. Denote $U^{j,i}\eqdef \partial_t^j \Lambda^i  U$. Applying $\partial_t^j \Lambda^i$ to~\eqref{eq.linearized}, taking the $L^2(\RR^d)^n$ inner-product against $U^{j,i}$, and making use of integration by parts, the symmetry of $\S_l(V)$ ($l\in\{0,\dots,d\}$), that $\G(\bm0)$ is positive-semidefinite (Hypothesis~\ref{H1}) and $\L^\delta$ is skew-symmetric (Hypothesis~\ref{H3}) we have
\begin{align*} \frac12\frac{\dd}{\dd t} \bra{\S_0(V)U^{j,i},U^{j,i}} 
	&\leq \frac12  \bra{[\partial_t,\S_0(V)]U^{j,i},U^{j,i}}+ \sum_{l=1}^d\frac12  \bra{[\partial_{x_l},\S_l(V)]U^{j,i},U^{j,i}}\\
	&\quad - \bra{[\partial_t^j\Lambda^i,\S_0(V)]\partial_t U, U^{j,i}}- \sum_{l=1}^d\bra{[\partial_t^j\Lambda^i,\S_l(V)]\partial_{x_l} U,U^{j,i}}\\
	&\quad - \bra{\partial_t^j\Lambda^i(\G(V) U-\G(\bm0) U),U^{j,i}}- \bra{\partial_t^j\Lambda^i R ,U^{j,i}}.
\end{align*}
All the terms on the right-hand side are estimated by Cauchy-Schwarz inequality.
The first two terms are estimated with
\begin{align*} \norm{ [\partial_t,\S_0(V)]U^{j,i}}_{L^2}+ \norm{[\partial_{x_l},\S_l(V)] U^{j,i}}_{L^2}
	&\leq C(\norm{V}_{L^\infty})\,\big(\norm{\partial_tV}_{L^\infty}+\norm{\partial_{x_l}V}_{L^\infty}\big) \norm{U^{j,i}}_{L^2}\\
	&\leq C(\norm{V}_{H^{s_0}})\,\big(\norm{\partial_tV}_{H^{s_0}}+\norm{V}_{H^{s_0+1}}\big)\norm{\partial_t^j U}_{H^i}\\
	&\leq C(\norm{V}_{H^{s_0}})\,  \big( N_1(V;U)\big\vert_{1,s_0+1}^{j,i+j}+N_1(V;U)\big\vert_{0,s_0+1}^{j,i+j}\big).
\end{align*}
For the third term we decompose
\[ [\partial_t^j\Lambda^i,\S_0(V)]\partial_t U = [\Lambda^i,\S_0(V)]\partial_t^{j+1} U+\Lambda^i\big([\partial_t^j,\S_0(V)]\partial_t U\big)\]
and estimate each term separately. By tame commutator and composition estimates in \Cref{L.commutator,L.composition} we estimate the first contribution as
\begin{align*} \norm{[\Lambda^i,\S_0(V)]\partial_t^{j+1} U}_{L^2}&\leq C(\norm{V}_{H^{s_0}})\,\Big(\norm{ V}_{H^{s_0+1}} \norm{ \partial_t^{j+1} U}_{H^{i-1}}+\big\langle \norm{ V}_{H^{i}} \norm{\partial_t^{j+1} U}_{H^{s_0}}\big\rangle_{i>s_0+1}\Big)\\
		&\leq C(\norm{V}_{H^{s_0}})\, N_1(V;U)\big\vert_{0,s_0+1}^{j+1,i+j}
	\end{align*}
(we used the condition $ i>s_0+1$ to infer $s_0+1\leq \min(\{i,s_0+j+1\})\leq \max(\{i,s_0+j+1\})\leq i+j$ for the second term, and assume for the time-being that $i\geq1$  to ensure that all indices are positive as required in notation \eqref{eq.def-Nell}).
For the second contribution, we use Faa-di Bruno's identity, the commutator and composition estimates \Cref{L.commutator,L.composition} and the trivial continuous embedding $H^{i}(\RR^d)\hookrightarrow H^{i-1}(\RR^d)$:
\begin{align*}
	\norm{ [\partial_t^j,\S_0(V)]\partial_t U }_{H^i}
	&\lesssim \sum_{m=1}^j  \norm{\partial_t^m\big(\S_0(V)\big)(\partial_t^{j+1-m} U)}_{H^i}\\
	&\lesssim  \sum_{m=1}^j \sum_{\ell=1}^m  \sum_{(j_1,\dots,j_\ell)} \norm{\S_0^{(\ell)}(V) \big(\prod_{n=1}^\ell \partial_t^{j_n} V \big)(\partial_t^{j+1-m} U)}_{H^i}\\
	&\lesssim  \sum_{m=1}^j \sum_{\ell=1}^m  \sum_{(j_1,\dots,j_\ell)}\norm{\S_0^{(\ell)}(V)}_{L^\infty} \norm{ \big(\prod_{n=1}^\ell \partial_t^{j_n} V \big)(\partial_t^{j+1-m} U)}_{H^i}\\
	& \qquad +\sum_{m=1}^j \sum_{\ell=1}^m  \sum_{(j_1,\dots,j_\ell)}\norm{[\Lambda^i,\S_0^{(\ell)}(V)] \big(\prod_{n=1}^\ell \partial_t^{j_n} V \big)(\partial_t^{j+1-m} U)}_{L^2}\\
	&\leq  C(\norm{V}_{H^{s_0}})(1+\norm{V}_{H^{s_0+1}}) \sum_{m=1}^j \sum_{\ell=1}^m  \sum_{(j_1,\dots,j_\ell)} \norm{ \big(\prod_{n=1}^\ell \partial_t^{j_n} V \big)(\partial_t^{j+1-m} U)}_{H^i}\\
	&\quad  + C(\norm{V}_{H^{s_0}}) \sum_{m=1}^j \sum_{\ell=1}^m  \sum_{(j_1,\dots,j_\ell)} \Big\langle \norm{V}_{H^{i}}\norm{ \big(\prod_{n=1}^\ell \partial_t^{j_n} V \big)(\partial_t^{j+1-m} U)}_{H^{s_0}}\Big\rangle_{i>s_0+1}
\end{align*}
where the sums are over $(j_1,\dots,j_\ell)$ such that $1\leq j_n\leq j$ and $\sum_{n=1}^\ell j_n=m$. Applying \Cref{L.Product} with $m_j=m_k=1$ for the first contribution and that $H^{s_0}(\RR^d)$ is a Banach algebra for the second yields
\begin{multline*}
		\norm{ [\partial_t^j,\S_0(V)]\partial_t U }_{H^i}
	\leq C(\norm{V}_{H^{s_0}}) (1+\norm{V}_{H^{s_0+1}}) \times
	\sum_{\ell=1}^{j}  N_\ell(V,\dots,V;U) \big\vert_{1,s_0+1}^{j+1-\ell,i+j+1-\ell}\\ +C(\norm{V}_{H^{s_0}})\times	\sum_{\ell=1}^{j+1} \Big\langle  N_\ell(V,\dots,V;U) \big\vert_{0,s_0+1}^{j+1,i+j+1-\ell} \Big\rangle_{i>s_0+1}.
\end{multline*}

The fourth term is estimated similarly: decompose
\[[\partial_t^j\Lambda^i,\S_l(V)]\partial_{x_l} U=[\Lambda^i,\S_l(V)]\partial_t^j\partial_{x_l} U+\Lambda^i\big([\partial_t^j,\S_l(V)]\partial_{x_l} U\big)\]
and observe
\begin{align*}
\norm{[\Lambda^i,\S_l(V)]\partial_t^j\partial_{x_l} U}_{L^2}
	&\leq C(\norm{V}_{H^{s_0}})\,\Big(\norm{ V}_{H^{s_0+1}} \norm{ \partial_t^j U}_{H^{i}}+\big\langle\norm{ V}_{H^{i}} \norm{\partial_t^j U}_{H^{s_0+1}}\big\rangle_{i>s_0+1}\Big)\\
	&\leq C(\norm{V}_{H^{s_0}})\, N_1(V;U)\big\vert_{0,s_0+1}^{j,i+j}.
\end{align*}
Using Faa-di Bruno's identity we obtain
\begin{align*}
\norm{ [\partial_t^j,\S_l(V)]\partial_{x_l} U }_{H^i}
	&\lesssim  C(\norm{V}_{H^{s_0}}) (1+\norm{V}_{H^{s_0+1}}) \sum_{m=1}^j \sum_{\ell=1}^m  \sum_{(j_1,\dots,j_\ell)} \norm{ \big(\prod_{n=1}^\ell \partial_t^{j_n} V \big)(\partial_t^{j-m} \partial_{x_l}U)}_{H^i}\\
&\quad  + C(\norm{V}_{H^{s_0}})  \sum_{m=1}^j \sum_{\ell=1}^m  \sum_{(j_1,\dots,j_\ell)} \Big\langle \norm{V}_{H^{i}}\norm{ \big(\prod_{n=1}^\ell \partial_t^{j_n} V \big)(\partial_t^{j-m}\partial_{x_l} U)}_{H^{s_0}}\Big\rangle_{i>s_0+1}
\end{align*}
where the sums are over $(j_1,\dots,j_\ell)$ such that $1\leq j_n\leq j$ and $\sum_{n=1}^\ell j_n=m$, and hence by \Cref{L.Product} with $m_j=0$ and $m_k=1$ and the fact that $H^{s_0}(\RR^d)$ is a Banach algebra we find
\begin{equation*}
	\norm{ [\partial_t^j,\S_l(V)]\partial_{x_l} U }_{H^i}
	\leq C(\norm{V}_{H^{s_0}}) (1+\norm{V}_{H^{s_0+1}}) \times
	\sum_{\ell=1}^{j+1}  N_\ell(V,\dots,V;U) \big\vert_{0,s_0+1}^{j,i+j+1-\ell}.
\end{equation*}

For the fifth term we write 
\[ \partial_t^j\Lambda^i(\G(V) U-\G(\bm0) U) = [\Lambda^i,\G(V)-\G(\bm0)] \partial_t^j U+\Lambda^i \big( [\partial_t^j,\G(V)] U\big) +(\G(V)-\G(\bm0)) \partial_t^j\Lambda^i  U.\]
Using the trivial continuous embedding $H^{i}(\RR^d)\hookrightarrow H^{i-1}(\RR^d)$ we can estimate the first two contributions exactly as above, and the third contribution brings no difficulty (and is a special case, $\ell=1$) using the composition estimates in \Cref{L.composition} (here we use that the positive-semidefinite constant-coefficient matrix $G(\bm0)$ was subtracted to avoid a contribution that does not vanish when $V=\bm0$ corresponding to $\ell=0$).

Collecting all the above estimates yields by Cauchy-Schwarz inequality
\begin{equation}\label{eq.est-energy} 
	\frac12\frac{\dd}{\dd t} \bra{\S_0(V)\partial_t^j\Lambda^i U,\partial_t^j\Lambda^i U} 
\leq  \norm{\partial_t^j U}_{H^i} \big(C(\norm{V}_{H^{s_0}})(1+\norm{V}_{H^{s_0+1}}) A_{j,i}+\norm{\partial_t^jR}_{H^i}\big) 
\end{equation}
where
\begin{align*}
A_{j,i}&\eqdef   \sum_{\ell=1}^{j+1}  
N_\ell(V,\dots,V;U) \big\vert_{1,s_0+1}^{j+1-\ell,i+j+1-\ell}
+ N_\ell(V,\dots,V;U) \big\vert_{0,s_0+1}^{j,i+j+1-\ell}
+  N_\ell(V,\dots,V;U) \big\vert_{0,s_0+1}^{j+1,i+j+1-\ell} .
\end{align*}
Recall we assumed $i\geq 1$ when estimating $\norm{V}_{H^{s_0+1}}\norm{ \partial_t^{j+1} U}_{H^{i-1}}$. When $i\in[0,1)$ we simply use the continuous embedding $L^2(\RR^d)\hookrightarrow H^{i-1}(\RR^d)$ and add $\norm{V}_{H^{s_0+1}}\norm{ \partial_t^{j+1} U}_{L^2}$ to $A_{j,i}$ in that case. Notice however that terms that led to this contribution and $ N_\ell(V,\dots,V;U) \big\vert_{0,s_0+1}^{j+1,i+j+1-\ell}$ stem respectively from commutators $[\Lambda^i,\S_0(V)] $ or $[\Lambda^i,\S_0^{(\ell)}(V)]$ ($\ell\in \{1,\dots,j\}$) and hence both contributions vanish when $i=0$. Thus we have in that case
\begin{equation}\label{eq.est-energy-0} 
	A_{j,0}\eqdef \sum_{\ell=1}^{j+1} N_\ell(V,\dots,V;U) \big\vert_{1,s_0+1}^{j+1-\ell,j+1-\ell}
	+ N_\ell(V,\dots,V;U) \big\vert_{0,s_0+1}^{j,j+1-\ell}.
\end{equation}

\paragraph{Forward estimates.} We use the equation~\eqref{eq.linearized} on the regular domain of the operator $\L_\delta$. To this aim, and since $\L_\delta$ is skew-Hermitian, we introduce $\P_\delta$ the pseudo-differential projection operator associated with the symbol $\sPi_\delta$ defined in Hypothesis~\ref{H4}. Recall  $\P_\delta^2=\P_\delta$, $\P_\delta$ is self-adjoint and commutes with $\L_\delta$. Let $j\in\NN^\star$, $i\geq 0$ and recall $\Lambda^i=(\Id-\Delta_{\bx})^{i/2}$ commutes with $\P_\delta$ and $\L_\delta$. Applying $\P_\delta\partial_t^{j-1}\Lambda^i$ to~\eqref{eq.linearized} we find after some algebra
\begin{multline*} 	\P_\delta \S_0(V)\P_\delta\partial_t^j\Lambda^i U = -\P_\delta \S_0(V)(\Id-\P_\delta)\partial_t^j\Lambda^i U  - \P_\delta [\partial_t^{j-1}\Lambda^i, \S_0(V)]\partial_t U \\
	- \P_\delta\partial_t^{j-1}\Lambda^i\big( \sum_{l=1}^d\S_l(V)\partial_{x_l}U +\G(V)U+R\big) +\frac1\eps\P_\delta\L_\delta\P_\delta \partial_t^{j-1}\Lambda^i U.
\end{multline*}
Testing against $\P_\delta\partial_t^j\Lambda^i U$ for the $L^2(\RR^d)^n$ inner-product and using Hypothesis~\ref{H2}, we infer that if 
\[\forall(t,\bx)\in [0,T]\times\RR^d, \quad V(t,\bx)\in K\subset \Omega\]
there exists $c_K>0$ such that
\begin{multline*} 
	c_K\norm{\P_\delta\partial_t^j U }_{H^i}\leq 
	\norm{ \S_0(V)(\Id-\P_\delta)\partial_t^j\Lambda^i U}_{L^2}  
	+\norm{ [\partial_t^{j-1}\Lambda^i, \S_0(V)]\partial_t U}_{L^2} \\
	+\sum_{l=1}^d\norm{ \partial_t^{j-1}\big( \S_l(V)\partial_{x_l}U\big) }_{H^i}+\norm{ \partial_t^{j-1}\big( \G(V)U\big) }_{H^i}+\norm{ \partial_t^{j-1}R}_{H^i}+\frac1\eps\norm{\P_\delta\L_\delta\P_\delta \partial_t^{j-1} U}_{H^i} .
\end{multline*}
We now estimate each term on the right-hand side. 
	
	For the first component we simply write
	\[ \norm{ \S_0(V)(\Id-\P_\delta)\partial_t^j\Lambda^i U}_{L^2} \leq C(\norm{V}_{L^\infty}) \norm{(\Id-\P_\delta)\partial_t^j\Lambda^i U}_{L^2}\leq C(\norm{V}_{H^{s_0}})\norm{(\Id-\P_\delta)\partial_t^j U}_{H^i}.\]

	The second term is estimated using Faa-di Bruno's Lemma, the triangle inequality, commutator and composition estimate in \Cref{L.product-2,L.composition}: 
	\begin{align*} \norm{ [\partial_t^{j-1}\Lambda^i, \S_0(V)]\partial_t U}_{L^2}
		&\leq \norm{ [\Lambda^i, \S_0(V)]\partial_t^j U}_{L^2}+\norm{ [\partial_t^{j-1}, \S_0(V)]\partial_t U}_{H^i}\\
		&\leq   C(\norm{V}_{H^{s_0}}) \Big( \norm{V}_{H^{s_0+1}}\norm{\partial_t^j U}_{H^{i-1}}+\big\langle\norm{V}_{H^{i}}\norm{\partial_t^j U}_{H^{s_0}}\big\rangle_{i>s_0+1}\Big)\\
		&\qquad + C(\norm{V}_{H^{s_0}}) (1+\norm{V}_{H^{s_0+1}}) \sum_{\ell=1}^{j-1}    \sum_{(j_0,\dots,j_\ell)} \norm{\big(\prod_{n=1}^\ell \partial_t^{j_n} V \big)(\partial_t^{j_0} U)}_{H^i}\\
		&\qquad + C(\norm{V}_{H^{s_0}})\big\langle \norm{V}_{H^i} \sum_{\ell=1}^{j-1} \sum_{(j_0,\dots,j_\ell)} \norm{\big(\prod_{n=1}^\ell \partial_t^{j_n} V \big)(\partial_t^{j_0} U)}_{H^{s_0}}\big\rangle_{i>s_0+1}
	\end{align*}
	where the sums are over $(j_0,\dots,j_\ell)$ such that $1\leq j_n\leq j-\ell$ and $\sum_{n=0}^\ell j_n = j$.
	Hence by \Cref{L.Product} with $m_j=m_k=1$ and using that $H^{s_0}(\RR^d)$ is a Banach algebra for contributions arising when $i>s_0+1$ we infer
	\begin{multline*} \norm{ [\partial_t^{j-1}\Lambda^i, \S_0(V)]\partial_t U}_{L^2} \leq  C(\norm{V}_{H^{s_0}}) (1+\norm{V}_{H^{s_0+1}}) \times
		 \Big(  \norm{\partial_t^j U}_{H^{i-1}} + \sum_{\ell=1}^{j-1} N_\ell(V,\dots,V;U)\big\vert_{1,s_0+1}^{j-\ell,i+j-\ell} \Big)\\
		 + C(\norm{V}_{H^{s_0}})\times \sum_{\ell=1}^{j}\Big\langle N_\ell(V,\dots,V;U)\big\vert_{0,s_0+1}^{j,i+j-\ell}\Big\rangle_{i>s_0+1}.
		\end{multline*}
	
	Proceeding as above we have
	\begin{align*}  \norm{\partial_t^{j-1}\big(\S_l(V)\partial_{x_l}U\big) }_{H^i}
		&\leq C(\norm{V}_{H^{s_0}}) (1+\norm{V}_{H^{s_0+1}}) \sum_{\ell=0}^{j-1}    \sum_{(j_0,\dots,j_\ell)} \norm{\big(\prod_{n=1}^\ell \partial_t^{j_n} V \big)(\partial_t^{j_0} \partial_{x_l} U)}_{H^i}\\
		&\qquad + C(\norm{V}_{H^{s_0}})\big\langle \norm{V}_{H^i} \sum_{\ell=0}^{j-1} \sum_{(j_0,\dots,j_\ell)} \norm{\big(\prod_{n=1}^\ell \partial_t^{j_n} V \big)(\partial_t^{j_0} \partial_{x_l} U)}_{H^{s_0}}\big\rangle_{i>s_0+1}
	\end{align*}
	where the sums are over $(j_0,\dots,j_\ell)$ such that $0\leq j_0\leq j-\ell-1$, $1\leq j_n\leq j-\ell$ and $\sum_{n=0}^\ell j_n = j-1$, and hence by \Cref{L.Product} with $m_j=0$ and $m_k=1$  we have
		\begin{multline*} \norm{\partial_t^{j-1}\big(\S_l(V)\partial_{x_l}U\big) }_{H^i} \leq  C(\norm{V}_{H^{s_0}}) (1+\norm{V}_{H^{s_0+1}}) \times \\
			 \Big(\norm{\partial_t^{j-1} U}_{H^{i+1}}+\sum_{\ell=1}^{j}  N_\ell(V,\dots,V;U)\big\vert_{0,s_0+1}^{j-1,i+j-\ell}\Big).
	\end{multline*}
	
	The contribution stemming from $\G(V)$ can be estimated exactly as above using the trivial continuous embedding $H^{i_0}(\RR^d)\hookrightarrow H^{i_0-1}(\RR^d)$ for any $i_0\in\RR$, and the contribution stemming from $R$ is straightforward, as is the one from $\P_\delta\L_\delta\P_\delta \partial_t^{j-1} U$, using Hypothesis~\ref{H4} (see \Cref{R.H3H4}).

	Collecting the above we find
	\begin{multline}\label{eq.est-forward}
		\norm{\P_\delta\partial_t^j U }_{H^i} \leq c_K^{-1} C(\norm{V}_{H^{s_0}})(1+\norm{V}_{H^{s_0+1}})\times \Big(\norm{(\Id-\P_\delta)\partial_t^j U}_{H^i}+\norm{\partial_t^{j-1} R}_{H^i}	+(1+C_\L)\norm{\partial_t^{j-1} U}_{H^{i+1}}\\
	+\norm{\partial_t^{j} U}_{H^{i-1}}+ \sum_{\ell=1}^{j} N_\ell(V,\dots,V;U)\big\vert_{1,s_0+1}^{j-\ell,i+j-\ell} + N_\ell(V,\dots,V;U)\big\vert_{0,s_0+1}^{j-1,i+j-\ell}+ \Big\langle N_\ell(V,\dots,V;U)\big\vert_{0,s_0+1}^{j,i+j-\ell}\Big\rangle_{i>s_0+1}\Big).
	\end{multline}

\paragraph{Backward estimates.}
We now use the equation~\eqref{eq.linearized} on the singular domain of the operator $\L_\delta$.
Let $j\in\NN$, $i\geq 1$ and recall $\Lambda^i=(\Id-\Delta_{\bx})^{i/2}$ and $\P_\delta$ is the pseudo-differential projection operator associated with the symbol $\sPi_\delta$ defined in Hypothesis~\ref{H4}. One has  $(\Id-\P_\delta)^2=\Id-\P_\delta$ commutes with $\L_\delta$, $\Lambda^{i-1}$ and time differentiation. Applying $(\Id-\P_\delta)\Lambda^{i-1}\partial_t^j$ to the equation~\eqref{eq.linearized} we find
\[ (\Id-\P_\delta)\frac1\eps\L_\delta (\Id-\P_\delta)\partial_t^j\Lambda^{i-1}U = (\Id-\P_\delta)\partial_t^j\Lambda^{i-1} \big(  \S_0(V)\partial_t U + \sum_{l=1}^d\S_l(V)\partial_{x_l}U +\G(V)U+R\big).\]
Using the hypothesis~\ref{H4} (see \Cref{R.H3H4})
and the triangle inequality we infer that
\begin{multline*}\frac1\eps\norm{(\Id-\P_\delta)\partial_t^j\Lambda^{i-1} U}_{L^2}+\frac\delta\eps \norm{(\Id-\P_\delta)\partial_t^j\Lambda^{i-1} U}_{H^1}
	\\
	\leq c_\L^{-1}\,\Big(\norm{\partial_t^j\Lambda^{i-1} \big(\S_0(V)\partial_t U\big)}_{L^2} +\sum_{l=1}^d\norm{\partial_t^j\big(\S_l(V)\partial_{x_l}U\big) }_{H^{i-1}}+
	\norm{\partial_t^j \big(\G(V)U\big)}_{H^i}+\norm{\partial_t^j R}_{H^{i-1}}\Big).
\end{multline*}
We now estimate each term on the right-hand side. 

Decomposing 
\[ \partial_t^j \Lambda^{i-1} \big(\S_0(V)\partial_t U\big) = \S_0(V)\partial_t^{j+1} \Lambda^{i-1} U+[\partial_t^j\Lambda^{i-1}, \S_0(V)]\partial_t U\]
and using a preceding estimate for the second term (with the change of indices $j$ to $j-1$ and $i-1$ to $i$) we find that
\begin{multline*}
	\norm{\partial_t^j\Lambda^{i-1} \big(\S_0(V)\partial_t U\big)}_{L^2}
	\leq C(\norm{V}_{H^{s_0}}) \norm{\partial_t^{j+1} U}_{H^{i-1}} +C(\norm{V}_{H^{s_0}}) (1+\norm{V}_{H^{s_0+1}}) \times \\
	\Big( \sum_{\ell=1}^{j+1} N_\ell(V,\dots,V;U)\big\vert_{1,s_0+1}^{j+1-\ell,i+j-\ell} +  \Big\langle N_\ell(V,\dots,V;U)\big\vert_{0,s_0+1}^{j+1,i+j-\ell}\Big\rangle_{i>s_0+1}\Big).
	\end{multline*}
In the same way, we readily infer from a preceding estimate that
	\begin{multline*}  \norm{\partial_t^j\big(\S_l(V)\partial_{x_l}U\big) }_{H^{i-1}}+\norm{\partial_t^j\big(\G(V)U\big) }_{H^{i-1}}	\\
	\leq C(\norm{V}_{H^{s_0}}) (1+\norm{V}_{H^{s_0+1}}) \Big( \norm{\partial_t^{j} U}_{H^{i}}+ \sum_{\ell=1}^{j+1}  N_\ell(V,\dots,V;U)\big\vert_{0,s_0+1}^{j,i+j-\ell}\Big).
	\end{multline*}

Combining the above we find
\begin{multline}\label{eq.est-backward}
	\delta^{-1}\norm{(\Id-\P_\delta)\partial_t^j U}_{H^{i-1}}+\norm{(\Id-\P_\delta)\partial_t^j U}_{H^{i}} \\ 
	\leq \frac\eps\delta c_\L^{-1} C(\norm{V}_{H^{s_0}})(1+\norm{V}_{H^{s_0+1}})\times \Big(\norm{\partial_t^{j+1} U}_{H^{i-1}}+\norm{\partial_t^{j} U}_{H^{i}}+\norm{\partial_t^j R}_{H^{i-1}}\\
	+\sum_{\ell=1}^{j+1}N_\ell(V,\dots,V;U)\big\vert_{1,s_0+1}^{j+1-\ell,i+j-\ell} + N_\ell(V,\dots,V;U)\big\vert_{0,s_0+1}^{j,i+j-\ell}\\+ \Big\langle N_\ell(V,\dots,V;U)\big\vert_{0,s_0+1}^{j+1,i+j-\ell}\Big\rangle_{i>s_0+1}  \Big).
\end{multline}

When $0\leq i \leq 1$ we apply the interpolation inequality
\[ \norm{(\Id-\P_\delta)\partial_t^j U}_{H^{i}} \leq \norm{(\Id-\P_\delta)\partial_t^j U}_{L^2}^{1-i}\norm{(\Id-\P_\delta)\partial_t^j U}_{H^{1}}^{i} \leq \norm{(\Id-\P_\delta)\partial_t^j U}_{L^2}^{1-i}\norm{(\Id-\P_\delta)\partial_t^j U}_{H^{1}}^{i},\]
and infer from the above estimate when $i=1$ that
\begin{multline}\label{eq.est-backward-0}
	\norm{(\Id-\P_\delta)\partial_t^j U}_{H^{i}} \leq  \frac\eps{\delta^i} c_\L^{-1}  C(\norm{V}_{H^{s_0}})(1+\norm{V}_{H^{s_0+1}})\times  \Big(\norm{\partial_t^{j+1} U}_{L^2}+\norm{\partial_t^{j} U}_{H^{1}}+\norm{\partial_t^j R}_{L^2}\\
	+\sum_{\ell=1}^{j+1} N_\ell(V,\dots,V;U)\big\vert_{1,s_0+1}^{j+1-\ell,j+1-\ell} + N_\ell(V,\dots,V;U)\big\vert_{0,s_0+1}^{j,j+1-\ell}\Big).
\end{multline}

\subsection{Application to stability estimates}\label{S.bootstrap}

We now apply the results of the previous section to bootstrap the control of suitable energy functionals for sufficiently regular solutions to~\eqref{eq.linearized}, under the assumption that the set of weights $\bm\alpha=( \alpha_{j,i})$ is $\lambda$-admissible in the sense of \Cref{D.alpha-admissible} {\em i.e.} satisfies relations~\ref{ii}--\ref{vi}.
The parameter $\lambda$ appearing therein will be set (sufficiently large) to close a finite-induction bootstrap argument concluding the proof of \Cref{P.Stability}. 

Notice $\alpha_{j,i}$ in \Cref{D.alpha-admissible} is defined also for negative values of the index $i$ because such values are used in intermediary steps of the proof, although only nonnegative indices arise in the definitions of our functionals $\cF_{k,{\bm\alpha}}$ and $ \cF_{k,{\bm\alpha}}^{(0)}$ defined in~\eqref{eq.def-F-stability} and~\eqref{eq.def-F0-stability}, and that we recall below:
\[
	\cF_{k,{\bm\alpha}}(U)\eqdef\sup\big(\big\{ \alpha_{j,i}^{-1} \norm{\partial_t^j U}_{H^{i}} \ : \ (i,j)\in\RR_+\times\NN,\ i+j\leq k\big\}\big)
\]
and
\[ \cF_{k,{\bm\alpha}}^{(0)}(U)\eqdef\sup\big( \big\{\alpha_{j,i}^{-1} \bra{\S_0(V)\partial_t^j\Lambda^i U,\partial_t^j\Lambda^i U}^{1/2}\ : \ (i,j)\in \omega_{k,{\bm\alpha}}^{(0)}\big\}\big)
\]
where 
\begin{align*}
	\omega_{k,{\bm\alpha}}^{(0)}
	\eqdef \big\{(i,j) \ : \ j\in\{0,\dots,j_0-1\},\ i\in [0,k-j]\big\} \bigcup  \big\{(0,j) \ : \ j\in \{0,1,\dots,k\}\big\}.
\end{align*}

		Let us introduce the following convenient notations:
		\[M_U=\cF_{k,{\bm\alpha}}(U), \ M_U^{(0)}=\cF_{k,{\bm\alpha}}^{(0)}(U), \ M_V=\cF_{k,{\bm\alpha}}(V), \ M_R=\cF_{k,{\bm\alpha}}(R).\]
		We denote by $C^{(0)}$ a positive number depending on $s_0>d/2$, $k\geq s_0+1$ and non-decreasingly on $C_{\L}>0$, $c_\L^{-1}>0$, $c_K^{-1}>0$  and $\norm{V}_{H^{s_0+1}}$,
		and whose value changes from line to line. We denote by $C$ a positive number which additionally depends non-decreasingly on $M_V$. 
		Notice that by assumption~\ref{ii} we have in particular $\alpha_{0,s_0+1}\leq1$ and hence $\norm{V}_{H^{s_0+1}}\leq M_V$, and $C^{(0)}\leq C$.

		We recall the notation
\begin{multline} \label{eq.def-Nell-recall}
		N_\ell(\underbrace{V,\dots,V}_{\text{$\ell$ times}};U)\big\vert_{j_\star,k_\star}^{j^\star,k^\star}\\
		\eqdef \sup\Big( \Big\{ \big(\prod_{n=1}^\ell \norm{ \partial_t^{j_n} V}_{H^{i_n}} \big)\norm{ \partial_t^{j_0} U}_{H^{i_0}} \ : \   ((i_n,j_n))_n\in(\RR_+\times\NN)^{1+\ell} \text{ satisfies } \\
		\sum_{n=0}^\ell j_n = j^\star+\ell j_\star, \quad \sum_{n=0}^\ell i_n+j_n = k^\star+\ell k_\star,\\
		\text{ and for all $ n\in\{0,\dots,\ell\}$, }  (j_n,i_n+j_n)\in[(j_\star,k_\star),(j^\star,k^\star)]
		 \Big\}\Big)
	\end{multline}
	where $	 (j_n,i_n+j_n)\in[(j_\star,k_\star),(j^\star,k^\star)]$ if and only if
	\[	\begin{cases}
		(j_n,i_n+j_n)\in[j_\star,j^\star]\times[k_\star,k^\star]  & \text{ if } k_\star \leq k^\star,\\
		(j_n,i_n+j_n)=\theta_n (j_\star,k_\star)+(1-\theta_n)(j^\star,k^\star) ,\ \theta_n=\frac{(i_n+j_n)-k^\star}{k_\star-k^\star} & \text{ if } k_\star > k^\star.
	\end{cases}\]
		Notice the important consequence of assumption~\ref{vi} that
		\begin{equation}\label{eq.key}
			N_\ell(V,\dots,V;U)\big\vert_{j_\star,k_\star}^{j^\star,k^\star} \leq  \big(\alpha_{j_\star,k_\star-j_\star}  M_V\big)^\ell  \big(\alpha_{j^\star,k^\star-j^\star} M_U\big).
		\end{equation}
		
		\medskip
		
		\paragraph{Control of $ \cF_{k,{\bm\alpha}}^{(0)}(U)$.} By the energy estimate~\eqref{eq.est-energy} we have for any $j\in\NN$ and $i\geq 1$
		\begin{equation}\label{eq.est-energy-2} 
			\frac12\alpha_{j,i}^{-2}\frac{\dd}{\dd t} \bra{\S_0(V)\partial_t^j\Lambda^i U,\partial_t^j\Lambda^i U} 
			\leq \big(\alpha_{j,i}^{-1} \norm{\partial_t^j\Lambda^i U}_{L^2}\big)\times\big(C^{(0)}\alpha_{j,i}^{-1}A_{j,i}+M_R\big) 
		\end{equation}
		where
		\begin{align*}
			A_{j,i}&\eqdef  \sum_{\ell=1}^{j+1}  
			N_\ell(V,\dots,V;U) \big\vert_{1,s_0+1}^{j+1-\ell,i+j+1-\ell}
			+ N_\ell(V,\dots,V;U) \big\vert_{0,s_0+1}^{j,i+j+1-\ell}
			+  N_\ell(V,\dots,V;U) \big\vert_{0,s_0+1}^{j+1,i+j+1-\ell} .
		\end{align*}
		By~\eqref{eq.key} stemming from assumption~\ref{vi} and then~\ref{ii}--\ref{iii} we have 
		\begin{align*}
			A_{j,i}&\leq \sum_{\ell=1}^{j+1} \Big( \alpha_{1,s_0}^\ell\alpha_{j+1-\ell,i}+\alpha_{0,s_0+1}^\ell\alpha_{j,i+1-\ell}+\alpha_{0,s_0+1}^\ell\alpha_{j+1,i-\ell}\Big) M_V^\ell M_U\\
			&\leq C \Big( \alpha_{j,i}+\alpha_{j+1,i-1}\Big)  M_V\,M_U.
		\end{align*}
		Applying this to $j\in\{0,\dots,j_0-1\}$
		we infer from the first assumption in~\ref{iv} that
		\[ 	A_{j,i} \leq C \, (1+\lambda)\alpha_{j,i} \ M_V\, M_U.\]
		When $i\in(0,1)$, one has to add to $A_{j,i}$ the harmless contribution $\norm{V}_{H^{s_0+1}}\norm{ \partial_t^{j+1} U}_{L^2}\leq \alpha_{j+1,0} M_V\,M_U$ and the above estimate still holds by the second assumption in~\ref{iv} and \ref{iii}.
		
		When $i=0$ we have the sharper estimate~\eqref{eq.est-energy-0}, that is~\eqref{eq.est-energy-2} holds with
		\begin{align*} 	A_{j,0} &\eqdef  \sum_{\ell=1}^{j+1} N_\ell(V,\dots,V;U) \big\vert_{1,s_0+1}^{j+1-\ell,j+1-\ell}
		+ N_\ell(V,\dots,V;U) \big\vert_{0,s_0+1}^{j,j+1-\ell}\\
		&\leq \sum_{\ell=1}^{j+1} \Big( \alpha_{1,s_0}^\ell\alpha_{j+1-\ell,0}+\alpha_{0,s_0+1}^\ell\alpha_{j,1-\ell}\Big) M_V^\ell M_U
		\leq C \alpha_{j,0}\, M_V\,  M_U.
	\end{align*}
	where we used assumptions~\ref{vi}, and then~\ref{ii}--\ref{iii}.
	
		Plugging these bounds in~\eqref{eq.est-energy-2} and using the definition of $\cF_{k,{\bm\alpha}}^{(0)}(U)$ and $\omega_{k,{\bm\alpha}}^{(0)}$ we obtain
		\begin{equation}\label{eq.control-F0}
			\frac{\dd}{\dd t}  \cF_{k,{\bm\alpha}}^{(0)}(U)  \leq C\, (1+\lambda)\, \cF_{k,{\bm\alpha}}(V)  \cF_{k,{\bm\alpha}}(U) +C\, \cF_{k,{\bm\alpha}}(R).
		\end{equation}
			
		\medskip

		\paragraph{Control of $ \cF_{k,{\bm\alpha}}(U)$.} There now remains to prove that
		\[\cF_{k,{\bm\alpha}}(U)\leq  C^{(0)}\, \cF_{k,{\bm\alpha}}^{(0)}(U) +  C^{(0)}\, \cF_{k,{\bm\alpha}}(R),\]
		by application of the forward and backward estimates of the previous section. Notice first that
		\begin{equation}\label{eq.trivial}
			\forall (i,j)\in \omega_{k,{\bm\alpha}}^{(0)} , \quad \alpha_{j,i}^{-1} \norm{\partial_t^j U }_{H^i}\leq C^{(0)}  \cF_{k,{\bm\alpha}}^{(0)}(U) 
		\end{equation}
		since $V(\cdot)\in K$ where $K$ is a compact subset of the hyperbolicity domain defined in Hypothesis~\ref{H2}. 
		Hence we focus below on the complement set $ (i,j)\in\RR_+\times\NN$ such that $i+j\leq k$ but $(i,j)\notin\omega_{k,{\bm\alpha}}^{(0)}$ and in particular $1\leq j_0\leq j\leq k-1$.
		By combining the estimates~\eqref{eq.est-forward} and~\eqref{eq.est-backward}\footnote{We also use the interpolation inequality
			\[\norm{\partial_t^j U}_{H^{i-1}}\leq \norm{\partial_t^j U}_{H^{i}}^{\frac{i-1}{i}}\norm{\partial_t^j U}_{L^2}^{\frac1i}\leq  \frac{i-1}{i} \iota \norm{\partial_t^j U}_{H^{i}} + \frac1i \iota^{1-i}\norm{\partial_t^j U}_{L^2}. \]
			Setting $\iota=\frac1{2C^{(0)}}$ to absorb the contribution $\norm{\partial_t^j U}_{H^{i}}$ and then augmenting $C^{(0)}$ allows to replace $\norm{\partial_t^j U}_{H^{i-1}}$ with $\norm{\partial_t^j U}_{L^2}$ in the right-hand side.} we infer when $i\geq 1$
		\begin{multline}\label{eq.est-equation}
			\norm{\partial_t^j U }_{H^i} \leq C^{(0)} \times \Big(\frac\eps\delta\norm{\partial_t^{j+1} U}_{H^{i-1}}+\frac\eps\delta\norm{\partial_t^{j} U}_{H^{i}}+\norm{\partial_t^{j-1} U}_{H^{i+1}}+\norm{\partial_t^j U}_{L^2}\\
			+\frac\eps\delta\norm{\partial_t^j R}_{H^{i-1}}+\norm{\partial_t^{j-1} R}_{H^i}+B_{j,i}\Big)
			\end{multline}
			where
			\begin{align*} B_{j,i} &\eqdef 	
			 \sum_{\ell=1}^{j} \Big(N_\ell(V,\dots,V;U)\big\vert_{1,s_0+1}^{j-\ell,i+j-\ell} 
			 +  N_\ell(V,\dots,V;U)\big\vert_{0,s_0+1}^{j-1,i+j-\ell}
			 + \Big\langle N_\ell(V,\dots,V;U)\big\vert_{0,s_0+1}^{j,i+j-\ell}\Big\rangle_{i>s_0+1}\Big)\\
			&\quad+	\frac\eps\delta\sum_{\ell=1}^{j+1}\Big(N_\ell(V,\dots,V;U)\big\vert_{1,s_0+1}^{j+1-\ell,i+j-\ell} 
				+  N_\ell(V,\dots,V;U)\big\vert_{0,s_0+1}^{j,i+j-\ell}
				 + \Big\langle N_\ell(V,\dots,V;U)\big\vert_{0,s_0+1}^{j+1,i+j-\ell}\Big\rangle_{i>s_0+1} \Big).
			\end{align*}
		
		When $0\leq i\leq 1$ we can (and must) use~\eqref{eq.est-backward-0} insead of~\eqref{eq.est-backward} to replace the estimate~\eqref{eq.est-equation} with
				\begin{multline}\label{eq.est-equation-0}
			\norm{\partial_t^j U }_{H^i} \leq C^{(0)} \times \Big(\frac\eps{\delta^i}\norm{\partial_t^{j+1} U}_{L^2}+\frac\eps{\delta^i}\norm{\partial_t^{j} U}_{H^{1}}+\norm{\partial_t^{j-1} U}_{H^{i+1}}+\norm{\partial_t^j U}_{L^2}\\
			+\frac\eps{\delta^i}\norm{\partial_t^j R}_{L^2}+\norm{\partial_t^{j-1} R}_{H^i}+B_{j,i}\Big)
		\end{multline}
		where
		\begin{align*} B_{j,i} &\eqdef 	
			\sum_{\ell=1}^{j} \Big(N_\ell(V,\dots,V;U)\big\vert_{1,s_0+1}^{j-\ell,i+j-\ell} 
			+  N_\ell(V,\dots,V;U)\big\vert_{0,s_0+1}^{j-1,i+j-\ell}\Big)\\
			&\quad+	\frac\eps{\delta^i}\sum_{\ell=1}^{j+1}\Big(N_\ell(V,\dots,V;U)\big\vert_{1,s_0+1}^{j+1-\ell,i+j-\ell} 
			+  N_\ell(V,\dots,V;U)\big\vert_{0,s_0+1}^{j,i+j-\ell} \Big).
		\end{align*}

We shall now use estimates~\eqref{eq.est-equation} and~\eqref{eq.est-equation-0}  to infer from the control of~\eqref{eq.trivial} and finite induction the corresponding control of  $\alpha_{j,i}^{-1}\norm{\partial_t^j U }_{H^i}$,  for $ (i,j)\in\RR_+\times\NN$ such that $i+j\leq k$ but $(i,j)\notin\omega_{k,{\bm\alpha}}^{(0)}$.  Let us consider $i> 0$ and $j\in\{j_0+1,\dots,k-1\}$.
 Denote
\[m_{j,i}\eqdef \alpha_{j,i}^{-1}\norm{\partial_t^j U }_{H^i}.\]
 When $i\geq 1$, we obtain from~\eqref{eq.est-equation}, using~\eqref{eq.key} stemming from assumption~\ref{vi} and then~\ref{ii}--\ref{iii}, that $m_{j,0}\leq C^{(0)} M_U^{(0)} $ by~\eqref{eq.trivial} since $(j,0)\in \omega_{k,{\bm\alpha}}^{(0)}$ and that $0<\eps\leq \delta$:
 \begin{multline*}
 	m_{j,i} \leq C^{(0)} \times \Big(\frac\eps\delta \frac{\alpha_{j+1,i-1}}{\alpha_{j,i}}m_{j+1,i-1}+\frac\eps\delta m_{j,i}+\frac{\alpha_{j-1,i+1}}{\alpha_{j,i}}m_{j-1,i+1}\\
 	+M_U^{(0)}+M_R
 	+C\, \big( \frac{\alpha_{j-1,i}}{\alpha_{j,i}}  +\big\langle\frac{\alpha_{j,i-1}}{\alpha_{j,i}}\big\rangle_{i>s_0+1}+\frac\eps\delta \frac{\alpha_{j,i-1}}{\alpha_{j,i}} +\frac\eps\delta\big\langle\frac{\alpha_{j+1,i-2}}{\alpha_{j,i}} \big\rangle_{i>s_0+1} \big) M_U \Big).
 \end{multline*}
 If $0\leq i\leq 1$ using instead the estimate~\eqref{eq.est-equation-0} and again that  $m_{j+1,0}\leq C^{(0)} M_U^{(0)} $ yields
 \begin{multline*}
 	m_{j,i} \leq C^{(0)} \times \Big(\frac\eps{\delta^i} \frac{\alpha_{j+1,0}}{\alpha_{j,i}}M_U^{(0)}+\frac\eps{\delta^i}\frac{\alpha_{j,1}}{\alpha_{j,i}}m_{j,1}+\frac{\alpha_{j-1,i+1}}{\alpha_{j,i}}m_{j-1,i+1}\\
 	+M_U^{(0)}+M_R
 	+C\, \big(\frac{\alpha_{j-1,i}}{\alpha_{j,i}}+\frac\eps{\delta^i}  \frac{\alpha_{j,i-1}}{\alpha_{j,i}} \big) M_U \Big).
 \end{multline*}

We deduce from assumptions \ref{iii},~\ref{v} and~\ref{l} that for any $j\in\{j_0+1,\dots,k-1\}$ and $i\geq 0$, one has
\begin{equation} \label{eq.induction}
	m_{j,i} \leq  C^{(0)} \times \Big( \big\langle m_{j+1,i-1}\big\rangle_{i> 1}+ (\frac\eps\delta+\frac1\lambda)(1+C)M_U+M_U^{(0)}
+M_R\Big).
\end{equation}
The case $j=j_0$ requires a slightly different analysis because the first assumption in \ref{l} cannot be use to estimate the contributions $\frac{\alpha_{j-1,i+1}}{\alpha_{j,i}}m_{j-1,i+1}$ and $\frac{\alpha_{j-1,i}}{\alpha_{j,i}}M_U$. However since $j=j_0$ we may replace them with $\frac{\alpha_{j-1,i+1}}{\alpha_{j,i}}M_U^{(0)}$, respectively $\frac{\alpha_{j-1,i}}{\alpha_{j,i}}M_U^{(0)}$, and by assumpion~\ref{iii} we immediately see that~\eqref{eq.induction} also holds when $j=j_0$.

Using~\eqref{eq.trivial} for $j\leq j_0-1$ or $i=0$ and~\eqref{eq.induction} and finite induction on decreasing $j\in\{k-1,\dots,j_0\}$, and recalling the definitions of $M_U=\cF_{k,{\bm\alpha}}(U)$, $M_U^{(0)}=\cF_{k,{\bm\alpha}}^{(0)}(U)$ and $ M_R=\cF_{k,{\bm\alpha}}(R)$ we find that
\[
	\cF_{k,{\bm\alpha}}(U) \ \leq\  C^{(0)} \, \cF_{k,{\bm\alpha}}^{(0)}(U)\ +\ C^{(0)}\, \cF_{k,{\bm\alpha}}(R)\ + \  C \,  (\frac\eps\delta+\frac1\lambda)\, \cF_{k,{\bm\alpha}}(U).
\]

Assuming that $\frac\eps\delta+\frac1\lambda$ is sufficiently small such that $ C\, (\frac\eps\delta+\frac1\lambda)\leq \frac12$, we can absorb the corresponding contribution and deduce
\begin{equation}\label{eq.control-F}
	\cF_{k,{\bm\alpha}}(U)\leq C^{(0)} \, \cF_{k,{\bm\alpha}}^{(0)}(U)+C^{(0)}\, \cF_{k,{\bm\alpha}}(R).
\end{equation}

	We are now ready to conclude the proof of our main result, \Cref{P.Stability}. 
		Plugging the estimate~\eqref{eq.control-F} in \cref{eq.control-F0} and using Grönwall's Lemma we find 
		\[\cF_{k,{\bm\alpha}}^{(0)}(U)\big\vert_{t=T}\leq \cF_{k,{\bm\alpha}}^{(0)}(U)\big\vert_{t=0} e^{C \, (1+\lambda)\, V(0,T)}+C \int_0^T \cF_{k,{\bm\alpha}}(R)  e^{C \, (1+\lambda)\, V(t,T)}\dd t\]
		where $V(t_1,t_2)\eqdef\int_{t_1}^{t_2} \cF_{k,{\bm\alpha}}(V)\dd t$ and $C$ depends uniquely on $k,c_K,C_{\L},c_\L$ and $\sup_{t\in[0,T]}\cF_{k,{\bm\alpha}}(V)$.
		This immediately yields the desired estimates 	and the proof of \Cref{P.Stability} is complete.

\section{Large-time well-posedness; proof of \Cref{T.Well-posedness}}\label{S.completion}
	
	We now complete the proof of \Cref{T.Well-posedness}, based on \Cref{P.Stability}. We first provide a large-time well-posedness result based on the propagation in time of the full space-time functional  $\cF_{k,{\bm\alpha}}$ is defined in \cref{eq.def-F}. 
	
	\begin{Proposition}\label{P.Well-posedness}
		Let $k\in\NN,\ k>d/2+1$. For any $(\eps,\delta)\in\cS$ and any $U_0\in H^k(\RR^d)$ satisfying the hyperbolicity condition $U_0(\RR^d)\subset K\subset\Omega$, there exists a unique $U\in \cC(I_{\eps,\delta};H^k(\RR^d))$ maximal-in-time classical solution to system~\eqref{eq.general} under Hypotheses~\ref{H1}--\ref{H3} emerging from the initial data $U\big\vert_{t=0}=U_0$, and one has $U\in \cap_{j=0}^k\cC^j(I_{\eps,\delta};H^{k-j}(\RR^d))$. 
		
		Assume additionally that hypothesis~\ref{H4} holds. For any $C_0>0$ and $s_0>d/2$ such that $s_0\leq 1+k$, there exists $T>0$, $C>0$ and $\lambda_0\geq 1$ depending uniquely on $k,s_0,c_K,C_\L,c_\L$ and $C_0$ such that the following holds. 
		
		Assume that $(\eps,\delta)\in\cS$ and $\frac\delta\eps\geq \lambda_0$ and
		\[ \cF_{k,{\bm\alpha}}(U) \big\vert_{t=0} \leq C_0, \quad \ \cF_{k,{\bm\alpha}} \eqdef\sup\big(\big\{ \alpha_{j,i}^{-1} \norm{\partial_t^j U}_{H^{i}} \ : \ (i,j)\in\RR_+\times\NN,\ i+j\leq k\big\}\big)\]
		with the set of weights $\bm\alpha=\big(\alpha_{j,i} \ : \  (i,j)\in\RR_+\times\NN,\ i+j\leq k\big)$ being $\lambda$-admissible in the sense of \Cref{D.alpha-admissible} with $\lambda\geq \lambda_0$. Then $I_{\eps,\delta}\supset [-T/(C_0(1+\lambda)),T/(C_0(1+\lambda))]$ and for any $t\in [-T/(C_0(1+\lambda)),T/(C_0(1+\lambda))]$, one has $\{U(t,\bx) \ :\  \bx\in\RR^d\}\subset \Omega$ and
		\[ \cF_{k,{\bm\alpha}}(U)(t) \leq C\, \cF_{k,{\bm\alpha}}(U) \big\vert_{t=0}
		.\]
	\end{Proposition}
	\begin{proof}
		The existence and uniqueness of a maximal-in-time classical solution emerging from initial data $U_0\in H^k(\RR^d)$ with $k>d/2+1$ is a standard result for  Friedrichs symmetrizable quasilinear systems (see {\em e.g.}~\cite{Benzoni-GavageSerre07}). 
		Assuming first that $U_0\in H^{k+1}(\RR^d)$, the maximal-in-time classical solution to system~\eqref{eq.general} emerging from the initial data $U\big\vert_{t=0}=U_0$ satisfies $U\in \cC(I_{\eps,\delta};H^{k+1}(\RR^d))$ and hence (repeatedly using \cref{eq.general} and product and composition estimates in \Cref{L.product-2,L.composition}) $U\in \cap_{j=0}^{k+1}\cC^j(I_{\eps,\delta};H^{k+1-j}(\RR^d))$. This regularity allows to apply \Cref{P.Stability} with $V=U$ and $R=0$. Anticipating that the desired result will hold assuming $U_0\in H^{k+1}(\RR^d)$, we infer the result without this additional restriction by mollifying an initial data $U_0\in H^{k}(\RR^d)$ with a sequence $U_{0,n}\in H^{k+1}(\RR^d)$ such that $U_{0,n}\to U_0$ in $H^{k}(\RR^d)$ and using the standard result that the maximal existence time (respectively the emerging solution in $\cC(I_{\eps,\delta};H^{k}(\RR^d))$) is a lower semi-continuous (respectively continuous) function of the initial data in $H^k(\RR^d)$.
		
		Let us set a compact  $K'\subset\Omega $ such that $K$ is compactly embedded in $K'$, and denote by $I_1\subset  I_{\eps,\delta}\cap \RR_+$ the interval of all $T\geq0$ such that
		\[\forall (\tau,\bx)\in[0,T]\times\RR^d, \quad U(\tau,\bx)\in  K'\subset\Omega\]
		and $I_2\subset  I_{\eps,\delta}\cap \RR_+$ the interval of all $T\geq0$ such that
		\[\forall \tau\in[0,T] , \quad \cF_{k,{\bm\alpha}}^{(0)}(U)\big\vert_{t=\tau}\leq 2\cF_{k,{\bm\alpha}}^{(0)}(U)\big\vert_{t=0},\]
		where $\cF_{k,{\bm\alpha}}^{(0)}(U)$ is defined through~\eqref{eq.def-F0-stability} with $V=U$.
		Using the continuous embedding $L^\infty(\RR^d)  \hookrightarrow H^{k+1}(\RR^d)$ and since $U\in \cap_{j=0}^{k+1}\cC^j(I_{\eps,\delta};H^{k+1-j}(\RR^d))$, $I_1\cap I_2\neq \emptyset$ and $T^\star\eqdef\sup (I_1\cap I_2)>0$. 
		
		Since the set of weights $\bm{\alpha}$ is $\lambda$-admissible, the stability estimates in \Cref{P.Stability} hold and we have for any $0\leq T<T^\star$
		\[	\sup_{\tau\in[0,T]}\cF_{k,{\bm\alpha}}(U)\big\vert_{t=\tau} \leq  C^{(0)}\, \sup_{\tau\in[0,T]} \cF_{k,{\bm\alpha}}^{(0)}(U)\big\vert_{t=\tau}\leq 2C^{(0)}\cF_{k,{\bm\alpha}}^{(0)}(U)\big\vert_{t=0},\]
		where $C^{(0)}$ depends uniquely on $k,c_{K'},C_{\L}$, $c_\L$ and $\cF_{k,{\bm\alpha}}^{(0)}(U)\leq 2\cF_{k,{\bm\alpha}}^{(0)}(U)\big\vert_{t=0}$, and hence
		\[ 
			\sup_{\tau\in[0,T]}	\cF_{k,{\bm\alpha}}^{(0)}(U)\big\vert_{t=\tau} \leq  \cF_{k,{\bm\alpha}}^{(0)}(U) \big\vert_{t=0}  \exp\big(C^{(1)}\, (1+\lambda)\,T\, \sup_{\tau\in[0,T]}\cF_{k,{\bm\alpha}}(U)\big\vert_{t=\tau}\big) \]
		where $C^{(1)}$ depends uniquely on $k,c_{K'},C_{\L},c_\L$ and $\cF_{k,{\bm\alpha}}(U) \leq 2C^{(0)}\cF_{k,{\bm\alpha}}^{(0)}(U)\big\vert_{t=0}$.
		Now, denoting 
		\[T_0=\big(2(1+\lambda) C^{(0)}C^{(1)} \cF_{k,{\bm\alpha}}^{(0)}(U)\big\vert_{t=0}\big)^{-1}\ln(\frac32),\]
		we find that for any $0\leq T<\min(\{T_0,T^\star\})$
		\[\sup_{\tau\in[0,T]}	\cF_{k,{\bm\alpha}}^{(0)}(U)\big\vert_{t=\tau} \leq \frac32  \cF_{k,{\bm\alpha}}^{(0)}(U) \big\vert_{t=0}.\]
		Moreover, using that for all $(t,\bx)\in[0,T]\times\RR^d$ one has
		\[ \big\vert U(t,\bx)-U(0,\bx) \big\vert\leq \int_0^t \norm{\partial_t U(\tau,\cdot)}_{L^\infty}\dd\tau\lesssim T \sup_{\tau\in[0,T]}\cF_{k,{\bm\alpha}}(U)\big\vert_{t=\tau} \leq  2T_0C^{(0)}\cF_{k,{\bm\alpha}}^{(0)}(U)\big\vert_{t=0}\]
		we have ---restricting $T_0$ if necessary--- for any $0\leq T<\min(\{T_0,T^\star\})$
		\[\forall (\tau,\bx)\in[0,T]\times\RR^d, \quad U(\tau,\bx)\in  \Omega\]
		where $\Omega$ is an open set such that $K\subset \Omega\subset K'$.
		By using again the continuity in time of $U$, this shows by the usual continuity argument that $[0,T_0]\cap I_1\cap I_2$ is an open, closed and non-empty subset of the interval $[0,T_0]$; hence $I_1\cap I_2\cap[0,T_0] =[0,T_0]$ and $T_0\leq T^\star<\sup I_{\eps,\delta}$. Of course, the corresponding result holds for negative times.
		
		This proves \Cref{P.Well-posedness} after noticing that since
		\[\forall \bx\in[0,T]\times\RR^d, \quad U_0(\bx)\in K\subset\Omega\] 
		one has
		\[
			  \cF_{k,{\bm\alpha}}^{(0)}(U)\big\vert_{t=0} \le C_{K}\cF_{k,{\bm\alpha}}(U)\big\vert_{t=0}
		\]
		with $C_K$ depending only on $K$ so that we have
		\[	\sup_{\tau\in[0,T_0]}\cF_{k,{\bm\alpha}}(U)\big\vert_{t=\tau} \leq  2 C^{(0)}\, \cF_{k,{\bm\alpha}}^{(0)}(U)\big\vert_{t=0}\leq 2 C^{(0)}\, C_{K}\cF_{k,{\bm\alpha}}(U)\big\vert_{t=0} .\]
		The proof is complete.
	\end{proof}
	\begin{Remark}	
	It is interesting to observe that the estimate in \Cref{P.Well-posedness} does not depend on the constant $N_\L$ defined in Hypothesis~\ref{H3}. In fact the result  holds with any $\L_\delta$ constant-coefficient pseudo-differential operator that is skew-adjoint for the $L^2(\RR^d)^n$ inner product ({\em i.e.} Fourier multipliers with skew-Hermitian symbols) satisfying Hypothesis~\ref{H4}, possibly of order higher than one. 
	\end{Remark}

\Cref{T.Well-posedness} is now a direct consequence of the following Proposition, providing an upper bound on an admissible functional $\cF_{k,{\bm\alpha}}(U)$ ---for solutions $U$ to system~\eqref{eq.general} and under some restrictions to the parameters $(\eps,\delta)\in\cS$--- using the control of a limited number of time derivatives. Specifically we apply \Cref{P.preparation} to infer that the control of the initial data \eqref{eq.well-prepared} assumed in the statement of \Cref{T.Well-posedness} implies the initial control of the energy functional required in \Cref{P.Well-posedness}. Since the set of weights $\bm\alpha^\lambda_{j_0,i_0}$ is $\lambda$-admissible as stated and proved in \Cref{C.weights}, \Cref{P.Well-posedness} applies and provides a control of the energy functional which is easily seen (using again the explicit formula of the set of weights $\bm\alpha^\lambda_{j_0,i_0}$) to imply the desired estimate, \eqref{eq.estimate}.  

\begin{Proposition}\label{P.preparation}
	Let $k\in\NN,\ k>d/2+1$, $(\eps,\delta)\in\cS$ and $U\in \cC(I;H^k(\RR^d))$ classical solution to system~\eqref{eq.general} under Hypotheses~\ref{H1}--\ref{H3} and satisfying
	\[ \{U(t,\bx) \ :\  (t,\bx)\in I\times \RR^d\}\subset K\subset \Omega\]
	with $I$ a real interval and $K$ compact subset of the hyperbolic domain $\Omega$.
	
	For any $j_0\in\NN^\star$ and $i_0\geq 0$ such that $ i_0+j_0>1+d/2$, any $ j_\sharp\in\NN $ such that $j_\sharp\geq j_0$ and $C_{\star}>0$, there exists $C>0$ depending uniquely on $k,j_0,i_0,j_\sharp,c_K$, $N_\L$ and $C_\star$ such that the following holds. 
	
	For any $(\eps,\delta)\in\cS$ and $\lambda\geq 1$ such that $\eps\lambda\leq \delta$, $\lambda\leq\delta^{-1}$ and $(\eps\lambda)^{j_\sharp-j_0}\leq \delta^{i_0}$ and provided
	\[M_{j_\sharp}\eqdef \sum_{j=0}^{j_\sharp-1} \norm{\partial_t^j U}_{H^{k-j}}+\frac{\eps^{ j_\sharp-j_0}}{\delta^{-i_0}} \norm{\partial_t^{j_\sharp}U}_{H^{k-j_\sharp}} \leq C_\star\]
	then one has
	\[\cF_{k,{\bm\alpha^\lambda_{j_0,i_0}}}(U)\leq C	M_{j_\sharp}  \]
	where  $\cF_{k,{\bm\alpha^\lambda_{j_0,i_0}}}(U)\eqdef\sup\big(\big\{ \alpha_{j,i}^{-1} \norm{\partial_t^j U}_{H^{i}} \ : \ (i,j)\in\RR_+\times\NN,\ i+j\leq k\big\}\big)$ with $\bm\alpha^\lambda_{j_0,i_0}$ the set of weights defined in \eqref{eq.def-alpha}, that is
	\[\alpha_{j,i}\eqdef\max(\{1,\lambda^{j-j_0}, \lambda^{j-j_0}\delta^{i_0+j_0-i-j},\eps^{j_0-j}\delta^{i_0-i}\}).\]
\end{Proposition}
\begin{proof}	
	It is obvious that for any $j\in \{0,\dots, j_\sharp-1\}$, $\alpha_{j,i}\geq 1$ yields
	\[\sup\big(\big\{ \alpha_{j,i}^{-1} \norm{\partial_t^j U(t,\cdot)}_{H^{i}} \ : \ 0\leq  i\leq k-j\big\}\big) \leq \norm{\partial_t^j U}_{H^{k-j}}\leq 	M_{j_\sharp}  .\]
	Now, for $j\geq  j_\sharp$, by definition and since $\eps^{j_0-j_\sharp}\delta^{i_0}\geq \lambda^{j_\sharp-j_0}$, we have $\alpha_{j,i}=\eps^{j_0-j}\delta^{i_0-i}$. In particular,
	\[\sup\big(\big\{ \alpha_{j_\sharp,i}^{-1} \norm{\partial_t^{j_\sharp} U(t,\cdot)}_{H^{i}} \ : \ 0\leq  i\leq k-j_\sharp\big\}\big) \leq \frac{\eps^{ j_\sharp-j_0}}{\delta^{-i_0}}\norm{\partial_t^{j_\sharp} U}_{H^{k-j_\sharp}}\leq 	M_{j_\sharp}  .\]
	We now prove by induction on  $j\in \{ j_\sharp+1,\dots,k\}$ that
	\begin{equation}\label{eq.goal} M_j\eqdef  \sup\big(\big\{ \alpha_{j,i}^{-1} \norm{\partial_t^j U(t,\cdot)}_{H^{i}} \ : \ 0\leq  i\leq k-j\big\}\big) \leq C	M_{j_\sharp}  .\end{equation}
	Let $j\geq  j_\sharp $. Using \cref{eq.general} and Hypothesis~\ref{H2} we have
	\[ 	\partial_t^{j+1} U = - \partial_t^j\Big( \sum_{l=1}^d\S_0(U)^{-1}\S_l(U)\partial_{x_l}U +  \S_0(U)^{-1}\G(U)U -\frac1\eps\S_0(U)^{-1}\L_\delta U \Big).\]
	By Faa-di Bruno's Lemma and the triangle inequality
	 we find
	\begin{align*}
		\norm{\partial_t^{j} \Big(\S_0(U)^{-1}\S_l(U)\partial_{x_l}U\Big)}_{H^i}
		&\lesssim \sum_{l=0}^j  \norm{\partial_t^l\big(\S_0^{-1}\S_l(U)\big)(\partial_t^{j-l}\partial_{x_l} U)}_{H^i}\\
		&\lesssim  \sum_{l=0}^j \sum_{\ell=0}^l  \sum_{(j_1,\dots,j_\ell)} \norm{(\S_0^{-1}\S_l)^{(\ell)}(U) \big(\prod_{n=1}^\ell \partial_t^{j_n} U \big)(\partial_t^{j-l}\partial_{x_\ell} U)}_{H^i}\\
		&\leq C  \sum_{\ell=0}^j    \sum_{(j_0,\dots,j_\ell)} \norm{\big(\prod_{n=1}^\ell \partial_t^{j_n} U \big)(\partial_t^{j_0} U)}_{H^i}\\
		&\quad + C \norm{U}_{H^i}\big\langle \sum_{\ell=0}^j    \sum_{(j_0,\dots,j_\ell)} \norm{\big(\prod_{n=1}^\ell \partial_t^{j_n} U \big)(\partial_t^{j_0} \partial_{x_\ell}U)}_{H^{s_0}}\rangle_{i>s_0}
	\end{align*}
	where $C$ depends on $k$, $C_K$, $s_0\eqdef i_0+j_0-1>d/2$ and $\norm{U}_{H^{s_0}}$, and the sums are over $(j_1,\dots,j_\ell)$ such that $1\leq j_n\leq j$ and $\sum_{n=1}^\ell j_n=l$ for the second line and $\sum_{n=0}^\ell j_n = j_0+\sum_{n=1}^\ell j_n = j$ for the last line.
	By using \Cref{L.Product} with $m_j=0=m_k=0$ and that $H^{s_0}(\RR^d)$ is a Banach algebra we infer that
	\[ \norm{\partial_t^{j} \Big(\S_0(U)^{-1}\S_l(U)\partial_{x_l}U\Big)}_{H^i}\leq C \sum_{\ell=0}^{j}\sup_{((i_n,j_n))_n } \prod_{n=0}^\ell \norm{\partial_t^{j_n} U_n}_{H^{i_n}}\]  
	where the sup is over $((i_n,j_n))_n\in(\RR_+\times\NN)^{1+\ell}$ such that $\sum_{n=0}^\ell j_n=j$, $\sum_{n=0}^\ell i_n=i+\ell s_0+1$ and for all $\ell\in\{0,\dots,\ell\}$, $(j_n,i_n+j_n)\in[(0,s_0),(j,i+j+1)]$ (recall \Cref{L.Product} for this notation). 
	
	By proceeding similarly with the other terms and using Hypothesis~\ref{H3} we infer that
	\begin{equation}\label{eq.est}
		\alpha_{j+1,i}^{-1}\norm{\partial_t^{j+1} U}_{H^i} \leq  C\sum_{\ell=0}^j\Big((1+\frac\delta\eps)  \frac{ \prod_{n=0}^\ell\alpha_{j_n,i_n}}{\alpha_{j+1,i}}  M_{j}^{1+\ell} +\frac1\eps  \frac{\prod_{n=0}^\ell \alpha_{j_n',i_n'}}{\alpha_{j+1,i}}  M_{j}^{1+\ell}\Big)\end{equation}
	where $C$ depends additionally on $N_\L$ and the products are over $((i_n,j_n))_n\in(\RR_+\times\NN)^{1+\ell}$  as above and
	$((i_n',j_n'))_n\in(\RR_+\times\NN)^{1+\ell}$ such that $\sum_{n=0}^\ell j_n'=j$, $\sum_{n=0}^\ell i_n'=i+\ell s_0$ and for all $\ell\in\{0,\dots,\ell\}$, $(j_n',i_n'+j_n')\in[(0,s_0),(j,i+j)]$. 
	
	Then using the property~\ref{vi} of admissible weights (see \Cref{L.weights}) and the fact that $\alpha_{0,s_0}=1$  and $\alpha_{j,i}=\eps^{j_0-j}\delta^{i_0-i}$ for $j\geq   j_\sharp  $ we infer that
	\[(1+\frac\delta\eps) \frac{\prod_{n=0}^\ell  \alpha_{j_n,i_n}}{\alpha_{j+1,i}}\leq (1+\frac\delta\eps)\alpha_{0,s_0}^\ell\frac{\alpha_{j,i+1}}{\alpha_{j+1,i}}= \frac\eps\delta+1\leq 2, \quad \frac1\eps\frac{\prod_{n=0}^\ell  \alpha_{j_n',i_n'}}{\alpha_{j+1,i}} \leq \frac1\eps \alpha_{0,s_0}^\ell \frac{\alpha_{j,i}}{\alpha_{j+1,i}}=1.\]
	Plugging these estimates in~\eqref{eq.est} we infer immediately ~\eqref{eq.goal} by induction on $j\in \{ j_\sharp +1,\dots,k\}$.
\end{proof}

\section{The case of weakly nonlinear singular contributions} \label{S.H5H6}
	
	In this section we consider the case when Hypothesis~\ref{H5} or Hypothesis~\ref{H6} hold, that is when the entries of the matrix $\S_0$ are asymptotically near-constant in the following sense: for any $K\subset \Omega$, there exists $C_{k,K}'>0$ such that for all $\ell\in\{1,\dots,k\}$ and any $0<\eps\leq \delta\leq 1$
	\[ \Norm{ \S_0^{(\ell)}}_{(\RR^n)^\ell\to M_n(\RR)}\leq \big(\frac\eps\delta\big)^{\ell'} C_{k,K}' \]
	where $\ell'=1$ in Hypothesis~\ref{H5} and $\ell'=\ell$ in  Hypothesis~\ref{H6}. Such hypothesis provides extra smallness to contributions of commutators with time or space differentiation with $\S_0$, which can be tracked to relax some assumptions in admissible weights in the sense of \Cref{D.alpha-admissible}. In other words, when Hypothesis~\ref{H5} or Hypothesis~\ref{H6} hold, our stability, well-posedness and convergence results hold for a larger set of functionals, hence either providing a stronger control on the solution or assuming less stringent assumptions on the initial data.
	
	\subsection{The case of  Hypothesis~\ref{H5}}\label{S.H5}
	
	\begin{Proposition}[Stability] \label{P.Stability-H5}
		If Hypothesis~\ref{H5} holds in addition to~\ref{H1}--\ref{H4}, then the statement of \Cref{P.Stability} also holds for sets of weights $\bm\alpha$ that are $\lambda$-admissible in the sense of \Cref{D.alpha-admissible} and replacing~\ref{iv} with
		\begin{enumerate}[({A}1),series=A]
				\item[\ref{iv}'] For all $j\in\{0,\dots,j_0-1\}$, $\frac\eps\delta\alpha_{j+1,0}\leq \lambda \alpha_{j,0}$.
		\end{enumerate}
	\end{Proposition}
	\begin{proof}
		It suffices to notice that in the proof of \Cref{P.Stability}, assumption~\ref{iv} was used only to deduce the estimate~\eqref{eq.control-F0} from the general estimate~\eqref{eq.est-energy}. Notice however that using Hypothesis~\ref{H5} the estimate~\eqref{eq.est-energy} may be improved to
		\[ 
			\frac12\frac{\dd}{\dd t} \bra{\S_0(V)\partial_t^j\Lambda^i U,\partial_t^j\Lambda^i U} 
			\leq  \norm{\partial_t^j U}_{H^i} \big(C(\norm{V}_{H^{s_0}})(1+\norm{V}_{H^{s_0+1}}) A_{j,i}+\norm{\partial_t^jR}_{H^i}\big) 
		\]
		where
		\begin{align*}
			A_{j,i}&\eqdef   \sum_{\ell=1}^{j+1}  
			N_\ell(V,\dots,V;U) \big\vert_{1,s_0+1}^{j+1-\ell,i+j+1-\ell}
			+ N_\ell(V,\dots,V;U) \big\vert_{0,s_0+1}^{j,i+j+1-\ell}
			+ \frac\eps\delta N_\ell(V,\dots,V;U) \big\vert_{0,s_0+1}^{j+1,i+j+1-\ell} .
		\end{align*}
		The prefactor $\frac\eps\delta$ is due to the fact that terms that led to the contribution $ N_\ell(V,\dots,V;U) \big\vert_{0,s_0+1}^{j+1,i+j+1-\ell}$ stem from commutators $[\Lambda^i,\S_0(V)] $ or $[\Lambda^i,\S_0^{(\ell)}(V)]$ ($\ell\in \{1,\dots,j\}$) and hence the contributions are of size $\cO(\frac\eps\delta)$ (notice the identity $[\Lambda^i,\S_0(V)]=\frac\eps\delta [\Lambda^i,\widetilde\S_0(V)]$ where $\widetilde\S_0(V)\eqdef\frac\delta\eps\big(\S_0(V)-\S_0(\bm0)\big)$ and its derivatives are uniformly bounded with respect to $0<\eps\leq \delta\leq 1$).
		By assumption~\ref{vi} and then~\ref{ii}--\ref{iii} we have for any $j\in\NN$ and $i\geq 1$
		\begin{align*}
			A_{j,i}&\leq \sum_{\ell=1}^{j+1} \Big( \alpha_{1,s_0}^\ell\alpha_{j+1-\ell,i}+\alpha_{0,s_0+1}^\ell\alpha_{j,i+1-\ell}+\frac\eps\delta\alpha_{0,s_0+1}^\ell\alpha_{j+1,i-\ell}\Big) M_V^\ell M_U\\
			&\leq C \Big( \alpha_{j,i}+\frac\eps\delta\alpha_{j+1,i-1}\Big)  M_V\,M_U,
		\end{align*}	
		where we use the same notations as in \Cref{S.bootstrap}. When $i\in[0,1]$ we need to add to the right-hand side above the contribution $\frac\eps\delta\norm{V}_{H^{s_0+1}}\norm{ \partial_t^{j+1} U}_{L^2}\leq \alpha_{1,s_0}\alpha_{j+1,0}M_V\,M_U$ where once again the prefactor $\frac\eps\delta$ is due to the commutator with $\S_0$. Thanks to the prefactor $\frac\eps\delta$, we can use assumption~\ref{v} instead of assumption~\ref{iv} to infer the control
		\[ A_{j,i}\leq C\,  \alpha_{j,i}\,  (1+\lambda)\, M_V\,M_U,\]
		from which estimate~\eqref{eq.control-F0} is deduced, and the rest of the proof is identical.
	\end{proof}
	\begin{proof}[Proof of \Cref{T.H5}]
		Using \Cref{P.Stability-H5} in place of \Cref{P.Stability} in the proof of  \Cref{P.Well-posedness}, and \Cref{P.preparation} with 
		the set of weights $\bm\alpha_{j_0,i_0}^{\lambda}$ defined
			for $j_0\geq 1$ and $i_0\geq 0$ satisfying $i_0+j_0>d/2+1$ by
			\[\bm\alpha_{j_0,i_0}^\lambda\eqdef \big( \alpha_{j,i}' \ : \ (i,j)\in\RR\times\NN,\ i+j\leq k\big), \quad \alpha_{j,i}'\eqdef\max(\{1,\lambda^{j-j_0}, \eps^{j_0-j}\delta^{i_0-i}\}),\]
			we obtain \Cref{T.H5}.
		
		The only thing we have to check is that $\bm\alpha_{j_0,i_0}^{\lambda}$ satisfies the assumptions of \Cref{D.alpha-admissible} with the exception of the first assumption in~\ref{iv}. Yet this is a consequence of \Cref{L.weights} and \Cref{C.weights} since $\alpha_{j,i}'$ corresponds to $P=3$ and $(\beta_1,\gamma_1)=(1,1)$, $(\beta_2,\gamma_2)=(\lambda,1)$,   and $(\beta_3,\gamma_3)=(\delta/\eps,1/\delta)$.
		\end{proof}

	\subsection{The case of  Hypothesis~\ref{H6}}\label{S.H6}
	
	\begin{Proposition}[Stability] \label{P.Stability-H6}
		If Hypothesis~\ref{H6} holds in addition to~\ref{H1}--\ref{H4}, then the statement of \Cref{P.Stability} also holds for sets of weights $\bm\alpha$ that are $\lambda$-admissible in the sense of \Cref{D.alpha-admissible}, replacing~\ref{iv} with \ref{iv}' as above, and withdrawing the assumption $ \alpha_{1,s_0}\leq 1$ from~\ref{ii}. Moreover, we can set $j_0=0$, in which case \ref{iv}' should be replaced with
	\begin{enumerate}[({A}1),series=A]
	\item[\ref{iv}''] $\frac\eps\delta\alpha_{1,0}\leq \lambda \alpha_{0,0}$.
	\end{enumerate}
	\end{Proposition}
	\begin{proof}
		Hypothesis~\ref{H6} is obviously more stringent than  Hypothesis~\ref{H5} and we already discussed in the preceding subsection why \ref{iv}' may replace~\ref{iv} in that case. There remains to track where the assumption $ \alpha_{1,s_0}\leq 1$ is used in the proof of \Cref{P.Stability}. It arises in \Cref{S.bootstrap} when estimating contributions of the form
		\[N_\ell(V,\dots,V;U)\big\vert^{j^\star,i^\star+j^\star}_{1,s_0+1} \leq  \big(\alpha_{1,s_0}  M_V\big)^\ell  \big(\alpha_{j^\star,i^\star} M_U\big).\]
		(where the inequality stems from assumption~\ref{vi}; see~\eqref{eq.key}). An inspection of the origin of such contributions reveals that all these contributions are multiplied with $\Norm{ \S_0^{(\ell)}}_{(\RR^n)^\ell\to M_n(\RR)}$, and hence by Hypothesis~\ref{H6}
		\[\Norm{ \S_0^{(\ell)}}_{(\RR^n)^\ell\to M_n(\RR)}N_\ell(V,\dots,V;U)\big\vert^{j^\star,i^\star+j^\star}_{1,s_0+1} \leq  \big( \tfrac\eps\delta\alpha_{1,s_0} M_V\big)^\ell  \big(\alpha_{j^\star,i^\star} M_U\big).\]
		There now remains to notice that $\tfrac\eps\delta\alpha_{1,s_0}\leq\alpha_{0,s_0+1}$ by assumption~\ref{v}, and the rest of the proof is identical if $j_0\in\NN^\star$.
		 
		If $j_0=0$ then a special analysis is needed. We set
			\begin{align*}
				\omega_{k,{\bm\alpha}}^{(0)}
				\eqdef \big\{(i,0) \ : \ i\in [0,k-j]\big\} \bigcup  \big\{(0,j) \ : \ j\in \{0,1,\dots,k\}\big\}.
			\end{align*}
			to define
			\[ \cF_{k,{\bm\alpha}}^{(0)}(U)\eqdef\sup\big( \big\{\alpha_{j,i}^{-1} \bra{\S_0(V)\partial_t^j\Lambda^i U,\partial_t^j\Lambda^i U}^{1/2}\ : \ (i,j)\in \omega_{k,{\bm\alpha}}^{(0)}\big\}\big)
			\]
			and the differential inequality \eqref{eq.control-F0} which yields the control of $\cF_{k,{\bm\alpha}}^{(0)}(U)$ requires to take into account new contributions: \[\cF_{k,{\bm\alpha}}^{(0),{\rm new}}(U)\eqdef \sup\big(\big\{ \alpha_{0,i}^{-1} \bra{\S_0(V)\Lambda^i U,\Lambda^i U}^{1/2}\ : \ i\in(0,k]\big\}\big).\]
			For this we use the energy estimate~\eqref{eq.est-energy} but take into account that commutators with $\S_0(U)$ yield a multiplicative prefactor $\frac\eps\delta$, by the Hypothesis~\ref{H6}. This yields (using the notations of the proof of \Cref{P.Stability}, \Cref{S.bootstrap})
			\begin{equation}\label{eq.est-energy-3} 
				\frac12\alpha_{0,i}^{-2}\frac{\dd}{\dd t} \bra{\S_0(V)\Lambda^i U,\Lambda^i U} 
				\leq \big(\alpha_{0,i}^{-1} \norm{\Lambda^i U}_{L^2}\big)\times\big(C^{(0)}\alpha_{0,i}^{-1}A_{0,i}+M_R\big) 
			\end{equation}
			where
			\begin{align*}
				A_{0,i}&\eqdef  \ 
				\frac\eps\delta N_1(V;U) \big\vert_{1,s_0+1}^{0,i}
				+ N_1(V;U) \big\vert_{0,s_0+1}^{0,i}
				+ \frac\eps\delta N_1(V;U) \big\vert_{0,s_0+1}^{1,i} +\frac\eps\delta\left\langle \norm{V}_{H^{s_0+1}}\norm{\partial_t U}_{L^2}\right\rangle_{i\in(0,1)}.
			\end{align*}
			This yields
			\begin{align*}
				A_{j,i}&\leq  \frac\eps\delta\alpha_{1,s_0}\alpha_{0,i}+\alpha_{0,s_0+1}\alpha_{0,i}+\frac\eps\delta\alpha_{0,s_0+1}\alpha_{1,\max(\{i-1,0\})}\Big) M_V M_U.
			\end{align*}
			By assumption~\ref{v} we have $\frac\eps\delta\alpha_{1,s_0}\leq \alpha_{0,s_0+1}$ and $\frac\eps\delta\alpha_{1,\max(\{i-1,0\})} \leq \alpha_{0,i}$ when $i\geq 1$, and Hypothesis~\ref{H4}'' with assumption \ref{iii} readily yield $\frac\eps\delta\alpha_{1,\max(\{i-1,0\})} \leq \lambda \alpha_{0,0}\leq \lambda \alpha_{0,i}$ when $i\in(0,1)$.
			We infer as desired
			\[ 	A_{j,i} \leq C \, (1+\lambda)\alpha_{j,i} \ M_V\, M_U,\]
			from which the estimate \eqref{eq.control-F0} follows. The proof then continues unchanged, except for the proof of \eqref{eq.induction} when $j=j_0=0$ which is no longer necessary.	
	\end{proof}
	
	\begin{proof}[Proof of \Cref{T.H6}]
	Using \Cref{P.Stability-H6} in place of \Cref{P.Stability} in the proof of  \Cref{T.Well-posedness}, and \Cref{P.preparation} with 
	the set of weights $\bm\alpha_{0,i_0}^{\lambda}$ defined
	for  $i_0>d/2+1$ by
\[\bm\alpha_{0,i_0}^\lambda\eqdef \big( \alpha_{j,i}'' \ : \ (i,j)\in\RR\times\NN,\ i+j\leq k\big), \quad \alpha_{j,i}''\eqdef\max(\{\lambda^{j}, \eps^{-j}\delta^{i_0-i}\}),\]
 we obtain \Cref{T.H6}.
	The only thing we have to check is that $\bm\alpha_{0,i_0}^{\lambda}$ satisfies the assumptions of \Cref{D.alpha-admissible} with $j_0=0$, withdrawing assumption $\alpha_{1,s_0} \leq 1$ from~\ref{ii} and replacing~\ref{iv} with \ref{iv}''.
	 Yet the restriction $j_0\geq 1$ in \Cref{L.weights} is used only to guarantee the assumption $\alpha_{1,s_0} \leq 1$, so that assumptions \ref{ii}--\ref{v} 
	 	 and \ref{vi} holds for $\alpha_{j,i}''$ which corresponds to $P=3$ and $(\beta_1,\gamma_1)=(1,1)$, $(\beta_2,\gamma_2)=(\lambda,1)$,   and $(\beta_3,\gamma_3)=(\delta/\eps,1/\delta)$. Moreover \ref{iv}'' is easy to check, using $i_0\geq 1$, and \ref{l} was already discussed in \Cref{C.weights}.
	\end{proof}

\section{Applications}\label{S.applications}
	
	In this section we apply our results to several systems arising in geophysics and oceanography, 
			selecting our cases of study so as to emphasize several key features of our results. We start with the shallow-water system with rapid rotation (small Rossby number) in conjunction with possibly small Mach number, in \Cref{S.rotating-SV}. We argue that in the case of flat bottom, our analysis is not necessary as the structural properties of the system allows to easily obtain the desired control of solutions by standard techniques. We then introduce a framework with mild topographical effects where such techniques fail, while our results apply. This may be the signature that spatially varying topography allows for the development of spatial oscillations with small wavelength when both the Rossby and Mach numbers are small, as described in \Cref{S.scale}.   
	
	Then we turn in \Cref{S.hyperbolization} to systems that recently arose as a strategy to approximate dispersive equations through stiff hyperbolic quasilinear systems. There the parameter $\eps$ is related to an artificial constraint-relaxation strategy, while the parameter $\delta$ bears physical significance (in our case measuring the strength of dispersive effects) and is typically assumed to be small in the regime of validity of the underlying dispersive equations. We start in \Cref{S.hBBM} with the ``hyperbolization'' of the Benjamin--Bona--Mahony equation. As previously we observe that standard techniques apply in the standard flat-bottom case but that taking into account topography variations may generate undesirable spurious oscillations when implementing the constraint-relaxation strategy. We then move in \Cref{S.hBP} to the case of the Boussinesq--Peregrine system, which can be interpreted as a bi-directional, multi-dimensional version of the Benjamin--Bona--Mahony equation with bottom topography. We show that the corresponding hyperbolized system still fits into our framework, but discuss additional restrictions associated with higher dimensions. Finally, we consider in \Cref{S.LCT} the case of the Serre--Green--Naghdi system which is a fully nonlinear version of the Boussinesq--Peregrine system, and was the key target of the artificial constraint-relaxation technique. 
	
\subsection{Rapidly rotating shallow-water system}\label{S.rotating-SV}

Let us consider the 2D shallow-water system with vertical Coriolis force:\footnote{See for instance~\cite[Chapter 3]{Vallis17}. This system has been intensively studied, and the singular limit of small Rossby number has been investigated in particular in the celebrated works of Embid and Majda~\cite{EmbidMajda96} and Babin, Mahalov, and Nicolaenko~\cite{BabinMahalovNicolaenko97,BabinMahalovNicolaenko97a}; see also Gallagher~\cite{Gallagher98}. However these authors considered prescribed limits ---see however~\cite{Cheng09} for a multi-scale approach--- and flat bottoms.}
\begin{equation}\label{eq.SV}
	\left\{\begin{array}{l}
		\partial_t \eta+\nabla\cdot(h\bu)=0,\\[1ex]
		\partial_t \bu+(\bu\cdot\nabla)\bu+g\nabla \eta +f\bu^\perp=\bm 0,
	\end{array}\right.
\end{equation}
where $\bu$ is the two-dimensional layer-averaged horizontal velocity, $(u_x,u_y)^\perp=(-u_y,u_x)$, $\eta$ is the surface deformation, $b$ the time-independent bathymetry and $h=H+\eta- b$ the layer depth (with $H\in\RR$ the reference depth). The positive constants $g$ and $f$ represent the (vertical) gravity acceleration  and  Earth rotation contribution in the $f$-plane approximation, respectively. After suitable scaling of the variables we can write the system under the form
\begin{equation}\label{eq.SV-adim}
	\left\{\begin{array}{l}
		\partial_t \eta+(\frac\delta\eps+\eta-\frac{\delta\beta}{\eps} b)\nabla\cdot\bu +\bu\cdot\nabla\eta-\frac{\delta\beta}\eps\bu\cdot\nabla b=0,\\[1ex]
		\partial_t \bu+\frac\delta\eps\nabla \eta +(\bu\cdot\nabla)\bu+\frac1{\eps}\bu^\perp=\bm 0.
	\end{array}\right.
\end{equation}
Here, $\eps$ is the Rossby number defined as the ratio of inertial force to Coriolis force while $\eps/\delta$ is the square of the Mach number defined as the ratio of the flow velocity to the surface wave celerity. Additionally, $\beta$ represents the magnitude of bathymetry, defined as the the ratio of the amplitude of topography variations to the reference depth. Let us distinguish between several cases, depending on the size of the bathymetry parameter.

\paragraph{The flat-bottom case}

Assuming $\beta=0$, we see that system~\eqref{eq.SV-adim} is readily of the form~\eqref{eq.general}, and our results apply (as exhibited below). Yet we would like to point out that a standard strategy quickly leads to stronger results, due to the fact that the symmetrizer $\S_0$ is almost-identity. This analysis is standard and follows the lines of~\cite{KlainermanMajda81,KlainermanMajda82} (see also~\cite{Schochet87} for the analysis specifically dedicated to the rotating shallow-water system in bounded domains), and we quickly sketch it for the sake of completeness.

It is convenient to denote $\zeta$ such that $1+\frac12 \frac\eps\delta \zeta=\sqrt{1+\frac\eps\delta \eta}$. Then~\eqref{eq.SV-adim} has the symmetric form
\begin{equation}\label{eq.SV-adim-b=0}
	\left\{\begin{array}{l}
		\partial_t \zeta+\frac\delta\eps (1+\frac12 \frac\eps\delta \zeta)\nabla\cdot\bu +\bu\cdot\nabla\zeta=0,\\[1ex]
		\partial_t \bu+\frac\delta\eps (1+\frac12 \frac\eps\delta \zeta)\nabla \zeta +(\bu\cdot\nabla)\bu+\frac1{\eps} \bu^\perp=\bm 0.
	\end{array}\right.
\end{equation}
Because the system has symmetric form and singular terms commute with space derivatives, we obtain from standard energy estimates (applying $\nabla^k$ to the system, testing against $(\nabla^k\zeta, \nabla^k\bu)$ and using commutator estimates) that for any $k\in\NN$,
\[ \frac12\frac{\dd}{\dd t} \norm{(\zeta,\bu)}_{H^k}^2 \leq C_k\, \norm{(\nabla\zeta,\nabla\bu)}_{L^\infty} \norm{(\zeta,\bu)}_{H^k}^2.
\]
where $C_k$ is a universal constant depending only on $k$. Hence assuming that $k>2$ we can use the continuous Sobolev embedding $H^{k-1} (\RR^2)\hookrightarrow L^\infty(\RR^2)$ to infer the following result:
\begin{Proposition}\label{T.SV-WP}
	Let $k>2$. For any $\eps>0,\delta>0$ and $(\zeta_0,\bu_0)\in H^k(\RR^2)^{1+2}$ (possibly depending on $\eps$ and $\delta$) 
	there exists a unique $(\zeta,\bu)\in \cC(I_{\eps,\delta};H^k(\RR^2)^{1+2})$ maximal-in-time solution to~\eqref{eq.SV-adim-b=0}. Moreover, for any $C_0>0$ there exists $C>0$ and $T>0$ depending uniquely on  $k$ and $C_0$ such that for all $(\zeta_0,\bu_0)$ satisfying
	\[ M_0\eqdef \norm{(\zeta_0,\bu_0)}_{H^k}\leq C_0,\] then the maximal time interval $I_{\eps,\delta}\supset [-T/M_0,T/M_0]$ and for all $t\in  [-T/M_0,T/M_0]$ one has
	\[  \norm{(\zeta,\bu)(t,\cdot)}_{H^k} \leq C M_0.\]
\end{Proposition}
Notice that this proposition is stronger than \Cref{T.Well-posedness} in the sense that it does not require that the initial data is well-prepared. However, similar energy estimates on time derivatives of the solution yields the following result for well-prepared initial-data.
\begin{Proposition}\label{T.SV-WP2}
	Under the assumptions and notations of \Cref{T.SV-WP}, if one assumes additionally that 
	\[ M_1\eqdef \frac1{\eps}  \norm{(\delta\nabla\cdot\bu_0,\delta\nabla \zeta_0+\bu_0^\perp)}_{H^{k-1}}\leq C_0,\]
	then one has  for all $t\in  [-T/M_0,T/M_0]$ (and augmenting $C$ if necessary)
	\[ \norm{(\partial_t\zeta,\partial_t\bu)(t,\cdot)}_{H^{k-1}} \leq C M_1.\]
\end{Proposition}
This result shows that in the situation of well-prepared initial data, the uniform control of time derivatives propagates for positive times. In this situation it is easy to conclude as in \Cref{C.convergence} the following result.
\begin{Corollary}
	Consider $(\eps_n)_{n\in\NN}$ and $(\delta_n)_{n\in\NN}$ two sequences of positive numbers such that $\eps_n\to 0$ and $\delta_n\to \delta\in [0,\infty]$, and a sequence of initial data such that for all $n\in\NN$, $(\zeta_{0,n},\bu_{0,n})\in H^k(\RR^2)^{1+2}$ satisfies the hypotheses of \Cref{T.SV-WP} and \Cref{T.SV-WP2} uniformly with respect to $n$, and converges towards $(\zeta_{0,\infty},\bu_{0,\infty})$ in $H^k(\RR^2)^{1+2}$ as $n\to\infty$. Then denoting $I\coloneqq[-T/M_0,T/M_0]$ and $(\zeta_n,\bu_n)$ the solutions to~\eqref{eq.SV-adim-b=0} emerging from the initial data $(\zeta_{0,n},\bu_{0,n})$, one has $(\zeta_n,\bu_n) \to (\zeta_\infty,\bu_\infty)$ pointwisely in $I\times\RR^2$ where $ (\zeta_\infty,\bu_\infty)\in \cC^1(I\times\RR^2)$ is a solution to the limit system
	\begin{equation}\label{eq.SV-infty}
		\left\{\begin{array}{l}
			\partial_t \zeta_\infty+\bu_\infty\cdot\nabla\zeta_\infty +\delta\nabla^\perp\cdot\bm v_\infty=0,\\[1ex]
			\partial_t \bu_\infty+(\bu_\infty\cdot\nabla)\bu_\infty+\bm v_\infty=\bm 0,\\[1ex]
			\bu_\infty=\delta\nabla^\perp \zeta_\infty,
		\end{array}\right.
	\end{equation}
	where $\nabla^\perp=(-\partial_y,\partial_x)$ and $\bm v_\infty\in \cC(I\times\RR^2)$, and $(\zeta_\infty,\bu_\infty)\big\vert_{t=0}=(\zeta_{0,\infty},\bu_{0,\infty})$.
\end{Corollary}
Notice that in~\eqref{eq.SV-infty}, $\bm v_\infty$ is the Lagrange multiplier associated with the constraint $\bu_\infty=\delta\nabla^\perp \zeta_\infty$. It is customary to consider the variable $q_\infty\eqdef \zeta_\infty-\delta\nabla^\perp\cdot\bu_\infty=(\Id-\delta^2\Delta)\zeta_\infty$ which satisfies the closed equation
\begin{equation}\label{eq.SQG}
	\partial_t q_\infty+\bu_\infty\cdot\nabla q_\infty =0, \qquad \bu_\infty=\delta\nabla^\perp (\Id-\delta^2\Delta)^{-1} q_\infty.
\end{equation}
Equation~\eqref{eq.SQG} is the shallow-water quasi-geostrophic potential vorticity equation \cite[Chapter 5.3]{Vallis17}.

In the situation when $\delta=0$,~\eqref{eq.SV-infty} reduces to the trivial equations (on the considered timescale)
\begin{equation}\label{eq.delta-0}
	\partial_t \zeta_\infty =0, \qquad \bu_\infty=\bm 0, 
\end{equation}
When $\delta=\infty$,~\eqref{eq.SV-infty} reduces to the incompressible Euler equation: $\zeta_\infty=0$, $\bm v_\infty=\nabla p_\infty$ and 
\begin{equation}\label{eq.delta-infty}
	\left\{\begin{array}{l}
		\partial_t \bu_\infty+(\bu_\infty\cdot\nabla)\bu_\infty+\nabla p_\infty=\bm 0,\\[1ex]
		 \nabla\cdot \bu_\infty =0.
	\end{array}\right.
\end{equation}

\paragraph{The small-slope case} Let us now consider a partial relaxation of the flat-bottom framework, by considering the system obtained when withdrawing the term $\frac{\delta\beta}\eps\bu\cdot\nabla b$ from the first equation in~\eqref{eq.SV-adim}:
\begin{equation}\label{eq.SV-small-slope}
	\left\{\begin{array}{l}
		\partial_t \eta+(\frac\delta\eps+\eta-\frac{\delta\beta}{\eps} b)\nabla\cdot\bu +\bu\cdot\nabla\eta=0,\\[1ex]
		\partial_t \bu+\frac\delta\eps\nabla \eta +(\bu\cdot\nabla)\bu+\frac1{\eps}\bu^\perp=\bm 0.
	\end{array}\right.
\end{equation}
Alternatively, we could consider the full system~\eqref{eq.SV-adim} and consider bottom variations such that $\frac{\delta\beta}{\eps} \nabla b=\cO(1)$ while not assuming $\beta b=\cO(1)$  (assuming  $\frac{\delta\beta}{\eps} b=\cO(1)$, the situation would be similar to the flat-bottom situation described previously, up to harmless source terms). Hence the framework we consider is that of small-slope bathymetry.

We remark that~\eqref{eq.SV-small-slope} augmented with the identity $\partial_t b=0$ is of the form~\eqref{eq.general}:
\[
	\S_0(U)\partial_t U + \S_x(U)\partial_{x}U+ \S_y(U)\partial_{y}U= \frac1{\eps}\L_\delta U,
\]
where $U\eqdef (\eta,u_x,u_y,b)^\top$ and
\[ \S_0(U)\eqdef\begin{pmatrix}
	\frac1{h}&0&0&0\\
	0&1&0&0\\
	0&0&1&0\\
	0&0&0&1
\end{pmatrix}, \ S_x(U)\eqdef\begin{pmatrix}
\frac{ u_x}{h}&0&0&0\\
	0&u_x&0&0\\
0&0&u_x&0\\
0&0&0&0
\end{pmatrix}, \ S_y(U)\eqdef\begin{pmatrix}
\frac{ u_y}{h}&0&0&0\\
0&u_y&0&0\\
0&0&u_y&0\\
0&0&0&0
\end{pmatrix},\]
with $h\eqdef 1+\frac\eps\delta\eta-\beta b$; and
\[ \quad \L_\delta\eqdef \begin{pmatrix}
	0&-\delta\partial_x&-\delta\partial_y&0\\
	-\delta\partial_x&0&1&0\\
	-\delta\partial_y&-1&0&0\\
	0&0&0&0
\end{pmatrix}\]
We note that
\[\Ker(\L_\delta)=\{(\eta,\bu,b), \quad \bu=\delta\nabla^\perp \eta=(-\delta\partial_y\eta,\delta\partial_x\eta)\}\]
whose orthogonal projection is
\[ \P_\delta\eqdef \frac1{1-\delta^2\Delta} 
\begin{pmatrix}
	1&\delta\partial_y&-\delta\partial_x&0\\
	-\delta\partial_y&-\delta^2\partial_y^2&\delta^2\partial_x\partial_y&0\\
	\delta\partial_x&\delta^2\partial_x\partial_y&-\delta^2\partial_x^2&0\\
	0&0&0&1-\delta^2\Delta
\end{pmatrix}.
\]
For any $Y\eqdef(\zeta,\bm v,c)\in \Ran(\Id-\P_\delta)\cap L^2(\RR^2)^4$ (i.e. $\zeta=\delta\nabla^\perp\cdot\bm v,c=0$), there exists a unique solution $X\eqdef(\eta,\bu,b)\in \Ran(\Id-\P_\delta)\cap H^1(\RR^2)^4$ to the equation $\L_\delta X=Y$, and one has the formula
\[\eta=(1-\delta^2\Delta)^{-1}\delta\nabla\cdot\bm v, \quad 
\bu^\perp =-(1-\delta^2\Delta)^{-1}(\bm v-\delta\nabla^\perp\zeta).\]
This yields immediately the bounds
\[ \big\vert \eta\big\vert_{L^2} +\delta\big\vert \eta\big\vert_{H^1}\lesssim \big\vert \bm v\big\vert_{L^2}, \quad \big\vert \bu\big\vert_{L^2}+\delta\big\vert \bu\big\vert_{H^1}\lesssim
\big\vert (\zeta,\bm v)\big\vert_{L^2}.\]
Hence Hypothesis~\ref{H4} holds with $C_\L=0$ (see also \Cref{R.H3H4}). Hypotheses~\ref{H1}--\ref{H3} are obvious, and the hyperbolicity criterion associated with the matrix $\S_0$ being definite positive is the standard non-cavitation assumption:
\[ \inf_{\RR^2} 1+\frac\eps\delta\eta-\beta b >0.\]

We can therefore apply \Cref{T.Well-posedness} as well as \Cref{C.convergence}. In order to clarify sufficient conditions on the initial data and restrictions on the parameters $(\eps,\delta)$ for which these results hold, we set $j_0=1$, $j_\sharp=2$ and $i_0\in(1,2]$. Hence we need to find sufficient conditions on the initial data such that  $(\eta,\bu)\in \cap_{j=0}^k\cC^j(I_{\eps,\delta};H^{k-j}(\RR^2)^{1+2})$   (with $k\geq 3$) the emerging solution to~\eqref{eq.SV-small-slope} satisfies 
\begin{equation}\label{eq.WP-SV}
\norm{(\eta,\bu)}_{H^k} +\norm{ (\partial_t\eta,\partial_t\bu)}_{H^{k-1}}+\norm{\tfrac{\eps}{\delta^{i_0}} (\partial_t^2\eta,\partial_t^2\bu)}_{H^{k-2}} =\cO(1).
\end{equation}
By product estimates in \Cref{L.Product} we have immediately

\[\norm{\partial_t \bu}_{H^{k-1}}\lesssim \norm{\bu}_{H^{k-1}}\norm{\bu}_{H^{k}}+\frac1\eps\norm{\bu-\delta\nabla^\perp \eta}_{H^{k-1}}\]
and, since $\nabla\cdot\nabla^\perp \eta=0$,
\[	\norm{	\partial_t \eta}_{H^{k-1}}\lesssim \norm{\bu}_{H^{k-1}}\norm{\eta}_{H^{k}} +\norm{\eta}_{H^{k-1}} \norm{\bu}_{H^{k}} +\frac\delta\eps(1+\beta\norm{b}_{H^{k-1}})\norm{\bu-\delta\nabla^\perp \eta}_{H^{k}}.\]
Finally, differentiating with time~\eqref{eq.SV-small-slope} yields
\[		\norm{\partial_t^2 \eta}_{H^{k-2}}\lesssim 	 (\norm{\bu}_{H^{k-1}}+\norm{\eta}_{H^{k-1}} ) (\norm{\partial_t\bu}_{H^{k-1}}+\norm{\partial_t \eta}_{H^{k-1}} )+\frac{\delta}{\eps} (1+\beta\norm{b}_{H^{k-2}})\norm{\partial_t\bu}_{H^{k-1}}\]
		and
	\[	\norm{\partial_t^2 \bu}_{H^{k-2}}\leq \frac{\delta}{\eps} \norm{ \partial_t\eta}_{H^{k-1}} +	\norm{\bu}_{H^{k-1}}\norm{\partial_t\bu}_{H^{k-1}}+\frac1{\eps}\norm{\partial_t\bu}_{H^{k-1}}.
\]
From the above we infer that \eqref{eq.WP-SV} holds provided $\norm{\bu-\delta\nabla^\perp \eta}_{H^{k}}=\cO(\delta^{i_0}\eps)$ (which yields in particular $\norm{\bu}_{H^{k}}=\cO(\delta)$ and hence $\norm{\partial_t \bu}_{H^{k-1}}=\cO(\delta^{i_0})$ and $\norm{\partial_t \eta}_{H^{k-1}}=\cO(\delta)$). This yields the following results ensuing from  \Cref{T.Well-posedness} and \Cref{C.convergence} (with $j_0=1$ and $j_\sharp=2$).
\begin{Proposition}[Well-posedness]\label{P.Well-posedness-SV}
	Let $k\in\NN,\ k>2$, and $h_\star>0$. For any $\eps,\delta,\beta>0$ 
	 and any $(\eta_0,\bu_0,b)\in H^k(\RR^2)^{1+2+1}$ satisfying 
	\[\inf_{\RR^2} 1+\frac\eps\delta\eta_0-\beta b \geq h_\star>0, 
	\]
	there exists a unique $(\eta,\bu)\in \cC(I_{\eps,\delta};H^k(\RR^2)^{1+2})$ maximal-in-time classical solution to system~\eqref{eq.SV-small-slope} emerging from the initial data $(\eta,\bu)\big\vert_{t=0}=(\eta_0,\bu_0)$, and one has $(\eta,\bu)\in \cap_{j=0}^k\cC^j(I_{\eps,\delta,\beta};H^{k-j}(\RR^2)^{1+2})$. 
	
	Moreover, for any $C_0>0$ and $i_0\in(1,k-1]$, there exists $T>0$, $C>0$ and $\lambda\geq1$ depending uniquely on $k,h_\star,i_0$ and $C_0$ such that the following holds. 
	
	Assume that  $(\eps,\delta,\beta)$ is such that $0<\eps\leq\delta\leq\beta\leq 1$, $\delta\leq1/\lambda$ and $\eps \leq \delta^{i_0}/\lambda$ and $(\eta_0,\bu_0,b)$ is such that
	\[ M_0\eqdef \norm{\eta_0}_{H^{k}}+\frac1{\delta^{i_0}\eps}\norm{\bu_0-\delta\nabla^\perp \eta_0}_{H^k} +\beta\norm{b}_{H^{k}}\leq C_0 .
	\]
	 Then $I_{\eps,\delta,\beta}\supset [-T/M_0,T/M_0]$ and for any $t\in [-T/M_0,T/M_0]$, one has $\inf_{\RR^2} 1+\frac\eps\delta\eta(t,\cdot)-\beta b \geq h_\star/2$ and for all $i\in[i_0+1,k]$,
	\[\big\vert (\eta,\bu)(t,\cdot)\big\vert_{H^{i}}+\big\vert (\partial_t\eta, \partial_t\bu)(t,\cdot)\big\vert_{H^{i-1}} \\ \leq C\, \frac{M_0}{\delta^{i-(i_0+1)}} .\]
\end{Proposition}

\begin{Corollary}[Convergence]\label{C.convergence-SV}
	Let $(\delta_n)_{n\in\NN}$ and $(\eps_n)_{n\in\NN}$ be sequences of positive numbers such that $\eps_n\to 0$ and $\delta_n\to \delta_\infty\in[0,1]$  as $n\to\infty$, and $\beta>0$, $b\in H^k(\RR^2)$ and a sequence of initial data such that for all $n\in\NN$, $(\eta_{0,n},\bu_{0,n})_{n\in\NN}\in H^k(\RR^2)^{1+2}$ and the assumptions of \Cref{P.Well-posedness-SV} hold uniformly with respect to $n\in\NN$, and $(\eta_{0,n},\bu_{0,n})_{n\in\NN}\to (\eta_{0,\infty},\bu_{0,\infty})$ in $H^k(\RR^2)^{1+2}$. Denote the time interval $I\eqdef [-T/M_0,T/M_0]$ and ${(\eta_n,\bu_n)\in \cC(I;H^k(\RR^2))}$ the solution to ~\eqref{eq.SV-small-slope} emerging from the initial data $(\eta_n,\bu_n)\big\vert_{t=0}=(\eta_{0,n},\bu_{0,n})$.
	
	Then there exists $(\eta_\infty,\bu_\infty)\in \cC^1(I\times \RR^2)$ and $\bm v_\infty \in \cC(I\times\RR^2)$ such that $(\eta_n,\bu_n)\to (\eta_\infty,\bu_\infty)$ as $n\to\infty$ pointewisely in $I\times\RR^2$, and $(\eta_\infty,\bu_\infty,\bm v_\infty)$ satisfies the limit system
	\begin{equation}\label{eq.SV-infty-small-slope}
		\left\{\begin{array}{l}
			\partial_t \eta_\infty+\bu_\infty\cdot\nabla\eta_\infty +(1-\beta b)\delta_\infty\nabla^\perp\cdot\bm v_\infty=0,\\[1ex]
			\partial_t \bu_\infty+(\bu_\infty\cdot\nabla)\bu_\infty+\bm v_\infty=\bm 0,\\[1ex]
			\bu_\infty=\delta_\infty\nabla^\perp \eta_\infty,
		\end{array}\right.
	\end{equation}
	where $\nabla^\perp=(-\partial_y,\partial_x)$, and $(\eta_\infty,\bu_\infty)\big\vert_{t=0}=(\eta_{0,\infty},\bu_{0,\infty})$.
\end{Corollary}

\paragraph{The large bottom variations case}

The general setting~\eqref{eq.SV-adim} without smallness assumption on the bathymetry does not fit into our framework as the system does not take the form \eqref{eq.general} with linear stiff terms, and the uniform control of solutions as $\eps\searrow0$ is open as far as we know. We would like to comment that the situation without Coriolis force was treated by Bresch and Métivier in~\cite{BreschMetivier10}. By exploiting a delicate structure of the system, they managed to obtain uniform bounds in the limit of vanishing Mach number (that is $\frac\eps\delta\searrow 0$ with our notations) and convergence results even for ill-prepared initial data. We expect that their analysis may be extended to incorporate Coriolis force when $\delta\gtrsim 1$, but another analysis is required to study the rotation-dominant case, $\delta\ll 1$.

\subsection{Augmented hyperbolized systems}\label{S.hyperbolization}

As previously mentioned, the ``augmented hyperbolized systems'' we consider in this section stem from a strategy of constraint-relaxation applied to dispersive equations. As a consequence, the parameter $\eps$ which is related to the constraint relaxation can be freely chosen and is expected to drive the accuracy of the approximation, while the parameter $\delta$ is characterizing by the physical situation.
In the same spirit, some unknowns of the relaxed systems are ``augmented unknowns'' that may be directly related to the physical quantities of the underlying dispersive equation only in the asymptotic limit $\eps\searrow 0$. Consequently a crucial discussion is the way these augmented unknowns are initially set, since it is expected to have a consequence on the accuracy of the approximation. From a technical viewpoint, our assumptions of well-prepared initial data may be secured only if these augmented variables are set  ``sufficiently well''. Since preparing the initial data in practice is typically cumbersome and may be numerically costly, it was the main incentive of this work to relax as much as possible the requirements on initial data to allow only {\em mildly well-prepared} initial data. We refer to~\cite{Duchene19} for further discussion on that matter, and to \cite[Chapter~9]{MM4WW} and~\cite{KetchesonBiswas25,GiesselmannRanocha25} for further results and references on augmented hyperbolized systems.

\subsubsection{The Benjamin--Bona--Mahony equation with variable bottom}\label{S.hBBM}

Consider the Benjamin--Bona--Mahony equation with variable bottom:
\footnote{See \cite[Chapter~7]{Lannes} and~\cite{Israwi10} for its rigorous justification.}
\begin{equation}\label{eq.BBM-b}
	(1-\delta^2 B^4\partial_x^2)\partial_t u+(B+\tfrac32 u)\partial_x u +\tfrac32 (\partial_x B)u=0,
\end{equation}
which is an extension of the celebrated Benjamin--Bona--Mahony equation for the propagation of unidirectional free-surface gravity water waves~\cite{BenjaminBonaMahony72} that incorporates bottom topography contributions.
In~\eqref{eq.BBM-b}, the parameter $\delta>0$ measures the strength of dispersion, and $B(x)=h_b(x)^{1/2}$ where $h_b(x)$ is given as the depth of the layer at rest (after nondimensionalization). Notice that~\eqref{eq.BBM-b} is derived using a small-amplitude, small-slope assumption $h_b(x)=1-\beta b(\alpha x)$ with $\alpha,\beta\ll 1$ but that we will not make use of such hypothesis in our analysis.  
Following~\cite{GavrilyukShyue22} that introduced an augmented ``hyperbolized'' system for approximating the standard Benjamin--Bona--Mahony equation, we shall consider the following augmented system:
\begin{equation}\label{eq.hBBM-b}
	\left\{\begin{array}{l}
		\partial_t u+(B+\frac32 u)\partial_x u +\delta^2B^4\partial_xv+\frac32 (\partial_x B)u=0,\\[1ex]
		\eps^2\partial_t v+\partial_xu+w=0,\\[1ex]
		\partial_t w-v=0.
	\end{array}\right.
\end{equation}
Formal asymptotic expansion from the second equation of~\eqref{eq.hBBM-b} yields
\[ w = -\partial_xu+\cO(\eps^2)\quad \text{and hence}\quad  v = \partial_t w=-\partial_x\partial_tu+\cO(\eps^2).\]
Plugging the latter expansion in the first equation of~\eqref{eq.hBBM-b}, we expect that the system~\eqref{eq.hBBM-b} may provide approximate solutions to the dispersive equation~\eqref{eq.BBM-b} with precision (at most) $\cO(\delta^2\eps^2)$.

Based on these formal asymptotic expansions we complement the system~\eqref{eq.hBBM-b} with the following initial data:
\begin{equation}\label{eq.hBBM-b-id}
	u\vert_{t=0}=u_0, \quad v\vert_{t=0}=0, \quad w\vert_{t=0}=-(u_0)'.
\end{equation}
Notice that considering $ w\vert_{t=0}=-(u_0)'$ allows to secure the uniform boundedness of time derivatives of solutions at initial time: we are in the framework of well-prepared initial data.  
A stronger notion of well-prepared initial data would consist in setting $v\vert_{t=0}=v_0$ where $v_0$ is determined by the elliptic problem $v_0-\delta^2(B^4(v_0)')'=[(B+\frac32 u_0) (u_0)' +\frac32 B'u_0] '$. While we expect that such refined initial data shall improve the behavior of emerging solutions to~\eqref{eq.hBBM-b} as approximate solutions to~\eqref{eq.BBM-b} in the limit $\eps\searrow 0$, we do not consider this case as the primary aim of this work is precisely to demonstrate that convergence still holds with mildly well-prepared initial data. 

In order to apply our results, let us notice that~\eqref{eq.hBBM-b} can be written in matrix form as
\begin{equation}\label{eq.hBBM-b-matrix}
 \S_0(U)	\partial_t U + \S(U)\partial_{x}U+\G(U)U= \frac1\eps\L_\delta U,
\end{equation}
where $U\eqdef (u,\eps\delta v,\delta w,B-1,\partial_x B)^\top$ and
\[ \S_0(U)\eqdef\begin{pmatrix}
	\frac1{B^4}&0&0&0\\
	0&1&0&0&0\\
	0&0&1&0&0\\
	0&0&0&1&0\\
	0&0&0&0&1
\end{pmatrix}\!,\ \S(U)\eqdef\begin{pmatrix}
	\frac{B+\frac32u}{B^4}&0&0&0&0\\
	0&0&0&0&0\\
	0&0&0&0&0\\
	0&0&0&0&0\\
	0&0&0&0&0
\end{pmatrix}\!, \ \G(U)\eqdef\begin{pmatrix}
\frac{\frac32(\partial_x B)}{B^4}&0&0&0&0\\
0&0&0&0&0\\
0&0&0&0&0\\
0&0&0&0&0\\
0&0&0&0&0
\end{pmatrix}, \]
and
\[ \L_\delta\eqdef \begin{pmatrix}
	0&-\delta\partial_x&0&0&0\\
	-\delta\partial_x&0&-1&0&0\\
	0&1&0&0&0\\
	0&0&0&0&0\\
	0&0&0&0&0
\end{pmatrix}.\]
Hence we see that~\eqref{eq.hBBM-b-matrix} takes the form~\eqref{eq.general} and satisfies Hypotheses~\ref{H1},~\ref{H2} with $\Omega=\RR$,~\ref{H3}, and~\ref{H4} with $C_\L=0$. Indeed, we have
\[\Ker(\L_\delta)=\{(u,v,w,a,b)\in L^2(\RR)^5, \quad v=0,\ w=-\delta\partial_x u\}\]
whose orthogonal projection is
\[ \P_\delta\eqdef 
\begin{pmatrix}
	\frac1{1-\delta^2\partial_x^2} &0&\frac{\delta\partial_x}{1-\delta^2\partial_x^2} &0&0\\
	0&0&0&0&0\\
	\frac{-\delta\partial_x}{1-\delta^2\partial_x^2} &0&\frac{-\delta^2\partial_x^2}{1-\delta^2\partial_x^2} &0&0\\
	0&0&0&1&0\\
	0&0&0&0&1
\end{pmatrix}.
\]
For any $Y\eqdef(\tilde u,\tilde v,\tilde w,\tilde a,\tilde b)\in \Ran(\Id-\P_\delta)\cap L^2(\RR)^5$ (i.e. $\tilde u=-\delta\partial_x\tilde w,\tilde a=\tilde b=0$), there exists a unique solution $X\eqdef(u,v,w,a,b)\in \Ran(\Id-\P_\delta)\cap H^1(\RR)^5$ to the equation $\L_\delta X=Y$, and one has the formula
\[w=-(1-\delta^2\partial_x^2)^{-1}\tilde v, \quad u = \delta\partial_x (1-\delta^2\partial_x^2)^{-1}\tilde v, \quad v=\tilde w, \quad a=b=0.\]
This yields immediately the second estimate in~\ref{H4} (the first one with $C_\L=0$ being an obvious consequence of the definition of $\P_\delta$): for all $\delta\in(0,1]$,
\[ \big\vert u \big\vert_{L^2}+\delta\big\vert  u\big\vert_{H^1}+ \big\vert w\big\vert_{L^2} +\delta\big\vert w\big\vert_{H^1}\lesssim \big\vert \tilde v\big\vert_{L^2}, \quad \big\vert v \big\vert_{L^2}+\delta\big\vert  v\big\vert_{H^1}\lesssim \big\vert \tilde w\big\vert_{L^2}+\delta \big\vert \partial_x \tilde w\big\vert_{L^2} \leq \big\vert (\tilde u,\tilde w)\big\vert_{L^2}.\]
\medskip

We wish to apply \Cref{T.Well-posedness}. We set $j_0=1$, $j_\sharp=2$ and $i_0=1$, and hence we need to verify that
$U\in \cap_{j=0}^k\cC^j(I_{\eps,\delta};H^{k-j}(\RR))$   (with $k\geq2$) the solution to~\eqref{eq.hBBM-b} emerging from the initial data~\eqref{eq.hBBM-b-id} satisfies 
\[\norm{U}_{H^k} +\norm{\partial_t U}_{H^{k-1}}+\norm{\tfrac{\eps}{\delta}\partial_t^2 U}_{H^{k-2}} =\cO(1).\]
We have immediately
\[ \big\vert U\vert_{t=0}\big\vert_{H^k}\leq \big\vert u_0\big\vert_{H^k}+\delta\big\vert (u_0)'\big\vert_{H^k}+\big\vert B-1\big\vert_{H^{k+1}}.\]
Since $\partial_tB=0$, using the equations~\eqref{eq.hBBM-b} and initial data~\eqref{eq.hBBM-b-id} yields $\partial_t v\vert_{t=0}=\partial_t w\vert_{t=0}=0$ and
\begin{align*} \big\vert \partial_t U\vert_{t=0}\big\vert_{H^{k-1}}&=\big\vert \partial_t u\vert_{t=0}\big\vert_{H^{k-1}}= \big\vert (B+\tfrac32 u_0)\partial_x u_0+\tfrac32 (\partial_x B)u_0\big\vert_{H^{k-1}}\\
	&\lesssim (1+\big\vert u_0\big\vert_{H^{k-1}}+\big\vert B-1\big\vert_{H^{k-1}})\big(\big\vert u_0\big\vert_{H^k}+\big\vert B-1\big\vert_{H^k}\big),
\end{align*}
where we used \Cref{L.product} for the last inequality. Finally, differentiating with time the equations~\eqref{eq.hBBM-b}, we easily check that 
\[\big\vert \partial_t^2 u\vert_{t=0}\big\vert_{H^{k-2}}\lesssim (1+\big\vert u_0\big\vert_{H^{k-1}}+\big\vert B-1\big\vert_{H^{k-1}})\big\vert \partial_t u\vert_{t=0}\big\vert_{H^{k-1}},\]
and from the identities $(\partial_t^2 w)\vert_{t=0}=(\partial_t v)\vert_{t=0}=0 $ and $(\partial_t^2 v)\vert_{t=0}=-\frac1{\eps^2}(\partial_x \partial_t u)\vert_{t=0}$ and recalling $U\eqdef (u,\eps\delta v,\delta w,B-1,\partial_x B)^\top$, we infer that for any $0<\eps\leq \delta\leq 1$ one has
\[ \frac\eps\delta\big\vert \partial_t^2 U\vert_{t=0}\big\vert_{H^{k-2}}\lesssim  (1+\big\vert u_0\big\vert_{H^{k-1}}+\big\vert B-1\big\vert_{H^{k-1}})^2\big(\big\vert u_0\big\vert_{H^k}+\big\vert B-1\big\vert_{H^k}\big).\]

We therefore obtain the following Proposition as a consequence of \Cref{T.Well-posedness} and \Cref{C.convergence} (with $i_0=j_0=1$ and $j_\sharp=2$).

\begin{Proposition}\label{P.WP-BBM-b}
	Let $k\geq 2$. For any $\eps,\delta>0$, any $u_0\in H^{k+1}(\RR)$ and $B$ such that $B-1\in H^{k+1}(\RR)$ and $B-1\geq b_\star>0$, there exists a unique $(u,v,w)^\top\in \cC(I_{\eps,\delta};H^k(\RR)^3)$ maximal-in-time classical solution to system~\eqref{eq.hBBM-b} emerging from the initial data~\eqref{eq.hBBM-b-id} and one has $U\in \cap_{j=0}^k\cC^j(I_{\eps,\delta};H^{k-j}(\RR)^3)$. 

Moreover, for any $C_0>0$, there exists $T>0$, $C>0$ and $\lambda\geq1$ depending uniquely on $k,b_\star$ and $C_0$ such that the following holds. 
Assume that  ${0<\eps\leq \delta\leq 1}$, $\delta\leq1/\lambda$ and $\eps/\delta\leq 1/\lambda$ and
\[ M_0\eqdef\big\vert u_0\big\vert_{H^k}+\delta\big\vert (u_0)'\big\vert_{H^k}+\big\vert B-1\big\vert_{H^{k+1}} \leq C_0. 
\]
Then $I_{\eps,\delta}\supset [-T/M_0,T/M_0]$ and for any $t\in [-T/M_0,T/M_0]$, one has for all $i\in[2,k]$,
\begin{equation}\label{eq.estimate-hBBM} \big\vert (u,\delta\eps v,\delta w)(t,\cdot)\big\vert_{H^{i}}+\big\vert (\partial_tu,\delta\eps \partial_tv,\delta \partial_tw)(t,\cdot)\big\vert_{H^{i-1}} 
	 \leq  \frac{C\,M_0}{\delta^{i-2}}
\end{equation}
Moreover, as $\eps \searrow0$ (with $\delta>0$ fixed), denoting $(u_\eps,v_\eps,w_\eps)$ the emerging solution to~\eqref{eq.hBBM-b}-\eqref{eq.hBBM-b-id}, we have $(u_\eps, w_\eps)\to (u,-\delta\partial_x u)$ pointwisely in $[-T/M_0,T/M_0]\times\RR $ where $u$ satisfies~\eqref{eq.BBM-b} and $u\vert_{t=0}=u_0$.
\end{Proposition}
This result proves that the augmented system~\eqref{eq.hBBM-b} provides approximate solutions to the Benjamin--Bona--Mahony equation with bottom topography~\eqref{eq.BBM-b}. Notice however that the upper-bound~\eqref{eq.estimate-hBBM} underlines the possibility of undesirable spurious oscillations with wavelength $\cO(\delta)$.

\paragraph{The flat-bottom situation}
	Of course all the previous results hold in particular in the flat-bottom situation, that is $B\equiv 1$, and provide a rigorous justification of the augmented hyperbolized system introduced in~\cite{GavrilyukShyue22} as an approximation of the Benjamin--Bona--Mahony equation. Notice however that, as in the flat-bottom situation of \Cref{S.rotating-SV}, owing to the fact that $\S_0=\Id$, standard energy estimates quickly provide a uniform control of solutions to~\eqref{eq.hBBM-b}, even in the case of ill-prepared initial data, for instance $u\vert_{t=0}=u^0$ and  $ v\vert_{t=0}=w\vert_{t=0}=0$. Moreover, quantitative convergence results for well-prepared initial data are easily obtained by means of similar energy estimates applied to the difference between solutions to~\eqref{eq.hBBM-b} and approximations built from the corresponding solutions to~\eqref{eq.BBM-b}. The aforementioned small-scale development does not arise in that case. We refer the reader to~\cite{GiesselmannRanocha25} for a detailed analysis.

	\subsubsection{The Boussinesq--Peregrine system}\label{S.hBP}
	
	Consider the Boussinesq--Peregrine system for the propagation of water waves:\footnote{See \cite[Chapter~6]{Lannes} and~\cite{DucheneIsrawi18} for its rigorous justification in the weakly nonlinear, weakly dispersive regime.}
	\begin{equation}\label{eq.BP}
		\left\{\begin{array}{l}
			\partial_t h+\nabla\cdot(h\bu)=0,\\[1ex]
			\partial_t\bu +\nabla(h+ b)+(\bu\cdot\nabla)\bu+\frac{\delta^2}{\hb}\big( \nabla\big( \hb p \big)+ \pb \nabla b\big)=\bm{0},\\[1ex]
			p= \hb \big(\frac{-\hb\partial_t\nabla\cdot\bu}3 +\frac{\partial_t\bu\cdot\nabla b}2\big), \quad \pb=\hb\big(\frac{-\hb\partial_t\nabla\cdot\bu}2 +\partial_t\bu\cdot\nabla b\big),
		\end{array}\right.
	\end{equation}
	where $\delta>0$ is the shallowness parameter measuring the strength of dispersion.
	There, $h$ represents the depth of the layer of water, $b$ accounts for the bottom topography (so that $h+b-1$ is the surface elevation), $\hb=1-b$ is the depth of the layer of water at rest, $\bu$ the layer-avereged horizontal velocity, and $p$ and $\pb$ are the linear non-hydrostatic pressure contributions. All these quantities have been non-dimensionalized.
	
	Denote $q\eqdef \hb p$ and $ r \eqdef \dot h/h$. Notice $r=-\nabla\cdot\bu$ (using the mass conservation equation $\partial_t h+\nabla\cdot(h\bu)=0$), $q = \frac{\hb^3}{3}\partial_t r +\frac{\hb^2}{2}\partial_t\bu\cdot\nabla b$ and $\pb=\frac{\hb^2}{2}\partial_t r +\hb\partial_t\bu\cdot\nabla b$.
	We infer that~\eqref{eq.BP} may be written equivalently as
	\begin{equation}\label{eq.BP-q}
		\left\{\begin{array}{l}
			\partial_t h+\nabla\cdot(h\bu)=0,\\[1ex]
			\hb\big(\partial_t \bu+(\bu\cdot\nabla)\bu+\nabla(h+b)\big)+\delta^2\nabla q+\frac{\delta^2}{2}(\hb^2\partial_t r )\nabla b+\delta^2 \hb(\partial_t\bu\cdot\nabla b)\nabla b=\bm{0}, \\[1ex]
			 r + \nabla\cdot\bu=0,\\[1ex]
			\frac{\hb^3}{3}\partial_t r +\frac{\hb^2}{2}\partial_t\bu\cdot\nabla b =q,
		\end{array}\right.
	\end{equation}
	The hyperbolization strategy consists in relaxing the third identity in~\eqref{eq.BP-q} to consider
	\begin{equation}\label{eq.hBP}
		\left\{\begin{array}{l}
			\partial_t h+\nabla\cdot(h\bu)=0,\\[1ex]
			\hb\big(\partial_t \bu+(\bu\cdot\nabla)\bu+\nabla(h+b)\big)+\delta^2\nabla q+\frac{\delta^2}{2}(\hb^2\partial_t r )\nabla b+\delta^2 \hb(\partial_t\bu\cdot\nabla b)\nabla b=\bm{0}, \\[1ex]
			\eps^2\partial_t q=-\hb( r + \nabla\cdot\bu),\\[1ex]
			\frac{\hb^3}{3}\partial_t r +\frac{\hb^2}{2}\partial_t\bu\cdot\nabla b =q.
		\end{array}\right.\end{equation}
	We now write~\eqref{eq.hBP} in compact matricial form. 
	Denoting $U=(h-\hb,\bu,\eps \delta q,\delta  r )^\top$, system~\eqref{eq.hBP} reads
	\begin{equation}\label{eq.hBP-matrix}	
		\S_0(U)\partial_t U + (\S(U)\nabla)U+B(U)=\frac1\eps\L_{\delta} U 
	\end{equation}
	with
	\[ \S_0(U)\eqdef
	\begin{pmatrix}
		\frac{\hb}{h} & 0 & 0& 0\\
	\bm{0}	& \hb {\sf I}[\nabla b] & \bm{0}&\frac\delta2 \hb^2\nabla b \\
	0	& \bm{0}^\top& \hb^{-1} &  0\\
	0	& \frac\delta2 \hb^2(\nabla b)^\top& 0 & \frac13\hb^3
	\end{pmatrix}, \quad 
	\L_{\delta}\eqdef
	\begin{pmatrix}
		0 & \bm{0}^\top& 0& 0\\
		\bm{0}& {\sf 0}&-\delta\nabla &\bm{0} \\
		0& -\delta\nabla^\top& 0& -1 \\
		0& \bm{0}^\top& 1 & 0
	\end{pmatrix}, \quad 
	\]
	where ${\sf I}[\nabla b]\circ\eqdef\Id+\delta^2 (\nabla b)(\nabla b)\cdot\circ$ and
	\[ \S(U)\nabla\eqdef
	\begin{pmatrix}
		\frac{\hb}{h}\bu\cdot\nabla& \hb\nabla^\top & 0& 0\\
		\hb\nabla& \Id(\hb \bu\cdot\nabla)   &\bm{0} & \bm{0}\\
		0& \bm{0}^\top& 0& 0\\
		0& \bm{0}^\top& 0& 0
	\end{pmatrix}, \quad
	B(U) \eqdef \begin{pmatrix}
		-\frac{\hb}{h}\bu\cdot\nabla b\\
		\bm{0}\\
		0 \\
		0
	\end{pmatrix}.
	\]
	Proceeding as in the preceding section, we can artificially augment the set of unknowns $U$ to include $b,\nabla b$ so that~\eqref{eq.hBP-matrix} is of the form~\eqref{eq.general}, and satisfies~\ref{H1}--\ref{H4}. Notice in particular the coercivity of $\S_0$ in Hypothesis~\ref{H2} under the hyperbolicity assumption $h,h_b>0$ (physically corresponding to non-cavitation) that follows from the following identity: for all $V=(v_1,\bm{v}_2,v_3,v_4)\in \RR^{3+d}$,
	\begin{equation}\label{eq.hyperbolicity} \bra{\S_0(U) V,V}=  \frac{\hb}{h}v_1^2+\hb|\bm{v}_2|^2+\hb(\tfrac12 \hb v_4+\delta \bm{v}_2\cdot\nabla b)^2+\hb^{-1}v_3^2+\tfrac1{12} \hb^3 v_4^2.\end{equation}
	Hence the results of \Cref{T.Well-posedness} and \Cref{C.convergence} apply for sufficiently well-prepared initial data.
	
	Let us analyze the size of initial data given $\eta_0=(h-\hb)\vert_{t=0}\in H^k(\RR^d)$, $\bu_0=\bu\vert_{t=0}\in H^k(\RR^d)^d$ and setting the augmented unknowns as follows:
\begin{equation}\label{eq.hBP-init0}
		q\vert_{t=0}=0, \quad  r \vert_{t=0}=-\nabla\cdot\bu_0. 
\end{equation}
	Recall $U=(h-\hb,\bu,\eps \delta q,\delta  r )^\top$. We have immediately
	\[\norm{U\vert_{t=0}}_{H^k}\lesssim \norm{\eta_0}_{H^k}+\norm{\bu_0}_{H^k}+\norm{\delta\nabla\cdot\bu_0}_{H^k}.\]
	Then we use that $U_0\eqdef U\vert_{t=0}\in \Ker(\L_\delta)$ to infer from~\eqref{eq.hBP-matrix} that
	\[\norm{(\partial_t U)\vert_{t=0}}_{H^{k-1}}=\norm{\S_0(U)^{-1}\big((\S(U_0)\nabla)U_0+B(U_0)\big)}_{H^{k-1}}\leq C\,\norm{U_0}_{H^k}\]
	where $C$ depends non-decreasingly on $\norm{U_0}_{H^{k-1}}$, $\norm{\nabla b}_{H^{k-1}}$ and $\norm{( h,  \hb,h^{-1},\hb^{-1})}_{L^\infty}$. Here we used the pointwise coercivity of the matrix $\S_0(U)$ as well as product, commutator and composition estimates in \Cref{L.product-2,L.commutator,L.composition}.
	
	The estimate on the second derivative follows similarly after differentiating in time~\eqref{eq.hBP-matrix}. The only non-uniform contribution is
	\[(\frac1\eps\L_\delta \partial_tU )\vert_{t=0}= \frac1\eps (0,-\eps \delta^2\nabla\partial_t q,-\delta\partial_t r -\delta\nabla\cdot\partial_t\bu,\eps \delta\partial_t q)^\top\vert_{t=0}.\]
	Now, considering~\eqref{eq.hBP} and the initial condition~\eqref{eq.hBP-init0}, we have
	\[	(\partial_t r )\vert_{t=0}=-\frac{3}{2\hb}(\partial_t\bu)\vert_{t=0}\cdot\nabla b \quad \text{ and } \quad \partial_t q\vert_{t=0}=0 .\]
	We infer
		\[\norm{(\partial_t^2 U)\vert_{t=0}}_{H^{k-2}}\leq \frac\delta\eps C\,\norm{U_0}_{H^k}\]
	with the same convention as above for the multiplicative prefactor $C$.
	
	With these estimates at hand, one concludes as in the preceding section the following results.
	\begin{Proposition}\label{P.WP-BP}
		Let $d=1$ and $k\geq 2$. For any $\eps,\delta>0$, any $\eta_0\in H^k(\RR),\bu_0\in H^{k+1}(\RR)$ and $b\in H^{k+1}(\RR)$ such that $\inf_\RR 1-b\geq h_\star$ and $\inf_\RR 1-b+\eta_0\geq h_\star$ with $h_\star>0$, there exists a unique $U\eqdef(\eta,\bu,q, r)\in \cC(I_{\eps,\delta};H^k(\RR)^4)$ maximal-in-time classical solution to system~\eqref{eq.hBP} with $h=\hb+\eta$ emerging from the initial data 	\begin{equation}\label{eq.hBP-init}
			\eta\vert_{t=0}=\eta_0, \quad \bu\vert_{t=0}=\bu_0, \quad q\vert_{t=0}=0,\quad  r\vert_{t=0}=-\nabla\cdot\bu_0\end{equation}
		and one has $U\in \cap_{j=0}^k\cC^j(I_{\eps,\delta};H^{k-j}(\RR)^4)$. 
		
		Moreover, for any $C_0>0$, there exists $T>0$, $C>0$ and $\lambda\geq1$ depending uniquely on $k,h_\star$ and $C_0$ such that the following holds. 
		Assume that  ${0<\eps\leq \delta\leq 1}$, $\delta\leq1/\lambda$ and $\eps/\delta\leq 1/\lambda$ and
		\[ M_0\eqdef\big\vert \eta_0\big\vert_{H^k}+ \big\vert \bu_0\big\vert_{H^k}+\delta\big\vert \bu_0\big\vert_{H^{k+1}}+\big\vert b\big\vert_{H^{k+1}} \leq C_0. 
		\]
		Then $I_{\eps,\delta}\supset [-T/M_0,T/M_0]$ and for any $t\in [-T/M_0,T/M_0]$, one has for all $i\in[2,k]$,
		\begin{equation}\label{eq.estimate-hBP} \big\vert (\eta,\bu,\eps\delta q,\delta  r )(t,\cdot)\big\vert_{H^{i}}+\big\vert (\partial_t\eta,\partial_t\bu,\eps\delta \partial_tq,\delta \partial_t r )(t,\cdot)\big\vert_{H^{i-1}} 
			\\
			\leq \frac{C\, M_0}{ \delta^{i-2} }.
		\end{equation}
		Moreover, as $\eps \searrow0$ (with $\delta>0$ fixed), denoting $(\eta_\eps,\bu_\eps,q_\eps, r_\eps)$ the solution to~\eqref{eq.hBP} emerging from~\eqref{eq.hBP-init}, we have $(1-b+\eta_\eps, \bu_\eps, r_\eps)\to (h,\bu,-\delta\nabla\cdot \bu)$ pointwisely in $[-T/M_0,T/M_0]\times\RR $ where $(h,\bu)$ satisfies~\eqref{eq.BP} and $(h,\bu)\vert_{t=0}=(1-b+\eta_0,\bu_0)$.
	\end{Proposition}
	
	This result proves that the augmented hyperbolized system~\eqref{eq.hBP} provides approximate solutions to the Boussinesq--Peregrine system~\eqref{eq.BP}. Once again we emphasize that the upper-bound~\eqref{eq.estimate-hBP} is consistent with the development of undesirable spurious oscillations with wavelength $\cO(\delta)$.
	
	\begin{Remark}\label{R.dimension}
		\Cref{P.WP-BP} is limited to dimension $d=1$, because the estimate $\norm{(\partial_t^2 U)\vert_{t=0}}_{H^{k-2}}=\cO(\frac\eps\delta)$ is too weak to apply \Cref{T.Well-posedness} which requires (setting $j_0=1$) $\norm{(\partial_t^2 U)\vert_{t=0}}_{H^{k-2}}=\cO(\frac\eps{\delta^{i_0}})$ with $i_0$ such that $i_0+j_0>1+d/2$, and the case $i_0=1$ falls just outside of the admissible indices when $d=2$. In order to obtain a positive result when $d=2$ one could set $i_0>1$ and either (i) require some smallness assumption on the initial data, that is $\big\vert \eta_0\big\vert_{H^k}+ \big\vert \bu_0\big\vert_{H^k}+\delta\big\vert \bu_0\big\vert_{H^{k+1}}+\big\vert b\big\vert_{H^{k+1}} =\cO(\delta^{i_0-1})$ or (ii) further prepare the initial data by considering $q\vert_{t=0}$ such that $(\partial_t\bu)\vert_{t=0}=\cO(\delta^{i_0-1})$. We refer to~\cite{Duchene19} for the analysis of strongly well-prepared initial data.
	\end{Remark}

\subsubsection{The Serre--Green--Naghdi system}\label{S.LCT}

We consider the Serre--Green--Naghdi system for the propagation of water waves.\footnote{See \cite[Chapter~6]{Lannes} and \cite[Chapter~8]{MM4WW} for its rigorous justification in the fully nonlinear, weakly dispersive regime.} Using dimensionless variables (but setting the nonlinearity and topography amplitude parameters to $1$ for simplicity), it reads
\begin{equation}\label{eq.GN}
	\left\{\begin{array}{l}
		\partial_t h+\nabla\cdot(h\bu)=0,\\[1ex]
		\partial_t\bu +\nabla(h+ b)+(\bu\cdot\nabla)\bu+\frac{\delta^2}{h}\big( \nabla\big( hp \big)+ \pb \nabla b\big)=\bm{0},\\[1ex]
		p= h \big(\frac{\ddot{h}}3 +\frac{\ddot{b}}2\big), \quad \pb=h\big(\frac{\ddot{h}}2 +\ddot{b}\big),
	\end{array}\right.
\end{equation}
where $\delta>0$ is the shallowness parameter measuring the strength of dispersion and we denote $\dot{h}=\partial_t h+\bu\cdot\nabla h$, $\ddot{h}=\partial_t \dot{h}+\bu\cdot\nabla \dot{h}$, and similarly $\dot{b},\ddot{b}$. 
There, $h$ represents the depth of the layer of water, $b$ accounts for the bottom topography (so that $h+b-1$ is the surface elevation), $\bu$ the layer-avereged horizontal velocity, and $p$ and $\pb$ are non-hydrostatic pressure contributions.

Denote $q\eqdef hp$ and $ r \eqdef \dot h/h$. Notice $r=-\nabla\cdot\bu$ (using the mass conservation equation $\partial_t h+\nabla\cdot(h\bu)=0$), $q = \frac{h^3}{3}\big(\partial_t r +\bu\cdot\nabla r - r \nabla\cdot\bu\big)+\frac{h^2}{2}\ddot{b}$ and $\pb=\frac{h^2}{2}\big(\partial_t r +\bu\cdot\nabla r - r \nabla\cdot\bu\big)+h\ddot{b}$.
We infer that~\eqref{eq.GN} may be written equivalently as
\begin{equation}\label{eq.GN-q}
	\left\{\begin{array}{l}
		\partial_t h+\nabla\cdot(h\bu)=0,\\[1ex]
		h\big(\partial_t \bu+(\bu\cdot\nabla)\bu+\nabla(h+b)\big)+\delta^2\nabla q+\frac{\delta^2}{2}h^2\nabla b\big(\partial_t r +\bu\cdot\nabla r - r \nabla\cdot\bu\big)+\delta^2 h\ddot{b}\nabla b=\bm{0}, \\[1ex]
		 r + \nabla\cdot\bu=0,\\[1ex]
	 \frac13h^3\big(\partial_t r +\bu\cdot\nabla r - r \nabla\cdot\bu\big)+\frac{h^2}{2}\ddot{b} =q.
	\end{array}\right.
\end{equation}
The hyperbolization strategy consists in relaxing the third identity in~\eqref{eq.GN-q} to consider
\begin{equation}\label{eq.LCT}
	\left\{\begin{array}{l}
	\partial_t h+\nabla\cdot(h\bu)=0,\\[1ex]
	h\big(\partial_t \bu+(\bu\cdot\nabla)\bu+\nabla(h+b)\big)+\delta^2\nabla q+\frac{\delta^2}{2}h^2\nabla b\big(\partial_t r +\bu\cdot\nabla r - r \nabla\cdot\bu\big)+\delta^2 h\ddot{b}\nabla b=\bm{0}, \\[1ex]
	\eps^2\big(\partial_t q+\bu\cdot\nabla q+q\nabla\cdot\bu\big)=-h(r + \nabla\cdot\bu),\\[1ex]
	\frac13h^3\big(\partial_t r +\bu\cdot\nabla r -r \nabla\cdot\bu\big)+\frac{h^2}{2}\ddot{b} =q,
\end{array}\right.\end{equation}
which we can check is equivalent to the systems proposed in \cite[Remark 9.1]{MM4WW} and \cite[(105)--(108)]{Richard21} and similar to systems proposed by other authors (see references in the aforementioned works).

We now restrict to time-independent topography, $\partial_t b=0$, and infer $\ddot{b}=(\partial_t\bu)\cdot\nabla b+\bu\cdot\nabla(\bu\cdot\nabla b)$.
and write~\eqref{eq.LCT} in compact matricial form. Denoting $U=(h+b-1,\bu,\eps \delta q,\delta r )^\top$, system~\eqref{eq.LCT} reads
\begin{equation}\label{eq.LCT-matrix}	
	\A_0(U)\big(\partial_t U +(\bu\cdot\nabla)U\big)+ (\A(U)\nabla)U+B(U)=\frac1\eps\L_{\delta} U 
\end{equation}
with
\[ \A_0(U)\eqdef
\begin{pmatrix}
	1 &\bm{0}^\top &0 &0 \\
	\bm{0}& h {\sf I}[\nabla b] & \bm{0}&\frac\delta2 h^2\nabla b \\
	0& \bm{0}^\top& h^{-1} & 0\\
	0& \frac\delta2 h^2(\nabla b)^\top& 0& \frac13h^3
\end{pmatrix}, \quad 
\L_{\delta}\eqdef
\begin{pmatrix}
	0 & \bm{0}^\top& 0&0 \\
	\bm{0}& {\sf 0}&-\delta\nabla &0 \\
	0& -\delta\nabla^\top& 0& -1 \\
	0&\bm{0}^\top & 1 & 0
\end{pmatrix}, \quad 
\]
where ${\sf I}[\nabla b]\circ\eqdef\Id+\delta^2 (\nabla b)(\nabla b)\cdot\circ$ and
\[ \A(U)\nabla\eqdef
\begin{pmatrix}
	0 & h\nabla^\top & 0&0 \\
	h\nabla&  -\frac\delta2h^2( r \nabla b) \nabla^\top  & \bm{0}&\bm{0} \\
	0& \frac{q}{h}\nabla^\top& 0& 0\\
	0& -\frac13r  h^3\nabla^\top&0 & 0
\end{pmatrix}, \quad
B(U) \eqdef \begin{pmatrix}
	-\bu\cdot\nabla b\\
\delta^2h (\bu\cdot\nabla b)(\bu\cdot\nabla)(\nabla b)\\
	0 \\
	\frac\delta2 h^2\bu\cdot(\bu \cdot \nabla )(\nabla b)
\end{pmatrix}.
\]
Notice~\eqref{eq.LCT-matrix} is almost under symmetric form, except for the two entries on the last two lines of $\A(U)$.
For that reason, we introduce a near-identity symmetrizer which allows to balance non-symmetric terms with contributions stemming from the singular operator $\L_\delta$. 

We claim the existence of a symmetrizer of the form $\Id+\eps \M(U)$ such that  $\big(\Id+\eps \M(U)\big)\A_0(U)$ is symmetric and $\big(\Id+\eps \M(U)\big)\big(\A(U)\nabla-\frac1\eps\L_\delta\big)$ is skew-adjoint for the $L^2(\RR^d)^n$ inner-product, up to harmless order-zero terms. Specifically, we set
\[ \M(U)\eqdef
\begin{pmatrix}
	0 & \bm{0}^\top & 0 & 0 \\
	\bm{0} & {\sf 0} & (\frac\delta2 h^3\nabla b)f & \bm{0} \\
	0 & \bm{0}^\top & g & f \\
	0 & \bm{0}^\top & \frac{h^4}3f & 0 
\end{pmatrix}
\quad \text{ with } \quad f\eqdef\frac{r  }{\delta h+\eps q}\text{ and } g\eqdef\frac{-q}{\delta h+\eps q}+\frac\eps3\frac{ h^4 r^2}{(\delta h+\eps q)^2} ,
\]
and observe that
\[\big(\Id+\eps \M(U)\big)\A_0(U)=\A_0(U)+\eps \A_1(U), \quad \A_1(U)\eqdef
\begin{pmatrix}
	0 & \bm{0}^\top & 0 & 0 \\
	\bm{0} & {\sf 0} & (\frac\delta2 h^2\nabla b)f & \bm{0} \\
	0 & (\frac\delta2 h^2\nabla b)^\top  f& h^{-1}g & \frac{h^3}3 f \\
	0 & \bm{0}^\top & \frac{h^3}3f & 0 
\end{pmatrix}\]
and
\[ \big(\Id+\eps \M(U)\big)\big(\A(U)\nabla-\frac1\eps \L_\delta\big)=\S(U)\nabla+\G(U)U-\frac1\eps \L_\delta
\]
where
\[ \S(U)\nabla\eqdef
\begin{pmatrix}
	0& h\nabla^\top & 0& 0\\
	h\nabla&  {\sf 0}  &\bm{0} &\bm{0} \\
	0& \bm{0}^\top& 0& 0\\
	0& \bm{0}^\top& 0& 0
\end{pmatrix}, \quad
\G(U) \eqdef \begin{pmatrix}
	0 & \bm{0}^\top & 0 & 0 \\
	\bm{0} & {\sf 0} & 0&(\frac\delta2 h^3\nabla b)f \\
	0 & \bm{0}^\top & -f & g \\
	0 & \bm{0}^\top & 0 & \frac{h^4}3f  
\end{pmatrix}.
\]
Altogether we find that~\eqref{eq.LCT} reads
\begin{equation}\label{eq.LCT-matrix-sym}	\big(\A_0(U)+\eps \A_1(U)\big)\big(\partial_t U +(\bu\cdot\nabla)U\big)+ \S(U)\nabla U+\G(U)U+(\Id+\eps \M(U))B(U)=\frac1\eps\L_{\delta} U 
\end{equation}
where
$U=(h+b-1,\bu,\eps \delta q,\delta r )^\top$. Proceeding as in the preceding sections, we can artificially augment the set of unknowns $U$ to include $b,\nabla b$ so that~\eqref{eq.LCT-matrix-sym} is of the form~\eqref{eq.general}, and satisfies the symmetric structure of Hypothesis~\ref{H1}. Notice however that it is {\em a priori} not obvious that the symmetrizer $\A_0(U)+\eps \A_1(U)$ and source terms $\G(U)U+(\Id+\eps \M(U))B(U)$ are uniformly bounded with respect to $\eps$ and $\delta$ due to the contributions involving $f=\frac{r  }{\delta h+\eps q}$ and $g=\frac{-q}{\delta h+\eps q}+\frac\eps3\frac{ h^4 r^2}{(\delta h+\eps q)^2}$ (recall $(\eps \delta q,\delta r )$ are the considered unknowns). However one readily sees from the last two equations of~\eqref{eq.LCT} that a uniform control of $\norm{U}_{H^{s_0+1}}$ and $\norm{\partial_t U}_{H^{s_0}}$ readily provides a uniform control of $\norm{q}_{H^{s_0}}$ and $\norm{r }_{H^{s_0}}$. Hence assuming that the ratio $\frac\eps\delta$ is sufficiently small, the contributions stemming from $f$ and $g$ are indeed uniformly bounded (and small). We claim that the analysis in this work could be extended to incorporate such contributions (as was done in~\cite{Duchene19}). Hypothesis~\ref{H2} follows from the coercivity of $\A_0$ under the assumption $0<h<\infty$ (since $\A_0$ satisfies the identity~\eqref{eq.hyperbolicity} replacing $\hb$ with $h$) 
and, as discussed above, the fact that contributions of $\A_1(U)$ are perturbative if $\frac\eps\delta$ is sufficiently small. Finally, that $\L_\delta$ satisfies~\ref{H3} and~\ref{H4} has been discussed in preceding sections. Hence the results of \Cref{T.Well-posedness} and \Cref{C.convergence} apply up to modifications stemming from the above discussion, and the analysis on the size of the first two time derivatives of the unknowns is exactly as in the preceding section. We infer from \Cref{T.Well-posedness} and \Cref{C.convergence} the following result, analogous to \Cref{P.WP-BP}, which proves that the augmented hyperbolized system~\eqref{eq.LCT} provides approximate solutions to the Serre--Green--Naghdi system~\eqref{eq.GN}.

	\begin{Proposition}\label{P.WP-GN}
	Let $d=1$ and $k\geq 2$. For any $\eps,\delta>0$, any $\eta_0\in H^k(\RR),\bu_0\in H^{k+1}(\RR)$ and $b\in H^{k+1}(\RR)$ such that $\inf_\RR 1-b+\eta_0\geq h_\star$ with $h_\star>0$, there exists a unique $U\eqdef(\eta,\bu,q,r)^\top\in \cC(I_{\eps,\delta};H^k(\RR)^4)$ maximal-in-time classical solution to system~\eqref{eq.LCT} with $h=1-b+\eta$ emerging from the initial data 
	\begin{equation}\label{eq.LCT-init}\eta\vert_{t=0}=\eta_0, \quad \bu\vert_{t=0}=\bu_0, \quad q\vert_{t=0}=0,\quad r\vert_{t=0}=-\nabla\cdot\bu_0\end{equation}
	and one has $U\in \cap_{j=0}^k\cC^j(I_{\eps,\delta};H^{k-j}(\RR)^4)$. 
	
	Moreover, for any $C_0>0$, there exists $T>0$, $C>0$ and $\lambda\geq1$ depending uniquely on $k,h_\star$ and $C_0$ such that the following holds. 
	Assume that  ${0<\eps\leq \delta\leq 1}$, $\delta\leq1/\lambda$ and $\eps/\delta\leq 1/\lambda$ and
	\[ M_0\eqdef\big\vert \eta_0\big\vert_{H^k}+ \big\vert \bu_0\big\vert_{H^k}+\delta\big\vert \bu_0\big\vert_{H^{k+1}}+\big\vert b\big\vert_{H^{k+1}} \leq C_0.
	\]
	Then $I_{\eps,\delta}\supset [-T/M_0,T/M_0]$ and for any $t\in [-T/M_0,T/M_0]$, one has for all $i\in[2,k]$,
	\begin{equation}\label{eq.estimate-hGN} \big\vert (\eta,\bu,\eps\delta q,\delta r )(t,\cdot)\big\vert_{H^{i}}+\big\vert (\partial_t\eta,\partial_t\bu,\eps\delta \partial_tq,\delta \partial_tr )(t,\cdot)\big\vert_{H^{i-1}} 
		\
		\leq \frac{C\, M_0}{ \delta^{i-2}}.
	\end{equation}
	Moreover, as $\eps \searrow0$ (with $\delta>0$ fixed), denoting $(\eta_\eps,\bu_\eps,q_\eps,r_\eps)$ the solution to~\eqref{eq.LCT} emerging from~\eqref{eq.LCT-init}, we have $(1-b+\eta_\eps, \bu_\eps,r_\eps)\to (h,\bu,-\delta\nabla\cdot\bu)$ pointwisely in $[-T/M_0,T/M_0]\times\RR $ where $(h,\bu)$ satisfies~\eqref{eq.GN} and $(h,\bu)\vert_{t=0}=(1-b+\eta_0,\bu_0)$.
\end{Proposition}

\begin{Remark}
	Let us point out once again that the upper-bound~\eqref{eq.estimate-hGN} is non-uniform with respect to small $\delta$, which is consistent with the development of undesirable spurious oscillations with wavelength $\cO(\delta)$. The restriction to horizontal dimension $d=1$ has the same origin as in the preceding section, and is discussed in \Cref{R.dimension}. The rigorous justification of a closely related augmented hyperbolized system was already provided by the first author in~\cite{Duchene19}. We improve this result by considering topography variations and allowing for weaker restrictions on well-prepared initial data. 
\end{Remark}

\bibliographystyle{abbrv}

\begin{appendices}

\section{Product, commutator and composition estimates}\label{S.technical}

In this section we first recall standard product, commutator and composition estimates in Sobolev spaces; see for instance \cite[Appendix B]{Lannes} for relevant references. Then we turn to a key product estimate which we crucially use to control energy functionals involving space and time regularity with different weights. 

\subsection{Standard estimates}

	\begin{Lemma}[Continuous embedding]\label{L.embedding}
		Let $d\in \NN^\star$ and $s\in\RR$ such that $s>d/2$. Then for any $f\in H^{s}(\RR^d)$, one has $f\in L^\infty(\RR^d)\cap \cC(\RR^d)$ and
		\[\norm{f}_{L^\infty}\lesssim  \norm{ f}_{H^{s}}. \]
	\end{Lemma}
\begin{Lemma}[Products]\label{L.product}
	Let $d,\ell\in\NN^\star$ and $s,s_0,\dots,s_\ell\in\RR$ be such that $\sum_{n=0}^\ell s_n\geq 0$, $\sum_{n=0}^\ell s_n> s+\ell d/2$ and $s_n\geq s$ for $n\in\{0,\dots,\ell\}$. Then for any $f_n\in H^{s_n}(\RR^d)$, one has $\prod_{n=0}^{\ell} f_n\in H^s(\RR^d)$ and
	\[\norm{\prod_{n=0}^{\ell} f_n}_{H^s}\lesssim  \prod_{n=0}^{\ell} \norm{ f_n}_{H^{s_n}}. \]
	In particular, $H^{s_0}(\RR^d)$ is a Banach algebra as soon as $s_0>d/2$.
\end{Lemma}

\begin{Lemma}[Product]\label{L.product-2}
	Let $d\in\NN^\star$, $s_0>d/2$ and $s\geq 0$. Then for any $f\in H^s(\RR^d)\cap H^{s_0}(\RR^d)$ and $g\in H^s(\RR^d)$, one has $fg\in H^s(\RR^d)$ and
	\[\norm{fg}_{H^s}\lesssim   \norm{ f}_{H^{s_0}}\norm{ g}_{H^{s}}+\Big\langle \norm{ f}_{H^{s}}\norm{ g}_{H^{s_0}}\Big\rangle_{s>s_0}. \]
\end{Lemma}
\begin{Lemma}[Commutator]\label{L.commutator}
	Let $d\in\NN^\star$, $s_0>d/2$ and $s\geq -1$. Denote $\Lambda^s=(\Id-\Delta_{\bx})^{s/2}$. Then for any $f\in H^s(\RR^d)\cap H^{s_0+1}(\RR^d)$ and $g\in H^{s-1}(\RR^d)$, one has $[\Lambda^s,f]g\in L^2(\RR^d)$ and
	\[\norm{[\Lambda^s,f]g}_{L^2}\lesssim   \norm{ f}_{H^{s_0+1}}\norm{ g}_{H^{s-1}}+\Big\langle \norm{ f}_{H^{s}}\norm{ g}_{H^{s_0}}\Big\rangle_{s>s_0+1}. \]
\end{Lemma}
\begin{Lemma}[Composition]\label{L.composition}
	Let $d,n\in\NN^\star$, $s_0>d/2$ and $s\geq 0$. Let $F\in \cC^\infty(\RR^n; \RR)$ be such that $F(\bm0)=0$, and $M\geq0$. Then there exists $C_M>0$ such that for any $g\in H^{s}(\RR^d)\cap H^{s_0}(\RR^d)$ such that $\norm{g}_{H^{s_0}}\leq M$, one has $F(g)\in H^s(\RR^d)$ and
	\[\norm{F(g)}_{H^s}\lesssim C_M  \norm{ g}_{H^{s}}. \]
\end{Lemma}

\subsection{Revisited anisotropic space-time estimate}

We now provide a key product estimate for functions depending on space and time. This estimate, together with our notion of admissible weights (\Cref{D.alpha-admissible}), is a key ingredient for the control of energy functionals involving non-homogeneous weights.

\begin{Lemma}[Space-time products]\label{L.Product}
	Let $d,\ell\in\NN^\star$, $s_0>d/2$, $m_k\geq m_j\geq 0$. For $n\in \{0,\dots,\ell\}$, let $(j_n,\alpha_n)\in \NN^{1+d}$ be such that $j_n\geq m_j$ and $j_n+|\alpha_n|\geq m_k$. Let $i\geq 0$ and denote $I=i+\sum_{n=0}^\ell |\alpha_n|$ and $J=\sum_{n=0}^\ell j_n$. There exists $C>0$ and $i_n\geq |\alpha_n|$ ($n=0,\dots,\ell$) such that for any sufficiently regular functions $U_n$ ($n=0,\dots,\ell$) one has
	\[\norm{\prod_{n=0}^\ell (\partial_t^{j_n}\partial^{\alpha_n} U_n)}_{H^i}  \leq  C\sup_{(i_0,\dots,i_\ell)} \prod_{n=0}^\ell \norm{\partial_t^{j_n} U_n}_{H^{i_n}}\]  
	where the supremum is taken over $(i_0,\dots,i_\ell)\in \RR_+^{1+\ell}$ such that 
	\[\sum_{n=0}^\ell i_n = I+\ell s_0\]
		and for all $n\in\{0,\dots,\ell\}$, $(j_n,i_n+j_n)\in [(m_j,s_0+m_k),(J-\ell m_j,I+J-\ell m_k)]$ where we write
		$	 (j_n,i_n+j_n)\in[(j_\star,k_\star),(j^\star,k^\star)]$ if and only if
		\[	\begin{cases}
			(j_n,i_n+j_n)\in[j_\star,j^\star]\times[k_\star,k^\star]  & \text{ if } k_\star \leq k^\star,\\
			(j_n,i_n+j_n)=\theta_n (j_\star,k_\star)+(1-\theta_n)(j^\star,k^\star) ,\ \theta_n=\frac{(i_n+j_n)-k^\star}{k_\star-k^\star} & \text{ if } k_\star > k^\star.
		\end{cases}\]

	We represent in \Cref{F.intervals} the admissible values for $(j_n,i_n+j_n)$.
\end{Lemma}
\begin{figure}
	\begin{center}
		\subcaptionbox{Case $i_\star+j_\star<i^\star+j^\star$\label{F.interval}}{\includegraphics[width=.33\textwidth]{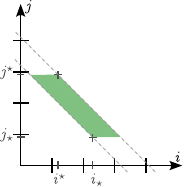}}\hspace{.11\textwidth} 	
		\subcaptionbox{Case $i_\star+j_\star>i^\star+j^\star$\label{F.interval2}}{\includegraphics[width=.33\textwidth]{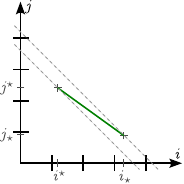}}	
	\end{center}
	\caption{Representation of the admissible set for $(j_n,i_n+j_n)\in[(j_\star,i_\star+j_\star),(j^\star,i^\star+j^\star)] $ in a $(i,j)$-coordinate system as in \Cref{F.diagram,F.diagram-2}. In \Cref{L.Product}, $j_\star=m_j$, $j^\star\eqdef J-\ell m_j$, $i_\star+j_\star= s_0+m_k$, and $i^\star+j^\star=I+J-\ell m_k$.}
	\label{F.intervals}
\end{figure}
\begin{proof}
	Assume first that $i\in[0,s_0]$. We use product estimates of \Cref{L.product} with well-chosen $s_n\in[0,s_0]$ such that ${|\alpha_n|+s_n=i_n}$, distinguishing between two cases.
	\begin{itemize}
		\item If $s_0+m_k \leq I+J-\ell m_k$, we set $s_n = \theta_n i+(1-\theta_n)s_0$ with $\theta_n=\frac{j_n+|\alpha_n|-m_k}{I+J-(\ell +1)m_k-i}$, so that $j_n+|\alpha_n| = \theta_n (I+J-\ell m_k-i)+(1-\theta_n)m_k$ and hence $j_n+i_n = \theta_n (I+J-\ell m_k)+(1-\theta_n)(s_0+m_k)$. In the situation where $I+J-(\ell +1)m_k=i$, one has $j_n+|\alpha_n|=m_k$ for all $n\in\{0,\dots,\ell\}$, and one can set $\theta_n=\frac{1}{\ell+1}$.
		\item If $s_0+m_k > I+J-\ell m_k$, we set $s_n$ through $j_n+|\alpha_n|+s_n=\theta_n (I+J-\ell m_k)+(1-\theta_n)(s_0+m_k)$, with  $\theta_n=\frac{j_n-m_j}{J-(\ell+1)m_j}$ (so that $j_n=\theta_n(J-\ell m_j) +(1-\theta_n)m_j$).
	\end{itemize} 
	In both cases, one easily check that $\theta_n\in[0,1]$, $\sum_{n=0}^\ell\theta_n=1$, and $(i,s_0,\dots,s_\ell)$ satisfy the requirements of \Cref{L.product}. If $i>s_0$, then one can simply use \Cref{L.product-2}. 
\end{proof}
\begin{Remark}\label{R.tame}
	The above result can be interpreted as a product estimate in anisotropic $L^\infty$-based (in time)-$L^2$-based (in space) Sobolev spaces. That we can restrict Sobolev indices $(i_0,\dots,i_\ell)\in \RR_+^{1+\ell}$ as in the statement is crucial in our framework where we measure the space-regularity of functions with non-homogeneously weighted functionals, of the form \eqref{eq.def-F}. 
		
		To be more specific, if one denotes
		\[
		\cF_{k,{\bm\alpha}}(U)\eqdef\sup\big(\big\{ \alpha_{j,i}^{-1} \norm{\partial_t^j U}_{H^{i}} \ : \ (i,j)\in\RR_+\times\NN,\ i+j\leq k\big\}\big)
		\]
		with weights $\alpha_{j,i}>0$ satisfying assumption \ref{vi} in \Cref{D.alpha-admissible},
		then for any $(i_0,\dots,i_\ell)\in \RR_+^{1+\ell}$ as in the statement of \Cref{L.Product} one has
		\[ \prod_{n=0}^\ell \norm{ \partial_t^{j_n} U_n}_{H^{i_n}}  \leq  (\alpha_{m_j,s_0+m_k-m_j})^\ell \alpha_{J-\ell m_j,I-\ell(m_k-m_j)} \prod_{n=0}^\ell \cF_{k,{\bm\alpha}}(U_n), \]
		a statement which does not hold for general non-homogeneous weights.

	Notice in the above inequality how the parameters $m_k\geq m_j\geq  0$ in \Cref{L.Product} allow to make use of some repartition of time and space-time derivatives between the different components to lower the indices of the higher-regularity norm $(J-\ell m_j,I-\ell(m_k-m_j))$, as in commutator estimates (see \Cref{L.commutator}). In this work we use either $m_j=m_k=0$, $m_j=m_k=1$, or $m_j=0$ and $m_k=1$.
\end{Remark}
	
\section{Notations and functional spaces}\label{S.notations}

	\begin{itemize}
		\item The notation $a\lesssim b$ means that 
		$a\leq C_0\ b$, where $C_0$ is a non-negative constant whose exact expression is of no importance, and $a=\cO(b)$ for $|a|\lesssim b$. We denote $a \searrow 0$ for $a\to 0,\ a>0$. 
		We denote $C(\cdot)$ a non-negative quantity depending non-decreasingly on its (real) variable.
		\item As much as possible we use bold typefaces ($\bu$) for $d$-dimensional vectors, capitalized letters ($U$) for $n$-dimensional vectors, and sans serif fonts (${\sf M}$) for matrices. Script typefaces ($\L$) is dedicated to pseudodifferential operators.
		\item $\Id$ is the identity operator.  
		\item We use the notation
		\[\big\langle A\big\rangle_{a>b}=\begin{cases}
			A  & \text{ if } a>b,\\
			0 & \text{ otherwise. }
		\end{cases}
		\]
		\item For $s\in\RR$, $\lceil s\rceil$ denotes the smallest integer larger or equal to $s$
		\item For $\bk\in\RR^d$, $\langle \bk\rangle\eqdef (1+|\bk|^2)^{1/2}$.
		\item We use the multi-index notation for multi-dimensional differentiation: for ${\alpha=(\alpha_1,\dots,\alpha_d)\in \NN^d}$, $\partial^{\alpha} f(x_1,\dots,x_d)\eqdef(\partial_{x_1}^{\alpha_1}\cdots\partial_{x_d}^{\alpha_d}f)(x_1,\dots,x_d)$ and $|\alpha|=\alpha_1+\dots+\alpha_d$. 
		\item For $d\in\NN^\star$, we denote $L^2(\RR^d)$ the Lebesgue space associated with square-integrable functions endowed with the real inner product and corresponding norm
		\[
		\big(f_1,f_2\big)_{L^2}\eqdef\int_{\RR^d}f_1(\bx)f_2(\bx) \dd \bx, \quad 	\vert f \vert_{L^2}\eqdef\left(\int_{\RR^d}\vert f(\bx)\vert^2 \dd \bx\right)^{\frac12}.\] 
		The space $L^\infty(\RR^d)$ is associated with essentially bounded, Lebesgue-measurable functions and endowed with the norm
		\[
		\norm{f}_{L^\infty}\eqdef \esssup_{\bx\in\RR^d} | f(\bx)|.
		\]
		We define similarly $L^\infty_{\rm loc}(\RR^d)$ the locally essentially bounded, Lebesgue-measurable functions.
		\item
		For any real constant $s\in\RR$, $H^s(\RR^d)$ denotes the Sobolev space obtained by completing the Schwartz space of smooth rapidly decreasing functions $\cS(\RR^d)$ for the norm 
		\[\vert f\vert_{H^s}\eqdef\vert \Lambda^s f\vert_{L^2} \] where $\Lambda^s$ is the Fourier multiplier (see below) $\Lambda^s\eqdef(\Id-\Delta_{\bx})^{s/2}=(1+|D|^2)^{s/2}$. 
		\item Given an interval $I\subset\RR$ and any of the previously defined functional spaces, $X$, we denote $L^\infty(0,T;X)$ the space of functions such that $u(t,\cdot)$, taking values in the Banach space $X$, is essentially bounded for $t\in I$, and denote the associated norm
		\[\Norm{u}_{L^\infty(0,T;X)} \ \eqdef \ {\rm ess\,sup}_{t\in(0,T)}\norm{u(t,\cdot)}_{X} \ < \ \infty.\]
		For $k\in\NN$, and $I$ a real interval, $\cC^k(I;X)$ denotes the space of $X$-valued continuous functions on $I$ with continuous derivatives up to order $k$. 
				\item Given two Banach spaces $X$ and $Y$ we write $X\hookrightarrow Y$ if $X\subset Y$ and the natural embedding of $X$ into $Y$ is continuous.
	\end{itemize}
	
	\begin{Definition}[Fourier multipliers]\label{D.Fourier-multipliers}
		Let $L\in L^\infty_{\rm loc}(\RR^d)$ be such that there exists $C>0$ and $m\in\RR$ such that for almost every $\bk \in\RR^d$,
		\[ |L(\bk)|\leq C \langle \bk \rangle^m .\]
		For any $s\in\RR$, we denote $\L=L(D):H^s(\RR^d) \to H^{s-m}(\RR^d)$ the operator defined by $\widehat{\L g}=L \,\widehat g$, {\em i.e.}
		\[ \forall g\in \cS(\RR^d),\quad \forall\bx\in\RR^d, \qquad (\L g)(\bx) = \frac1{(2\pi)^d}\int_{\RR^d} \int_{\RR^d}e^{\i(\bx-\by)\cdot\bk}\ L(\bk) \ g(\by)\dd\by\dd\bk\ .\]
		The operator is continuously extended to $g\in H^s(\RR^d)$ by the density of the Schwartz space of smooth rapidly decreasing functions $ \cS(\RR^d)$ in $H^s(\RR^d)$.
		
		We call such operators {\em Fourier multipliers} (or {\em constant coefficients pseudo-differential operators}) of order $ m$, and the associated function $L$ their {\em symbol}.
	\end{Definition}
	\begin{Remark}
		Fourier multipliers associated with polynomial symbols are standard differential operators: $P(D)=P(\frac1\i\nabla_{\bx})$. 
		
		Scalar Fourier multipliers commute: $F(D)G(D)=G(D)F(D)=(FG)(D)$.
		
		Matrix Fourier multipliers associated with matrix symbols act on vector-values functions through scalar Fourier multiplier entries. The adjoint of such operators in $L^2(\RR^d)^n$ is the Fourier multiplier associated with the Hermitian transpose of their symbol.
	\end{Remark}
	
	\begin{Remark}[Reformulation of Hypotheses~\ref{H3} and~\ref{H4}]\label{R.H3H4}
		Using the above notations, Hypothesis~\ref{H3} and~\ref{H4} read equivalently: 
		the Fourier multiplier $\L_\delta$ is skew-symmetric for $L^2(\RR^d)^n$ inner-product and
		there exists $N_\L>0$ such that for all $s\in\RR$, $(\eps,\delta)\in\cS$ and $U\in H^{s+1}(\RR^d)^n$ one has 
		\[ \forall \bk \in \RR^d, \quad \norm{ \sL_\delta U}_{H^s} \leq N_\L\, \big( \norm{U}_{H^{s}}+\delta\norm{ U}_{H^{s+1}}\big);\] 
		there exists $C_\L\geq0,c_\L>0$ and a Fourier multiplier $\P_\delta$ such that $\P_\delta^2=\P_\delta$, $\P_\delta(\Id-\P_\delta)=0$, $\P_\delta\L_\delta=\L_\delta\P_\delta$ and for all $s\in\RR$, $(\eps,\delta)\in\cS$ and $U\in H^{s+1}(\RR^d)^n$ one has 
		\begin{align*}  
			\norm{ \P_\delta \L_\delta \P_\delta U}_{H^{s}} &\leq \eps\, C_\L\,  \norm{ U}_{H^{s+1}},\\
			\norm{ (\Id-\P_\delta) \L_\delta (\Id-\P_\delta) U}_{H^s} &\geq   \, c_\L/2 \big( \norm{(\Id-\P_\delta) U}_{H^{s}}+\delta\norm{(\Id-\P_\delta) U}_{H^{s+1}}\big).
		\end{align*}
	\end{Remark}
	
	Finally we provide for the reader's convenience some tables of notations used in this document in \Cref{T.1,T.2,T.3,T.4}.


\begin{table}[!b]
\centering
\renewcommand{\arraystretch}{1.6}
\begin{tabular}{p{3cm} p{2.4cm} p{9.6cm}}
\hline
\textbf{Notation} & \textbf{Track} & \textbf{Meaning} \\
\hline

$d$ &  \Cref{S.presentation}
& Space dimension. \\

$n$ &  \Cref{S.presentation}
& Dimension of unknown vector $U \in \mathbb{R}^n$, (i.e number of variables in the system) \\

$\varepsilon$ &  \Cref{S.presentation}
& Fast time scale parameter (e.g.\ Rossby number), it drives the stiffness operator $\tfrac{1}{\varepsilon}\L_\delta$. \\

$\delta$ &  \Cref{S.presentation}
& Intermediate spatial scale parameter (e.g.\ Rossby / Mach number), its smallness triggers spatial oscillations. \\

$\cS$ &   \Cref{T.Well-posedness}
& $(\eps,\delta)$ parameter regime in \Cref{T.Well-posedness}.  \\

$\cS'$ &   \Cref{C.convergence}
& $(\eps,\delta)$ parameter regime in \Cref{C.convergence}.  \\

$k$ & \Cref{T.Well-posedness} 
& Maximal regularity index with condition $k > d/2 + 1$. \\

$j_0$ & \Cref{T.Well-posedness}, \Cref{D.alpha-admissible}
& Threshold time derivative index beyond which weights grow. It marks the boundary beyond which high spatial regularity is no longer uniformly bounded. \\

$j_\sharp$ & \Cref{T.Well-posedness}, \cref{eq.well-prepared} 
& Order of time derivatives controlled in the well-prepared initial data assumption~\eqref{eq.well-prepared}. \\

$i_0$ & \Cref{T.Well-posedness}
& Threshold spatial-regularity index. It ensures ${i_0+j_0 > d/2 + 1}$. \\

$s_0$ & \Cref{D.alpha-admissible}, \Cref{L.weights}, \Cref{C.weights}
& Sobolev embedding threshold, satisfying $s_0 > d/2$ so that $H^{s_0}(\RR^d)\hookrightarrow L^\infty(\RR^d)$ ($s_0 = i_0+j_0-1$). \\

$\lambda$ & \Cref{T.Well-posedness}, \Cref{D.alpha-admissible}, \Cref{C.weights}, \Cref{P.Stability} 
& Large bootstrap parameter used to close the finite induction in the stability estimate (\Cref{P.Stability}). \\

$C_0$ & \Cref{T.Well-posedness}, \Cref{P.Well-posedness} 
& Upper bound on the size of the initial data.\\

$T$ & \Cref{T.Well-posedness}, \Cref{P.Well-posedness} 
& Existence time, uniform with respect to $(\varepsilon,\delta)$. \\

\hline
\end{tabular}
\caption{Basic parameters and scale indices.}
\label{T.1}
\end{table}

\begin{table}[ht]
\centering
\renewcommand{\arraystretch}{1.6}
\begin{tabular}{p{3cm} p{3cm} p{9cm}}
\hline
\textbf{Notation} & \textbf{Track} & \textbf{Meaning} \\
\hline

$\Omega$ & Hypothesis~\ref{H2}
& Domain of hyperbolicity containing an open neighborhood of $\{\bm{0}\}\in\mathbb{R}^n$. Solutions are required to remain 
in a compact $K\subset\Omega$ uniformly in $(\varepsilon,\delta)$. \\

$c_K$ & Hypothesis~\ref{H2}
& Coercivity constant of $\S_0(U)$: $\langle \S_0(U)V,V\rangle \geq c_K \norm{V}^2$. \\

$C_{k,K}$ & \Cref{R.uniform}
& Uniform upper bound on all derivatives up to order $k$ of the coefficient matrices 
$\S_0, \S_l$ and $\G$. \\

$C'_{k,K}$ & Hypothesis~\ref{H5}
& Uniform upper bound on the derivatives up to order $k$ of $\S_0$ under additional Hypothesis~\ref{H5}.  \\

$C''_{k,K}$ & Hypothesis~\ref{H6}
& Uniform upper bound on the derivatives up to order $k$ of $\S_0$ under additional Hypothesis~\ref{H6}. \\

$N_\L$ & Hypothesis~\ref{H2}
& Uniform upper bound on the operator $\L_\delta$. \\

$\sPi_\delta(k)$ & Hypothesis~\ref{H4}
& Orthogonal projection (in Fourier space) onto the regular (asymptotically kernel) subspace of $\sL_\delta(\bk)$, used to split the unknown into regular and singular components. \\

$C_\L$ & Hypothesis~\ref{H4}, forward estimate
& Uniform upper bound on $\Vert \sPi_\delta \sL_\delta \sPi_\delta\Vert $. \\

$c_\L$ & Hypothesis~\ref{H4}, backward estimate
& Coercivity bound of $(\Id-\sPi_\delta)\sL_\delta(\Id-\sPi_\delta)$ on the singular subspace. \\

$\sPi$ & \Cref{C.convergence}, \Cref{R.convergence}
& Limit orthogonal projection:
$\sPi = \lim_{(\varepsilon,\delta)\to(0,\bar\delta)}\sPi_\delta$ 
in $\cS'$.\\

$\sT$ & \Cref{C.convergence}, \Cref{R.convergence}
& Limit operator defined as 
$\sT = \lim_{(\varepsilon,\delta)\to(0,\bar\delta)} 
\tfrac{1}{\varepsilon}\sPi_\delta \sL_\delta$ in $\cS'$.\\

$\P_\delta$ & \Cref{R.H3H4} 
& Pseudo-differential projection operator with symbol $\sPi_\delta$ and used in the forward and backward estimates in \Cref{S.general-estimates}. \\

$\P$, $\T$ & \Cref{C.convergence}, \cref{eq.limit} 
& Pseudo-differential operators with symbol $\sPi$, $\sT$, arising in the limit system \eqref{eq.limit}.  \\

$\lambda_\delta$ & \Cref{R.H4} 
& Eigenvalue of the symbol $\sL_\delta(k)$ lying on the imaginary axis 
by skew-symmetry of $\sL_\delta$. \\

$\lambda^{(m)}$ & \Cref{R.H4}
& Coefficient of order $m$ in the holomorphic expansion of $\lambda_\delta$: $\lambda_\delta=\lambda^{(0)}+\delta\lambda^{(1)}+\delta^2\lambda^{(2)}+\ldots$ 
with $\lambda^{(0)}\in\Spec(L_0)$. \\

\hline
\end{tabular}
\caption{Operators and their bounds.}
\label{T.2}
\end{table}

\begin{table}[ht]
\centering
\renewcommand{\arraystretch}{1.6}
\begin{tabular}{p{3cm} p{3cm} p{9cm}}
\hline
\textbf{Notation} & \textbf{Track} & \textbf{Meaning} \\
\hline

$i,\ j,\ k$ & 
& Respectively spatial, time and space-and-time regularity indices. \\

$m_j,\, m_k$ & \Cref{L.Product}, \Cref{R.tame}, \Cref{S.Stability}
& Parameters in the space-time product estimate \Cref{L.Product}: 
$m_k \geq m_j \geq 0$ controlling the minimal order of time derivative (for $m_j$)
and of space-and-time regularity (for $m_k$) involved in each factor. \\

$i$ & \Cref{L.Product}
& Sobolev regularity index of the product 
$\prod_{n=0}^\ell(\partial_t^{j_n}\partial^{\alpha_n}U_n)$ 
(estimated in $H^i$). 
\\

$j_n, \alpha_n$ & \Cref{L.Product}
& Respectively time and spatial multi-index explicitly present in the $n$-th factor of the 
product, $\partial_t^{j_n}\partial^{\alpha_n}U_n$.
\\

$I,\ J$ & \Cref{L.Product}
& Respectively total spatial regularity $I = i + \sum_{n=0}^\ell|\alpha_n|$ and total number of time derivatives $J = \sum_{n=0}^\ell j_n$. \\

$(i_n)_n$ & \Cref{L.Product} 
& Spatial index family, it
distributes the total budget $I + \ell s_0$ among the $\ell+1$ 
factors, with the $\ell s_0$ overhead from the Sobolev algebra 
embedding in \Cref{L.product}. \\

$(j_\star,\, k_\star)$, $(j^\star,\, k^\star)$ & \Cref{L.Product},
\Cref{D.alpha-admissible}\ref{vi},  \cref{eq.def-Nell,eq.def-Nell-recall} 
& Free variables that represent lower and upper bounds of admissible sets denoted $[(j_\star, k_\star),(j^\star, k^\star)]$ for index pairs representing the number of time derivatives (for the first element) and space-time regularity (for the second element). In \Cref{L.Product}, $j_\star = m_j$ and $k_\star = s_0 + m_k$, $j^\star = J - \ell m_j$ and  $k^\star = I + J - \ell m_k$. \\

$(j_n,\, i_n+j_n)$ &\Cref{L.Product},
\Cref{D.alpha-admissible}\ref{vi},  \cref{eq.def-Nell,eq.def-Nell-recall}
& Number of time derivatives and space-and-time regularity, $n$-th element of a tuple lying
in the admissible set 
$[(j_\star, k_\star),(j^\star, k^\star)]$.   In \Cref{L.Product}, it relates to components in the right-hand side of the product estimate.   \\

$\theta_n$ & \Cref{L.Product},
\Cref{D.alpha-admissible}\ref{vi},  \cref{eq.def-Nell,eq.def-Nell-recall}
& Interpolation coefficient of the $n$-th factor defined by
$\theta_n = \tfrac{(i_n+j_n)-k^\star}{k_\star-k^\star}\in[0,1]$ 
when $k_\star > k^\star$. One has $\sum_n\theta_n=1$. 
\\

$N_\ell(V,\dots,V;U)\big\vert_{j_\star,k_\star}^{j^\star,k^\star}$ & \cref{eq.def-Nell,eq.def-Nell-recall} 
& A bounded supremum of $\ell$ $V$-norms and one $U$-norm 
over admissible index tuples $((i_n,j_n))_n$ that absorbs all nonlinear 
contributions from commutator and composition estimates in \Cref{S.Stability} via \Cref{D.alpha-admissible}\ref{vi} and closes the bootstrap of the uniform stability estimates in \Cref{P.Stability} . \\

\hline
\end{tabular}
\caption{Regularity indices	and multilinear notation.}
\label{T.3}
\end{table}

\begin{table}[ht]
\centering
\renewcommand{\arraystretch}{1.6}
\begin{tabular}{p{3cm} p{3cm} p{9cm}}
\hline
\textbf{Notation} & \textbf{Track} & \textbf{Meaning} \\
\hline

$\beta_p,\,\gamma_p$ & \Cref{L.weights}, \cref{eq.def-alpha-general}
& Parameters defining explicit admissible weight families.
\\

$\alpha_{j,i}$ & \Cref{D.alpha-admissible}, \Cref{C.weights}, \cref{eq.def-F,eq.def-alpha}
& Positive (admissible) weights assigned to $\|\partial_t^j U\|_{H^i}$. The  
explicit form that is used in the proof of \Cref{T.Well-posedness} is
$\alpha_{j,i}=\max\bigl(1,\,\lambda^{j-j_0},\,
\lambda^{j-j_0}\delta^{i_0+j_0-i-j},\,
\varepsilon^{j_0-j}\delta^{i_0-i}\bigr)$.
This 
encodes in particular the emerging spatial scale of order $\delta$. \\

$\cF_{k,\bm\alpha}(U)$ & \Cref{R.weights}, \cref{eq.def-F}, \Cref{P.Stability}, \cref{eq.def-F-stability}
& Weighted energy functional encoding the size of $\|\partial_t^j U\|_{H^i}$:
$\cF_{k,{\bm\alpha}}(U)\eqdef\sup\big(\big\{ \alpha_{j,i}^{-1} \norm{\partial_t^j U}_{H^{i}}  :  (i,j)\in\RR_+\times\NN,\ i+j\leq k\big\}\big)$
\\

$\cF^{(0)}_{k,\bm\alpha}(U)$ & \cref{eq.def-F0-stability}
& Reduced weighted functional 
controlled by the energy 
method and Grönwall's inequality on \eqref{eq.control-F0}. \\

$\omega^{(0)}_{k,\alpha}$ & \cref{eq.def-F0-stability} 
& Index set of low time derivative and pure 
time derivative indices $\{j<j_0,\, i\in[0,k-j]\}\cup\{i=0,\,j\in[0,k]\}$ 
defining~$\cF^{(0)}_{k,\bm\alpha}$.\\

$\cF^{(0),\mathrm{new}}_{k,\bm\alpha}(U)$ & \Cref{P.Stability-H6}, \Cref{S.H6}
& Additional reduced functional needed when $j_0=0$ to 
controls spatial derivatives of $U$. \\

$M_{j_\sharp,0}$ & \cref{eq.well-prepared}
& Measure of the size of the initial data in \Cref{T.Well-posedness}:
$M_{j_\sharp,0}\eqdef\big(\sum_{j=0}^{j_\sharp-1}\|\partial_t^j U\|_{H^{k-j}} 
+ \tfrac{\varepsilon^{j_\sharp-j_0}}{\delta^{i_0}}
\|\partial_t^{j_\sharp}U\|_{H^{k-j_\sharp}}\big)\big\vert_{t=0}$. It controls $\cF_{k,\bm\alpha}(U)\big\vert_{t=0}$.\\

$M_V$, $M_U$, $M_R$,$M^{(0)}_U$ & \Cref{S.bootstrap} 
& Shorthands for $\cF_{k,\bm\alpha}(V)$, $\cF_{k,\bm\alpha}(U)$, $\cF_{k,\bm\alpha}(R)$ and $\cF^{(0)}_{k,\bm\alpha}(U)$. \\

$C^{(0)}$ & \Cref{P.Stability}, \Cref{S.bootstrap}
& Stability constant from  the inner product estimates depending only on $k,c_K,C_\L,c_\L$ and the reduced energy functional $\cF^{(0)}_{k,\bm\alpha}(U)$. It is the multiplicative constant in the control of $\cF_{k,\bm\alpha}(U)$ from $\cF^{(0)}_{k,\bm\alpha(U)}$. \\

$C$ & \Cref{P.Stability}, \Cref{S.bootstrap}
& Full stability constant additionally depending on the full energy functional $ \cF_{k,\bm\alpha}(U)$. It is the constant governing the Grönwall growth rate in the energy inequality for $\cF^{(0)}_{k,\bm\alpha}(U)$. \\

\hline
\end{tabular}
\caption{Energy functionals and norms.}
\label{T.4}
\end{table}

\end{appendices}

	\paragraph{Acknowledgments}
The research of VD was partially funded by ANR project HEAD, ANR-24-CE40-3260. VD thanks the France 2030 program Centre Henri Lebesgue ANR-11-LABX-0020-0, for fostering an attractive mathematical environment. 
AD and KMS acknowledge financial support from the French National Research Agency project NABUCO, ANR-17-CE40-0025 and from the SHOM research contract No21CP05. AD acknowledges financial support from the French National Research Agency project MOTIONS, grant ANR-23-CE56-0006-02.

For the purpose of Open Access, a CC-BY public copyright licence has been applied by the authors to the present document and will be applied to all subsequent versions up to the Author Accepted Manuscript arising from this submission

\end{document}